\documentclass[11pt]{article}
\usepackage{amsmath,amsthm,amsfonts,amssymb}
\usepackage[all]{xy}
\usepackage[usenames]{color} 

\usepackage{xr}

\newcommand\brefGM[1]{\ref{GM#1}}

\newcommand\refGM[1]{I.\brefGM{#1}}

\newcommand\prefGM[1]{(\brefGM{#1})}

\newcommand\trefGM[2]{\refGM{#1}~\prefGM{#2}}

\newcommand\recognition{\cite{FeighnHandel:recognition}}

\newcommand\BookZero{\cite{BFH:laminations}}
\newcommand\BookOne{\cite{BFH:TitsOne}}
\newcommand\BookTwo{\cite{BFH:TitsTwo}}

\newcommand\PartOne{Part I \cite{HandelMosher:SubgroupsI}}

\newcounter{enumitemp}
\newenvironment{enumeratecontinue}{
 \setcounter{enumitemp}{\value{enumi}}
 \begin{enumerate}
 \setcounter{enumi}{\value{enumitemp}}
}
{
 \end{enumerate}
}

\newcommand\pref[1]{(\ref{#1})}

\newtheorem{thm}{Theorem}[section]
\newtheorem{theorem}[thm]{Theorem}
\newtheorem{lemma}[thm]{Lemma}

\newtheorem{cor}[thm]{Corollary}
\newtheorem{corollary}[thm]{Corollary}
\newtheorem{proposition}[thm]{Proposition}
\newtheorem*{proposition*}{Proposition}
\newtheorem*{theorem*}{Theorem}

\newtheorem{prop}[thm]{Proposition}

\newtheorem{sublem}[thm]{Sublemma}
\newtheorem{fact}[thm]{Fact}

\theoremstyle{definition}

\newtheorem{IAStep}{Step}

\newtheorem{definition}[thm]{Definition} 

\newtheorem{notn}[thm]{Notation}

\newtheorem{remark}[thm]{Remark}

\theoremstyle{remark}

\newcommand\from\colon
\newcommand\inv{{-1}}
\newcommand\subgroup{<}
\newcommand\infinity\infty

\newcommand\supp{\text{supp}}
\newcommand\disjunion\amalg

\DeclareMathOperator{\Fix}{Fix}

\DeclareMathOperator{\IA}{IA}

\DeclareMathOperator{\bcc}{BCC}

\DeclareMathOperator{\GL}{GL}

\DeclareMathOperator\Isom{Isom}

\newcommand{\Z}{{\mathbb Z}}
\newcommand{\C}{{\mathcal C}}
\newcommand{\T}{{\mathbb T}}

\newcommand{\f}{F_n}

\newcommand{\E}{{\mathcal E}}
\newcommand{\V}{\mathcal V}

\newcommand{\Out}{\mathsf{Out}}
\newcommand{\Aut}{\mathsf{Aut}}

\newcommand{\Stab}{\mathsf{Stab}}
\newcommand{\ffs}{free factor system}

\newcommand{\upg}{UPG}
\newcommand{\F}{\mathcal F}

\DeclareMathOperator\MCG{\mathcal{MCG}}

\newcommand{\A}{\mathcal A}
\newcommand{\h}{\mathcal H}
\renewcommand\k{\mathcal K}
\renewcommand\T{\mathcal T}

\newcommand{\fG} {f : G \to G}
\newcommand{\ti} {\tilde}
\newcommand{\iNp} {indivisible Nielsen path}

\newcommand{\eg}{EG}
\newcommand{\noneg}{NEG}
\renewcommand\neg\noneg

\newcommand{\wt}{\widetilde}

\newcommand{\ct}{CT}
\newcommand\cts{CTs}
\newcommand\free{{\text{f}}}
\newcommand\nonfree{{\text{nf}}}

\DeclareMathOperator\interior{int}
\newcommand\bdy\partial

\newcommand\intersect\cap
\newcommand\union\cup
\newcommand\<\langle
\renewcommand\>\rangle
\newcommand\meet\wedge
\newcommand\composed{\circ}
\newcommand\cross\times
\newcommand\restrict{\bigm |}
\newcommand\wh{\widehat}
\newcommand\inject\hookrightarrow
\newcommand\reals{\mathbf{R}}

\newcommand\abs[1]{\left|#1\right|}

 \DeclareMathOperator\rank{rank}
 \newcommand\surjection\twoheadrightarrow
 
\newcommand\suchthat{\bigm|}

\newcommand\PGF{\PG^{\cal F}}
\newcommand\UPGF{\PGF \intersect \IA_n(\Z/3)}
\DeclareMathOperator\PG{PG}

\DeclareMathOperator\eigen{Eigen}
\DeclareMathOperator\Arc{Arc}
\newcommand{\ray}{\bdy F_n / F_n}
\newcommand{\cffs}{\F_\supp}

\DeclareMathOperator\core{core}
\DeclareMathOperator\height{height}

\newcommand{\gen}{\kappa}
\newcommand{\cN}{\mathcal N}

\newcommand\B{\mathcal B}
\renewcommand\L{\mathcal L}
\DeclareMathOperator\Eigen{Eigen}
\DeclareMathOperator\Twist{Twist}
\DeclareMathOperator\Axes\Twist
\DeclareMathOperator\Asym{Asym}
\DeclareMathOperator\Lam\Asym
\newcommand\M{\mathcal M}

\title{Subgroup decomposition in $\Out(F_n)$\\ Part II: A relative Kolchin theorem}

\author{Michael Handel and Lee Mosher}

\begin{document}

\maketitle

\begin{abstract}
This is the second in a series of four papers, announced in \cite{HandelMosher:SubgroupsIntro}, that develop a decomposition theory for subgroups of $\Out(F_n)$.

In this paper we relativize the ``Kolchin-type theorem'' of \BookTwo, which describes a decomposition theory for subgroups $\h \subgroup \Out(F_n)$ all of whose elements have polynomial growth. The Relative Kolchin Theorem, Theorem~E from \cite{HandelMosher:SubgroupsIntro}, allows subgroups $\h$ whose elements have exponential growth, as long as all such exponential growth is cordoned off in some free factor system~$\F$ which is invariant under every element of~$\h$. The conclusion is that a certain finite index subgroup of $\h$ has an invariant filtration by free factor systems going from $\F$ up to the full free factor system $\{[F_n]\}$ by individual steps each of which is a ``one-edge extension''. We also study the kernel of the action of $\Out(F_n)$ on $H_1(F_n;\Z/3)$, and we prove Theorem~B of \cite{HandelMosher:SubgroupsIntro} which describes strong finite permutation behavior of all elements of this kernel.
\end{abstract}

\section{Introduction}

The proof of the Tits alternative for subgroups $\h \subgroup \Out(F_n)$ breaks into cases handled separately in \BookOne\ and \BookTwo. The first of those papers focusses on the case that $\h$ has exponential growth, thereby reducing the Tits alternative to the case that each $\phi \in \h$ has polynomial growth, meaning that for some (any) choice of a marked graph $G$ and for any conjugacy class $c$ in~$F_n$, the length of the circuit in $G$ representing~$\phi^i(c)$ has a polynomial upper bound~in~$i$.

The second paper \BookTwo\ applies to subgroups $\h \subgroup \Out(F_n)$ such that each $\phi \in \h$ has polynomial growth and has unipotent image in $\Aut(H_1(F_n;\Z)) \approx \GL(n,\Z)$---such subgroups are said to be UPG, a property useful for ruling out certain finite order phenomena, as long as one is willing to pass to a finite index subgroup. The main theorem of \BookTwo, which we refer to here as the ``absolute Kolchin theorem'' in order to contrast with the main theorem of this article, has several equivalent formulations. We state one that is expressed in terms of the containment relation $\sqsubset$ amongst free factor systems (Section~\refGM{SectionSSAndFFS}).\footnote{``Section I.X.Y.Z'' or ``Theorem I.V.W'' refers to Section X.Y.Z or Theorem V.W of \PartOne.} A nested pair of free factor systems $\F'' \sqsubset \F'$ is said to be a \emph{one edge extension} (Definition~\ref{DefOneEdgExtFFS}) if $\F'',\F'$ are realized simultaneously by subgraphs $G'' \subset G' \subset G$ of some marked graph $G$ such that $G' \setminus G''$ is a single edge of~$G$. 

\begin{theorem*}[The absolute Kolchin theorem] Suppose that $\h \subgroup \Out(F_n)$ is a finitely generated UPG subgroup and that $\emptyset = \F_0 \sqsubset \F_1 \sqsubset \cdots \sqsubset \F_m = \{[F_n]\}$ is a maximal filtration by $\h$-invariant free factor systems. Then each step $\F_{i-1} \sqsubset \F_i$ is a one-edge extension.
\end{theorem*}


We are interested in a relative version of the absolute Kolchin theorem in which exponential growth is allowed but is entirely encapsulated in some proper free factor system $\F$ of~$F_n$ which is invariant under each element of the subgroup in question. Given $\phi \in \Out(F_n)$ and a proper, $\phi$-invariant free factor system~$\F$, we say that $\phi$ is of \emph{polynomial growth relative to~$\F$} if for some (any) marked graph $G$ having a subgraph $H \subset G$ that realizes~$\F$, and for any conjugacy class $c$ of $F_n$, there is a polynomial upper bound to the number of times that the circuit in $G$ representing $\phi^i(c)$ crosses edges of $G \setminus H$ (see Section~\ref{SectionPGF} for more details). Let $\PGF$ denote the subset of elements of $\Out(F_n)$ which are of polynomial growth relative to~$\F$; note that this is not a subgroup.

Our theorem also needs a hypothesis to rule out certain finite order phenomena. The concept of unipotence does not seem directly useful in our current relative context. Instead we focus on subgroups of the finite index subgroup $\IA_n(\Z/3) \subgroup \Out(F_n)$ which by definition is the kernel of the natural homomorphism $\Out(F_n) \to \Aut(H_1(F_n;\Z/3)) \approx \GL_n(\Z/3)$. This subgroup occurs in other results about $\Out(F_n)$: the usual proof that $\Out(F_n)$ is virtually torsion free shows that $\IA_n(\Z/3)$ is torsion free; and in \BookTwo\ it is proved in Proposition~3.5 that each polynomially growing element of $\IA_n(\Z/3)$ is unipotent. 
A significant proportion of this paper is devoted to the study of invariance properties of elements of $\IA_n(\Z/3)$; see the discussion below. We often restrict attention to subgroups of $\IA_n(\Z/3)$, but this is a mild restriction, in that any subgroup if $\Out(F_n)$ has a finite index subgroup contained in $\IA_n(\Z/3)$, namely its intersection with $\IA_n(\Z/3)$.

Here is our main theorem. 

\begin{theorem}[The relative Kolchin theorem] \label{relKolchin} 
Let $\h \subgroup \IA_n(\Z/3)$ be finitely generated and let $\F$ be a proper $\h$-invariant free factor system of $F_n$. If $\h \subset \PG^\F$, and if $\F = \F_0 \sqsubset \F_1 \sqsubset \F_2 \sqsubset \ldots \sqsubset \F_m = \{[F_n]\}$ is a maximal filtration by $\h$-invariant free factor systems containing $\F$, then each step $\F_{i-1} \sqsubset \F_i$ is a one-edge extension.  
\end{theorem}
\noindent
The hypothesis of maximality may be expressed in different words, saying that $\h$ is \emph{irreducible} relative to each step $\F_{i-1} \sqsubset \F_i$ of the filtration, which means that there is no $\h$-invariant free factor system strictly between $\F_{i-1}$ and $\F_i$. 

As in \BookTwo, Theorem~\ref{relKolchin} can be reformulated in terms of trees. Recall that the group $\Out(F_n)$ acts on the set of minimal actions of $F_n$ on simplicial trees (modulo $F_n$-equivariant homeomorphism). Given such a tree $T$ with trivial edge stabilizers, let $\F(T)$ be the free factor system consisting of the conjugacy classes of nontrivial vertex stabilizers of~$T$. The following theorem can be viewed as a relativization of Theorem~5.1 of \BookTwo. 

\begin{theorem} \label{ThmTreeRelKolchin} Let $\h \subgroup \IA_n(\Z/3)$ be finitely generated, let $\F$ be a proper $\h$-invariant free factor system of $F_n$, and suppose that $\h$ is irreducible relative to the extension $\F \subset \{[F_n]\}$. If $\h \subset \PG^\F$ then there exists an $\h$-invariant simplicial $F_n$-tree $T$ with trivial edge stabilizers, and with exactly one orbit of edges, such that $\F = \F(T)$.
\end{theorem}

\begin{proof}[Proof of equivalence of Theorems~\ref{relKolchin} and~\ref{ThmTreeRelKolchin}] \quad This proof is again similar to arguments found in \BookTwo\ on page~57. We need the following relations between trees and free factor systems, taken from \cite{HandelMosher:distortion} Section~4.1 which in turn depends on \BookOne\ Corollary~3.2.2:
\begin{enumerate}
\item For any $\h$-invariant free factor system $\F$ of $F_n$ such that $\F \sqsubset \{[F_n]\}$ is an one-edge extension there exists an $\h$-invariant simplicial tree $T$ having trivial edge stabilizers and just one edge orbit, such that $\F=\F(T)$.
\item For any $\h$-invariant simplicial tree $T$ with trivial edge stabilizers and one edge orbit, the free factor system $\F(T)$ is $\h$-invariant and $\F(T) \sqsubset \{[F_n\}$ is a one-edge extension.
\end{enumerate}
\noindent
Given $\F$ as in Theorem~\ref{ThmTreeRelKolchin}, and extending $\F \sqsubset \{[F_n]\}$ to a maximal $\h$-invariant filtration by free factor systems as denoted in Theorem~\ref{relKolchin}, we have $\F=\F_{m-1}$. The conclusion of Theorem~\ref{relKolchin} together with~(1) then implies the conclusion of Theorem~\ref{ThmTreeRelKolchin}.

For the converse, given a filtration by free factor systems as in Theorem~\ref{relKolchin}, and given $i=1,\ldots,m$, suppose first that $\F_i = \{[A]\}$ has a single component, and so we may consider the restricted subgroup $\h \restrict A \subgroup \Out(A)$ (see Section~\refGM{SectionRestrictedOuts}). The maximality hypothesis of Theorem~\ref{relKolchin} implies that $\h \restrict A$ is irreducible relative to the extension $\F_{i-1} \sqsubset \F_i=\{[A]\}$, and so the desired conclusion that $\F_i$ is a one-edge extension of $\F_{i-1}$ follows by applying (2) together with Theorem~\ref{ThmTreeRelKolchin} to the subgroup $\h \restrict A$. If $\F_i$ has more than one component then one simply works with the restriction of $\h$ to one component of $\F_i$ at a time.
\end{proof}

\subsection*{Invariance properties of $\IA_n(\Z/3)$}
Rotationless elements of $\Out(F_n)$ satisfy several invariance properties described in Lemma~3.30 of \recognition\ (and see Fact~\refGM{FactPeriodicIsFixed}). We prove analogous properties for elements of $\IA_n(\Z/3)$:
\begin{description}
\item[Lemma \ref{LemmaFFSComponent}:] For each $\psi \in \IA_n(\Z/3)$ we have:

$\bullet$ $\psi$ fixes each component of each $\psi$-invariant free factor system.

$\bullet$ $\psi$ fixes each element of its set of attracting laminations $\L(\psi)$.

\item[Theorem~\ref{ThmPeriodicConjClass}:] For each $\psi \in \IA_n(\Z/3)$, every $\psi$-periodic conjugacy class in $F_n$ is fixed by~$\psi$.
\item[Theorem~\ref{ThmPeriodicFreeFactor}:] For each $\psi \in \IA_n(\Z/3)$, every $\psi$-periodic free factor system in $F_n$ is $\psi$-invariant.
\end{description}
These four statement are ordered in increasing level of difficulty of proof. Lemma~\ref{LemmaFFSComponent} is proved quickly in Section~\ref{SectionIA3Elements}. The proof of Theorem~\ref{ThmPeriodicConjClass} takes up the bulk of Section~\ref{SectionIA3Elements}. The proof of Theorem~\ref{ThmPeriodicFreeFactor}, which depends on Theorem~\ref{ThmPeriodicConjClass}, takes up all of Section~\ref{SectionOneEdge}. 

Sections~\ref{SectionOneEdge}--\ref{SectionFCarriesAll} contain applications of Theorems~\ref{ThmPeriodicConjClass} and~\ref{ThmPeriodicFreeFactor} following several different strategies (Section~\ref{SectionNielsenPairsExist} depends on those theorems only by reference to other applications). For example, in applying Theorem~\ref{ThmPeriodicFreeFactor} to prove Theorem~\ref{ThmTreeRelKolchin} at the end of Section~\ref{SectionReduction}, we use the following strategy: given a finitely generated subgroup $\h \subgroup \IA_n(\Z/3)$, if one can produce a certain free factor system which is invariant under rotationless powers of each generator of $\h$, one can then use Theorem~\ref{ThmPeriodicFreeFactor} to conclude that the free factor system is invariant under the generators themselves, and hence is invariant under the entire subgroup~$\h$.


\paragraph{Connections with \BookTwo:} Our proofs of Theorems 1.1 and 1.2 closely follow the outline of the proof of Theorem 5.1 of \BookTwo, but there are many differences, arising from the need to make that outline work in the relative setting. Section~\ref{SectionLimitTrees} here corresponds to Section~4 of \BookTwo, Sections~\ref{SectionFCarriesAll} and~\ref{SectionNielsenPairsExist} here correspond to Section~5 of \BookTwo, and we have tried to indicate where the arguments in those sections have close parallels in \BookTwo. There are nonetheless substantial differences between the proofs here in Sections~\ref{SectionLimitTrees} and~\ref{SectionFCarriesAll} and the corresponding arguments in \BookTwo. 

Sections~\ref{SectionIA3Elements} and~\ref{SectionOneEdge}, on properties of $\IA_n(\Z/3)$, have no counterpart in \BookTwo. The analogue of Theorem~\ref{ThmPeriodicFreeFactor} for UPG subgroups of $\Out(F_n)$ is not used in \BookTwo; and the UPG analogue of Theorem~\ref{ThmPeriodicConjClass}, stated in \BookTwo\ as Proposition~3.16, is essentially a reference to Theorem~5.1.8 of \BookOne\ whose proof is quite different than that of Theorem~\ref{ThmPeriodicConjClass}.

\paragraph{Description of the contents and references to background material.} \quad \\
Section~\ref{SectionReduction} outlines the proofs of Theorems~\ref{relKolchin} and~\ref{ThmTreeRelKolchin} and reduces them to Theorem~\ref{ThmPeriodicFreeFactor} stated above and to two other propositions whose proofs are taken up in later sections. 

Readers familiar with background material on $\Out(F_n)$ such as relative train track theory may need to just lightly skim Section~\ref{SectionTheBasics} before picking up the main thread of the paper in Section~\ref{SectionReduction}. Readers who desire details of background material can consult Section~\refGM{SectionPrelim} of \PartOne\ for a full but terse outline including definitions, notations, and citations and/or quick proofs, regarding much preliminary material such as: free factor systems and more general subgroup systems; principle automorphisms; rotationless outer automorphisms; relative train track maps and \cts; complete splittings; et~cetera. Given the primary concerns of this paper, the preliminary material on attracting laminations and on exponentially growing or \eg\ strata of relative train track maps will be less important than the material on nonexponentially growing or \neg\ strata. 

Full citations to the original sources in \BookOne, \BookTwo, \recognition\ are found in Section~\refGM{SectionPrelim}. While we shall often cite a needed result from \PartOne\ without tracing back to the full citation (which can always be found in Section~\refGM{SectionPrelim}), nonetheless in the expectation that the reader may only have some of the above sources at hand we sometimes try to give a double citation for a major definition or result, for example the definition of a \ct\ is sometimes cited as ``\recognition\ Definition~7.4 (or see Definition \refGM{DefCT})''.

\setcounter{tocdepth}{2}
\tableofcontents

\section{Preliminaries}  
\label{SectionTheBasics}

As noted above this paper will depend heavily on the preliminary material laid out in Section~\refGM{SectionPrelim}. In this section we present additional preliminary material needed only in Part~II, some of which, as in Section~\refGM{SectionPrelim}, is just citations from the literature and/or quick proofs. The material on eigenrays in Section~\ref{SectionEigenrays} has some new material.

\subsection{Polynomial growth relative to a free factor system.} 
\label{SectionPGF}

Recall from Section~\refGM{SectionLineDefs} the set of lines $\B=\B(F_n)$, from Section~\refGM{SectionAttractingLams} the finite set $\L(\psi)$ of attracting laminations of~$\psi$ each of which is a subset of $\B$, and from Section~\refGM{SectionSubgroupLinesAndEnds} the concept of a free factor system carrying a subset of $\B$, which we apply here to the subset $\union\L(\psi) \subset \B$. 

Given $\psi \in \Out(F_n)$ and a $\psi$-invariant free factor system $\F$, we say that $\psi$ is of \emph{polynomial growth relative to~$\F$}, denoted $\psi \in \PGF$, if either of the equivalent conditions in the following lemma holds:

\begin{lemma}\label{LemmaFindingEG} 
Given $\psi,\F$ as above, the following are equivalent:
\begin{enumerate}
\item\label{ItemPGFGraph}
For some (any) marked graph $G$ and any core subgraph $K \subset G$ that realizes the free factor system~$\F$, the following holds: for any conjugacy class $[c]$ in $F_n$, the number of edges of $G \setminus K$ that are crossed by the circuit in $G$ realizing $\psi^i[c]$ is bounded above by a polynomial function of $i$.
\item\label{ItemPGFLam}
$\cup \L(\psi)$ is supported by $\F$.
\end{enumerate}
\end{lemma}

\begin{proof} The equivalence of the existential and universal quantifiers in item~\pref{ItemPGFGraph} follows from the bounded cancellation lemma (Fact~\refGM{FactBCCSimplicial}) applied to any homotopy equivalence of pairs from $(G,K)$ to any other candidate pair $(G',K')$ such that marking is preserved by that homotopy equivalence.

For proving equivalence of~\pref{ItemPGFGraph} and~\pref{ItemPGFLam} there is no loss in replacing $\psi$ by an iterate, so we may assume  that $\psi$ is rotationless and hence  $\psi$ is represented by a \ct\ $f \from G \to G$ in which $\F$ is realized by a core filtration element $K=G_r$ (\recognition\ Theorem~4.28, or see Theorem~\refGM{TheoremCTExistence}). We note by Fact~\refGM{FactLamsAndStrata} that \pref{ItemPGFLam} holds if and only if every \eg\ stratum is contained in~$G_r$, and the latter implies that every zero stratum is also contained in~$G_r$ by the clause (Zero Strata) in the definition of a \ct\ (\recognition\ Definition~7.4, or see Definition \refGM{DefCT}).

Assuming that \pref{ItemPGFLam} holds and so every stratum above $G_r$ is \neg, let $\sigma$ be any circuit in $G$ representing~$[c]$. The realization in $G$ of $\phi^k[c]$ is the circuit $f^k_\#(\sigma)$, and it follows from Fact~\refGM{FactNEGEdgeImage} that the number of edges of $G \setminus G_r$ that are crossed by the circuit $f^k_\#(\sigma)$ grows polynomially in $k$, with degree bounded by the number of strata above~$G_r$. 

Assuming that \pref{ItemPGFLam} does not hold and so some stratum $H_s$ with $s>r$ is \eg, there is a standard construction, using only the definition of a relative train tracks (Section~\refGM{SectionRTTDefs}), which produces an $s$-legal circuit $\sigma \subset G_s$ containing at least one edge of~$H_s$: any high enough iterate of any edge of $H_s$ contains an $s$-legal subpath in $G_s$ of the form $EwE$ for some oriented edge $E \subset H_s$, and one then takes $\sigma=Ew$. The number of edges of $H_s$ in the sequence of circuits $f^j_\#(\sigma)$ grows exponentially in~$j$, and so the conjugacy class $[c]$ represented by $\sigma$ exhibits that \pref{ItemPGFGraph} does not hold.
\end{proof}

\subsection{$F_n$-trees} 
\label{SectionFnTrees}
An \emph{$F_n$-tree} is an $\reals$-tree $T$ equipped with a minimal isometric action in which no point or end of the tree is fixed by the whole action. If the action of each element of $F_n$ is fixed point free then the tree is \emph{free}. We consider $T$ and $T'$ to be equivalent if there is an isometry $T \mapsto T'$ that conjugates the $F_n$ actions. Formally we use notations like $F \xrightarrow{\A} \Isom(T)$ for the action and $[T]$ for the equivalence class of $T$. Informally the actions are suppressed and the equivalence classes are implicit, and we often just write $T$ when we really mean~$[T]$. 

For each subset $A \subset T$ the \emph{stabilizer} of $A$ is the subgroup of $F_n$ defined by $\Stab(A) = \{g \in F_n \suchthat g \cdot x = x \,\,\,\text{for all}\,\,\, x \in A\}$. For each subgroup $A \subgroup F_n$ the \emph{fixed set of $A \subset F_n$ in $T$} is $\Fix(A) = \{x \in T \suchthat g \cdot x = x \,\,\,\text{for all}\,\,\, a \in A\}$.  

Each $F_n$ tree $T$ determines a translation length function $L_T \from \C(F_n) \to \reals$ that we think of as a point $L_T \in \reals^\C$. By \cite{CullerMorgan:Rtrees}  $T$ and $T'$ are equivalent if and only if  they determine the same length function.  There is an induced embedding of the set of equivalence classes of $F_n$ trees into $\reals^\C$ and we use this to topologize the space of  $F_n$ trees.  Thus $T_i \to T$ means that $L_{T_i} \to L_T$.

An $F_n$-tree $T$ is \emph{small} if for each nondegenerate arc $\alpha$ the subgroup $\Stab(\alpha)$ is trivial or cyclic, and furthermore $T$ is \emph{very small} \cite{CohenLustig:verysmall} if for each triod $\tau \subset T$ the subgroup $\Stab(\tau)$ is trivial, and for each $g \in F_n$ and each $i \ge 1$ we have $\Fix(g) = \Fix(g^i)$. It follows that for each nondegenerate arc $\alpha$ the subgroup $\Stab(\alpha)$ is either trivial or maximal infinite cyclic. 

For each small $F_n$-tree $T$ and each $x \in T$, the subgroup $\Stab(x)$ has finite rank, and there are only finitely many $F_n$-conjugacy classes of such subgroups~\cite{GJLL:index}; this set of conjugacy classes, denoted $\F(T)$, is an example of a \emph{subgroup system} of $F_n$ (Section~\refGM{SectionSSAndFFS}). Assuming furthermore that the stabilizer of every nondegenerate arc in~$T$ is trivial, there are in fact only finitely many $F_n$-orbits of points $x$ for which $\Stab(x)$ is nontrivial, and in this situation $\F(T)$ is called the \emph{vertex group system of $T$} (Section~\refGM{SectionVertexGroupSystems}). 
%
%
%
%
If furthermore the tree $T$ is simplicial---and still assuming edge stabilizers are trivial---then $\F(T)$ is a free factor system (\BookOne\ Section~2.6, or see Section~\refGM{SectionSSAndFFS}).

The group $\Out(F_n)$ acts on equivalence classes of $F_n$-trees, preserving the various properties considered above such as very small, simplicial, etc. Given $\phi \in \Out(F_n)$ and an $F_n$-tree $T$ with action denoted $F_n \xrightarrow{\A} \Isom(T)$, choose $\Phi \in \Out(F_n)$ representing~$\phi$, and define $T\phi$ to have the same underlying tree as $T$ but with the action $F_n \xrightarrow{\Phi} F_n \xrightarrow{\A} \Isom(T)$. We then have $L_{T\phi}[a] = L_T(\phi[a])$ for all $[a] \in \C$. The equivalence class of $T\phi$ is well-defined independent of the choice of~$\Phi$.

\begin{fact}[\cite{BestvinaFeighn:OuterLimits}] \label{LimitIsSimplicial}Suppose that  $T_i$ is a sequence of free simpicial $F_n$-trees and that $L_{T_i} \to L$ for some $L \in \reals^\C$. Suppose further that the set of positive values of $L$  is non-empty and bounded below. Then there is a very small simplicial $F_n$-tree $T$ such that $L = L_T$. \qed
\end{fact}

The following is a general version of Cooper's bounded cancellation theorem in which the target may be any very small simplicial $F_n$-tree.

\begin{fact}[Bounded Cancellation Lemma \BookZero] 
\label{FactBCCVerySmall} 
Suppose that $S$ is a free simplicial $F_n$-tree, that $T$ is a very small simplicial $F_n$-tree and that $f :S \to T$ is an $F_n$-equivariant map. Then there is a constant $B$ such that for any arc $[x,y] \subset S$ its image $f[x,y]$ is contained in the $B$ neighborhood of the arc $[f(x),f(y)] \subset T$. The smallest such value for $B$ is called the bounded cancellation constant for $f$ and is denoted $\bcc(f)$. \qed
\end{fact}

\subsection{One-edge extensions: free factor systems versus graphs}
The following definition incorporates evidently equivalent versions of the concept of a one-edge extension of free factor systems:

\begin{definition}  \label{DefOneEdgExtFFS}
Given a properly nested pair of free factor systems $\F \sqsubset \F'$ of $F_n$, we say that $\F'$ is a \emph{one-edge extension} of $\F$ if either of the following equivalent conditions holds:
\begin{enumerate}
\item There exists a marked graph $G$ and core subgraphs $K \subset K'$ realizing $\F,\F'$ respectively, such that $K' \setminus K$ is one edge of~$G$.
\item One of the following holds: 
\begin{enumerate} 
\item There exists $[F_1] \in \F$ and $[F_2] \in \F'$ such that $\rank(F_2) = \rank(F_1) + 1$ and $\F - \{[F_1]\} = \F' - \{[F_2]\}$; or
\item There exists $[F_1] \ne [F_2] \in \F$ and $[F_3] \in \F'$ such that $\rank(F_3) = \rank(F_1) + \rank(F_2)$ and $\F - \{[F_1],[F_2]\} = \F' - \{[F_3]\}$. 
\end{enumerate} 
\end{enumerate}
If $\F \sqsubset \F'$ is not a one-edge extension then it is a \emph{multi-edge extension}.
\end{definition} 

As alluded to in the introduction, by results of \cite{HandelMosher:distortion}, Section~4.1 it follows that if $\F$ is a free factor system of $F_n$ then $\F \sqsubset \{[F_n]\}$ is a one-edge extension if and only if there exists a simplicial $F_n$-tree $T$ with trivial edge stabilizers such that $\F(T) = \F$ and such that  $T$ has exactly one orbit of edges. 
Furthermore, this defines an $\Out(F_n)$-invariant bijection between the set of free factor systems $\F$ of which $\{[F_n]\}$ is a one-edge extension and the set of $F_n$-equivariant homeomorphism classes of simplicial $F_n$-trees with trivial edge stabilizers and exactly one orbit of edges.

\medskip

Consider a marked graph $G$ and two core subgraphs $G_1 \subset G_2$. We say that $G_2$ is a \emph{one-edge extension} of $G_1$ if $G_2 \setminus G_1$ is either an arc whose endpoints are attached to $G_1$ (possibly to the same point of $G_1$) or a loop disjoint from $G_1$. We say that $G_1$ is a \emph{lollipop extension} of $G_1$ if $G_2 \setminus G_1$ is the union of a loop called the ``lollipop'' and an arc called the ``stem'' such that one endpoint of the stem is identified with a vertex on the lollipop, and the opposite end of the stem is the unique point of $G_1 \intersect (G_2 \setminus G_1)$.

The following simple fact shows that there is a sleight ambiguity in the ``one-edge extension'' terminology which we must keep in mind:

\begin{lemma} 
\label{LemmaOneEdgeVersusLollipop}
If $\F_1 \sqsubset \F_2$ is a one-edge extension of free factor systems, and if $G$ is a marked graph with core subgraphs $G_1 \subset G_2$ realizing $\F_1,\F_2$ respectively, then $G_2$ is either a one-edge extension or a lollipop extension of $G_1$.
\qed\end{lemma}

The source of this ambiguity is that any lollipop extension $G_1 \subset G_2 \subset G$ can be converted into a one-edge extension by collapsing the stick of $G_2 \setminus G_1$ to a point. Conversely, for any one edge extension $G_1 \subset G_2 \subset G$, if $G_2 \setminus G_1$ is a single loop edge of $G$ intersecting $G_1$ at a single vertex then this construction can be reversed, pulling the loop apart from the rest of $G$ and inserting a stick.

\subsection{Asymptotic data: attracting laminations, eigenrays, and twistors}
\label{SectionAsymptotic}

Associated to any rotationless $\phi \in \Out(F_n)$ are its \emph{asymptotic data} which are organized into three finite sets: the set of expanding laminations $\L(\phi)$, each of which is a certain closed subset of the set of lines $\B(F_n)$ (see Section~\refGM{SectionLineDefs}); the set of eigenrays $\Eigen(\phi)$, each an element of abstract set of rays $\bdy F_n / F_n$ (also see Section~\refGM{SectionLineDefs}); and the set of twistors $\Twist(\phi)$, each of which is a periodic element of $\B(F_n)$. The union of these three sets is denoted
$$\Asym(\phi) = \L(\phi) \union \Eigen(\phi) \union \Twist(\phi)
$$ 
Each of these three sets has a description in terms of any \ct\ representing~$\phi$, and each has an invariant definition expressed without reference to relative train tracks. Very briefly, given a \ct: \eg\ strata correspond bijectively to expanding laminations; superlinear \neg\ strata correspond bijectively to eigenrays; linear families of linear \neg\ strata correspond bijectively to twistors. For expanding laminations, which are reviewed mathematically in \PartOne, we simply recall the citations. For twistors, we accompany the citations with a brief mathematical review for the readers convenience. For eigenrays we must develop some new material.

\subsubsection{Expanding laminations and twistors}
\label{SectionLaminationsAndTwistors}
\paragraph{Expanding laminations.} The invariant definition of the expanding laminations $\L(\phi)$ is given in \BookOne\ Definition~3.1.5 (see also Definition~\refGM{DefAttractingLaminations}). Given a relative train track representative $f \from G \to G$ of $\phi$, and assuming that $f$ is ``\eg-aperiodic'' meaning that the transition matrix of each \eg\ stratum of $f$ is a Perron-Frobenius matrix (which always holds if $\phi$ is rotationless), the bijection between $\L(\phi)$ and the set of \eg\ strata of $G$, and the definition of the expanding lamination associated to a particular \eg\ stratum, is given in \BookOne\ Definition 3.1.12 (see also Fact~\refGM{FactLamsAndStrata}).

\paragraph{Twistors.} Roughly speaking the ``twistors'' of a rotationless $\phi \in \Out(F_n)$ are the unoriented conjugacy classes around which Dehn twist pieces of $\phi$ do their twisting. The definition of a ``twistor of $\phi$'' in the context of a \ct, with alternate terminology ``axis of~$\phi$'', is given in \cite{FeighnHandel:abelian} just preceding Notation~2.12. The invariant definition, independent of relative train track representatives, is given in \recognition\ just preceding Remark~4.39, but using only the terminology ``axis''. We note also that the ``twistor'' terminology is used in \cite{CohenLustig:DehnTwist}, in the context of a Dehn twist outer automorphism $\phi$, to incorporate both the class around which $\phi$ does its twisting and the numerical amount of twisting.

In order to avoid conflict with other meanings of ``axis'', we adopt here the terminology ``twistor'' for both the \ct\ definition and the invariant definition; see Definitions~\ref{DefTwistorCT} and~\ref{DefTwistorInvariant} below.

Two elements $a,b \in F_n$ are said to be \emph{unoriented conjugate} if $a$ is conjugate to $b$ or~$b^\inv$. The unoriented conjugacy class of $a$ is denoted $[a]_u$. In any marked graph $G$ the circuit representing $[a]_u$ is unique up to orientation reversal. We say that $[a]_u$ is \emph{root free} if for some (any) representative $a$ and any $c \in F_n$, the equation $a=c^k$ implies $k=\pm 1$; equivalently, in some (any) marked graph $G$, the circuit representing $[a]_u$ is root free as in Section~\refGM{SectionLineDefs}, meaning that this circuit is not an iterate of a shorter circuit. Recall also from Section~\refGM{SectionLineDefs} that an \emph{axis} in the set $\B$ is an element $\gamma$ such that its lifts $\ti\gamma \in \wt\B$ are the lines fixed by representatives of some nontrivial conjugacy class, and that the set of axes corresponds one-to-one to the set of root free unoriented conjugacy classes in~$F_n$.

The definition of twistors in the context of a \ct\ is as follows:


\begin{definition} 
\label{DefTwistorCT} 
Consider a \ct\ $f \from G \to G$ representing $\phi$, and a linear edge $E_s \subset G$ of height~$s \ge 2$. From (Linear Edges) in the definition of a \ct\ (\recognition\ Definition~4.7, or Definition~\refGM{DefCT}) we have $f(E_s) = E_s w_s^{d_s}$ where $d_s \ne 0$ is an integer and $w_s$ is a closed, root free circuit of height $\le s-1$ which is a Nielsen path for~$f$. The unoriented conjugacy class $[a]_u$ determined by $w_s$ is called the \emph{twistor for $E_s$}. We also say that $[a]_u$ is a \emph{twistor for $f$} if it is a twistor for some linear edge of $G$, and we let $\Twist(f)$ denote the set of twistors for $f$. 

The definitions of some familiar notions can be formulated in terms of twistors. From (Linear Edges) in the definition of a \ct\ (\recognition\ Definition~4.7, or see Definition~\refGM{DefCT}), two linear edges $E_s,E_t$ of $G$ belong to the same \emph{linear family} of $f \from G \to G$ if and only if they have the same twistor, and if this is so then $f(E_s) = E_s w^{d_s}$, and $f(E_t) = E_t w^{d_t}$ where the root free closed path $w$ represents their common twistor, and where $d_s \ne d_t$ if $E_s \ne E_t$. And from (\recognition\ Definition~4.1, or see Definition~\refGM{DefSplittings}) an \emph{exceptional path} of $f$ is any path of the form $E_s w^p \overline E_t$ where $E_s \ne E_t$ are linear edges in the same linear family, $w$ represents their common twistor, and the exponents $d_s,d_t$ have the same sign.
\end{definition}

Here is the invariant definition of twistor:

\begin{definition}[\recognition, preceding Remark 4.39]
\label{DefTwistorInvariant} 
Given a rotationless $\phi \in \Out(F_n)$, an unoriented root-free conjugacy class $[a]_u$ is a \emph{twistor} of $\phi$ provided there exist $\Phi_1 \ne \Phi_2 \in P(\phi)$ and $c \in F_n$ conjugate to $a$ such that $\Phi_1(c)=\Phi_2(c)=c$. Let $\Twist(\phi)$ denote the set of all twistors of $\phi$.
\end{definition}
\noindent


\begin{fact}\label{FactTwistor} 
For any rotationless $\phi \in \Out(F_n)$ and any \ct\ $f \from G \to G$ representing $\phi$ we have $\Twist(\phi)=\Twist(f)$. Furthermore:
\begin{enumerate}
\item \label{ItemTwistorsFinite}
The set $\Twist(f)$ is finite, and corresponds bijectively to the set of linear families of linear \neg\ edges of $f$.
\item \label{ItemIntersectingFixed}
For all $\Phi \ne \Phi' \in P(\phi)$, if $\Fix(\Phi) \intersect \Fix(\Phi')$ is nontrivial then there exists a root-free $a \in F_n$ such that $[a]_u \in \Twist(\phi)$ and $\Fix(\Phi) \intersect \Fix(\Phi') = \<a\>$.
\end{enumerate}
\end{fact}

\begin{proof} The equation $\Twist(\phi) = \Twist(f)$ follows from \recognition\ Lemma~4.40 which says that for each $[a]_u \in \Twist(\phi)$ and each $b$ representing $[a]_u$ there exists a ``base principal automorphism'' $\Phi_0 \in P(\phi)$ fixing $b$ with the following property: for any \ct\ $f \from G \to G$ representing $\phi$, the linear edges in the linear family of $[a]_u$ correspond bijectively with the set of all $\Phi \ne \Phi_0 \in P(\phi)$ that fix~$b$. Item~\pref{ItemTwistorsFinite} follows from finiteness of the set of linear edges of~$f$; the bijection holds by definition. 

To prove~\pref{ItemIntersectingFixed}, we have $\Fix(\Phi) \intersect \Fix(\Phi') \subset \Fix(\Phi^\inv \Phi') = \Fix(i_c)$ for some inner automorphism $i_c \from x \mapsto cxc^\inv$. Let $c = a^i$ for some root free, nontrivial $a \in F_n$ and some integer $i \ge 1$. Since $c$ determines $a$ it follows that $\Fix(\Phi) \intersect \Fix(\Phi') = \<a\>$.
\end{proof}

\subsubsection{Eigenrays.} 
\label{SectionEigenrays}
Given a \ct\ $f \from G \to G$, each nonfixed \neg\ edge $E$ with fixed initial direction generates a ray fixed by $f_\#$ which has the form $R = E \cdot u \cdot f_\#(u) \cdot f^2_\#(u) \cdot f^3_\#(u) \cdot \ldots$ (\recognition\ Lemma~4.36 or see Fact~\refGM{FactSingularRay}), where $f(E) = E \cdot u$ (Fact~\refGM{FactNEGEdgeImage}). In the case that $E$ is not linear we have introduced the terminology of ``principal direction'' for $E$ and ``principal ray'' for~$R$ (Definition~\refGM{DefSingularRay}).
 
We focus here more narrowly on principal rays generated by nonfixed, nonlinear \neg\ edges (aka \emph{superlinear edges}), referring to such rays as ``eigenrays''. Although the term ``eigenray'' has not appeared before in the literature, this term has been used informally in the relative train track community for some time already, and the concept of an eigenray plays an important role in the Recognition Theorem of \recognition. Our main results here are the \ct\ definition of eigenrays just mentioned, an invariant definition independent of \cts, and a proof of equivalence of the two definitions. 


\begin{definition}\label{DefEigenrayCT}
Consider a \ct\ $f \from G \to G$. For each superlinear \neg\ edge $E_s$ with fixed initial direction, the ray $R$ generated by $E_s$ is called an \emph{eigenray of $f$ in $G$}, and the corresponding abstract ray, denoted $\Eigen(E_s) \in \bdy F_n / F_n$, is called an \emph{eigenray for $f$}. Let $\Eigen(f) \subset \bdy F_n / F_n$ be the set of eigenrays for~$f$.
\end{definition}

Next comes the invariant definition of eigenrays. Consider a ray $\xi \in \bdy F_n / F_n$. By Fact~\refGM{FactFFSPolyglot} the free factor system $\F_\supp(\xi)$ has one component. One may therefore choose a free factor $A \subgroup F_n$ and a point $\ti\xi \in \bdy A \subset \bdy F_n$ so that $\F_\supp(\xi) = \{[A]\}$ and so that $\ti\xi$ is a representative of $\xi$. Alternatively, one may just choose the representative $\ti\xi \in \bdy F_n$, and this determines the choice of $A$ uniquely by the requirement that $\ti\xi \in \bdy A$ (again using Fact~\refGM{FactBoundaries} and malnormality of~$A$).

\begin{definition}\label{DefEigenraysInvariant}
Given a rotationless $\phi \in \Out(F_n)$ and $\xi \in \bdy F_n / F_n$, we say that \emph{$\xi$ is an eigenray of $\phi$} if the following conditions \pref{ItemIsInFixPlus} and \pref{ItemNotInFixPlusRestricted} hold (with choices $A$, $\ti\xi$ as in the previous paragraph):
\begin{enumerate}
\item\label{ItemIsInFixPlus}
There exists a (necessarily unique) $\Phi \in P(\phi)$ such that $\ti\xi \in \Fix_+(\wh\Phi)$. In particular $\phi$ fixes $\xi$ and so $\phi$ fixes $\F_\supp(\xi) = \{[A]\}$. 
\item\label{ItemNotInFixPlusRestricted}
The set $\Fix_N(\Phi) \intersect \bdy A$ is a single point contained in $\Fix_+(\Phi)$.
\end{enumerate}
Let $\Eigen(\phi) \subset \bdy F_n / F_n$ denote the set of eigenrays of~$\phi$.
\end{definition}

\textbf{Remarks on the definition.} Once the choices are made, existence of $\Phi$ implies uniqueness, because $\ti\xi$ is not fixed by any inner automorphism of $F_n$ (Facts~\refGM{FactFPBasics} and~\refGM{LemmaFixPhiFacts}) but any two choices of $\Phi$ differ by inner automorphism. It is easy to see that item~\pref{ItemIsInFixPlus} is independent of the choice of $\ti\xi \in \bdy F_n$, and item~\pref{ItemNotInFixPlusRestricted} is independent of the choice of $A$ within its conjugacy class and of the choice of $\ti\xi$.


\begin{lemma}\label{LemmaEigenrayDefs}
For any rotationless $\phi \in \Out(F_n)$ and any \ct\ $f \from G \to G$ representing $\phi$ we have $\Eigen(\phi) = \Eigen(f)$. This set is finite, and the correspondence $E \mapsto \Eigen(E)$ is a bijection between the nonfixed, nonlinear \neg\ directions of $G$ and the set $\Eigen(f)$.
\end{lemma}

\begin{proof} Once the equation $\Eigen(\phi)=\Eigen(f)$ is proved, finiteness follows finiteness of the graph $G$. Surjectivity onto $\Eigen(f)$ of the map $E \mapsto \Eigen(E)$ is true by definition. 

\smallskip

To prove injectivity of $E \mapsto \Eigen(E)$, consider two superlinear \neg\ edges $E_i \subset G$, $i=1,2$, such that $\Eigen(E_1)=\Eigen(E_2) = \xi \in \bdy F_n / F_n$. Let $v_i$ be the initial vertex of $E_i$, a principal vertex by Definition~\refGM{DefPrincipalVertices} and so $E_i$ is a principal direction by Definition~\refGM{DefPrincipalDirection}. Let $R_i$ be the ray in $G$ generated by $E_i$, so $R_1,R_2$ each realize $\xi$ in $G$. For $i=1,2$, pick a principal lift $\ti f_i \from \wt G \to \wt G$ of $f$ with initial vertex $\ti v_i$ lifting~$v_i$, let $\wt R_i$ be the lift of~$R_i$ with initial vertex $\ti v_i$ and initial direction $\wt E_i$ lifting $E_i$, and let $\ti\xi_i \in \bdy F_n$ be the limit point of~$\wt R_i$. Let $\Phi_i \in \Aut(F_n)$ be the automorphism corresponding to $\ti f_i$. The two points $\ti\xi_1,\ti\xi_2 \in \bdy F_n$ have the same orbit~$\xi \in \bdy F_n / F_n$, and so after translating $\ti\xi_1$ by the appropriate element of $F_n$, which amounts to rechoosing $\Phi_1$ in its isogredience class, we may assume $\ti\xi_1=\ti\xi_2$ denoted~$\ti \xi \in \bdy F_n$. If $\Phi_1 \ne \Phi_2$ then $i_\gamma = \Phi_1 \Phi_2^\inv$ is a nontrivial inner automorphism, $\ti\xi \in \Fix(\hat\gamma)$, and $\gamma \in \Fix(\Phi_1)$ by Fact~\refGM{FactFPBasics}, and so $\xi \in \bdy\Fix(\Phi_1)$; but then by Fact~\trefGM{FactSingularRay}{ItemRayEndsAtAttr} it follows that $E_1$ is not a principal direction, a contradiction. Having shown that $\Phi_1 = \Phi_2$, if $\ti v_1 \ne \ti v_2$ then the path $[\ti v_1,\ti v_2]$ in $\wt G$ projects to a Nielsen path for~$f$ of the form $E_1 \alpha \overline E_2$ where $\alpha$ is a path in~$G_{r-1}$ and $r$ is the maximum of the heights of $E_1,E_2$; but then applying (\neg\ Nielsen Paths) in the definition of a \ct\ (\recognition\ Definition~4.7 or see Definition~\refGM{DefCT}) it follows that $E_1,E_2$ are linear, a contradiction. We have shown that $\ti v_1 = \ti v_2$ and so $\wt R_1 = \wt R_2$, $\wt E_1 = \wt E_2$, and $E_1 = E_2$, completing the proof of injectivity.

\smallskip

We prove next that $\Eigen(\phi) \subset \Eigen(f)$. Consider $\xi \in \Eigen(\phi)$ and adopt the notation of Definition~\ref{DefEigenraysInvariant}. By Fact~\refGM{FactSingularRay} and Definition~\refGM{DefSingularRay}, there exists a principal direction $E \subset G$ with initial vertex $v$ generating a principal ray $R$ which realizes $\xi$ in $G$. We must prove that $E$ is an \neg\ edge of $G$. If not then $E$ is contained in some \eg-stratum $H_r$ associated to some $\Lambda \in \L(\phi)$. Let $\ti f \from \wt G \to \wt G$ be the principal lift of $f$ corresponding to $\Phi$, and so there is a lift $\ti v$ of $v$ fixed by $\ti f$, and a lift $\wt E$ of $E$ which is a principal direction for $\ti f$ generating a ray $\wt R$ which is a lift of $R$ and which converges to $\ti\xi$. Applying Lemma~\refGM{LemmaEGPrincipalRays}, there is a leaf $\ell$ of $\Lambda$ lifting to a line $\ti\ell$ invariant under $\ti f_\#$ such that both ends of $\ti\ell$ are in $\Fix_N(\Phi)$. Applying \recognition\ Lemma~3.26~(2) (or see Fact~\trefGM{FactSingularRay}{ItemEGPrincipalRay}), the lamination $\Lambda$ is the weak accumulation set of $\xi$. Applying Fact~\trefGM{FactWeakLimitLines}{ItemWeakLimitRayAccSame} it follows that $\F_\supp(\ell) \sqsubset \F_\supp(\Lambda) = \F_\supp(\xi) = \{[A]\}$, and so $\bdy\ti\ell \subset \bdy (gAg^\inv)$ for some $g \in F_n$, but since $\bdy\ti\ell \intersect \bdy A \ne \emptyset$ it follows by malnormality of $A$ and Fact~\refGM{FactBoundaries} that $\bdy\ti\ell \in \bdy A$. Both endpoints of $\bdy\ti\ell$ are therefore in $\Fix_N(\Phi) \intersect \bdy A$, contradicting Definition~\ref{DefEigenraysInvariant}~\pref{ItemNotInFixPlusRestricted}.

\smallskip

We prove finally that $\Eigen(f) \subset \Eigen(\phi)$, by reducing it to Lemma~\ref{LemmaPointsAtInfinity} which is stated and proved below. Consider $\xi \in \Eigen(f)$. Adopting the notation of Definition~\ref{DefEigenrayCT}, $\xi$ is represented in $G$ by the ray $R$ generated by the \neg\ superlinear edge $E_s$ with initial principal vertex $v$ and terminal vertex~$w$. In $\wt G$ choose a lift $\ti v$ of $v$, let $\ti f \from \wt G \to \wt G$ be the principal lift of $f$ fixing $\ti v$, and let $\wt E_s$ be the lift of $E_s$ with initial vertex $s$. Let $\Gamma$ be the component of the full pre-image of $G_{s-1}$ in $\wt G$ such that $\Gamma$ contains the terminal endpoint $\ti w$ of $\wt E_s$. Let $B \subgroup F_n$ denote the stabilizer of $\Gamma$, a free factor representing the component of the free factor system $[\pi_1 G_{s-1}]$ that corresponds to the component of $G_{s-1}$ containing~$w$. Let $\wt R$ be the ray in $\wt G$ generated by $\wt E_s$, and let $\ti\xi$ be the endpoint of~$\wt R$, and so $\ti\xi \in \bdy F_n$ is a representative of $\xi \in \bdy F_n / F_n$. We may now apply Lemma~\ref{LemmaPointsAtInfinity} below, concluding that $\ti\xi$ is the unique point of $\Fix_N(\Phi) \intersect \bdy\Gamma = \Fix_N(\Phi) \intersect \bdy B$ and $\ti\xi \in \Fix_+(\Phi)$.
%
%
%
For verifying Definition~\ref{DefEigenraysInvariant}~\pref{ItemNotInFixPlusRestricted}, we may choose $A \subgroup F_n$ to be the unique free factor such that $\ti\xi \in \bdy A$ and $\F_\supp(\xi) = \{[A]\}$. 
Since $\ti\xi \in \bdy B$ it follows that $[A] \sqsubset [B]$, and since $\bdy A \intersect \bdy B \ne \emptyset$, it follows by Facts~\refGM{FactBoundaries} and~\refGM{FactGrushko} that $\bdy A \subset \bdy B$ and so $A \subgroup B$. 
It follows that $\ti\xi$ is the unique point of $\Fix_N(\wh\Phi) \intersect \bdy A$.
This completes the proof of Lemma~\ref{LemmaEigenrayDefs}, subject to stating and proving the next lemma.
\end{proof}

Item~\pref{ItemNonlinearPointsAtInfinity} of the next lemma is all that is needed in the last paragraph of the preceding proof, however both items will be used in Section~\ref{SectionOneEdge}. The proof of the lemma cites several results from \recognition\ Section~3.3.

\begin{lemma}\label{LemmaPointsAtInfinity}  
Let $f \from G \to G$ be a \ct\ and let $E = H_s \subset G$ be a non-fixed \neg\ edge of height~$s$. Let $\wt E \subset \wt G$ be a lift of $E$, let $\ti f \from \wt G \to \wt G$ be the principal lift that fixes the initial endpoint of $\wt E$, and let $ \Gamma \subset \wt G$ be the component of the full pre-image of $G_{s-1}$  that contains the terminal endpoint $\ti w$ of $\wt E$. Then $\Gamma$ is $\ti f$-invariant, $\Fix(\ti f \restrict  \Gamma) = \emptyset$, and the following also hold:
\begin{enumerate}
\item \label{ItemNonlinearPointsAtInfinity}
If $E$ is nonlinear then $\Fix_N(\hat f) \cap \partial \Gamma$ is a singleton, contained in $\Fix_+(\hat f)$, equal to the endpoint of the ray in $\wt G$ generated by $\wt E$.
\item \label{ItemLinearPointsAtInfinity}
If $E$ is linear then $\Fix_N(\hat f) \cap \partial  \Gamma = \bdy A$ where $A \subgroup F_n$ is the infinite cyclic group consisting of all covering translations that commute with $\ti f$ and preserve~$ \Gamma$.
\end{enumerate}
\end{lemma}

\begin{proof} Let $u$ be the oriented path in $G_{s-1}$ with initial point $w$ such that $f(E)=Eu$, and let $\ti u$ be the lift with initial point $\ti w$ such that $\ti f(\wt E) = \wt E \ti u$.  Since the full pre-image of $G_{s-1}$ is $\ti f$-invariant and since $\Gamma$ contains $\ti u$, it follows that $\Gamma$ is $\ti f$-invariant. If $\Fix(\ti f \restrict  \Gamma)$ were nonempty then letting $\ti\tau$ be the path in $\Gamma$ from $\ti w$ to a point of $\Fix(\ti f \restrict  \Gamma)$, it would follows that $\wt E \ti \tau$ projects to a Nielsen path $E \tau$ with $\tau \subset G_{s-1}$, contradicting Property (\neg\ Nielsen Paths) in the definition of a \ct\ (\recognition\ Definition~4.7, or Definition~\refGM{DefCT}). 

Given $\ti z \in  \Gamma$ and $Q \in \bdy  \Gamma$, we say that \emph{$\ti z$ moves toward $Q$ under the action of $\ti f$} if the ray from $\ti f(\ti z)$  to $Q$ does not contain $\ti z$.  Let $S$ be the set of points $Q \in \Fix_N(\hat f) \cap \partial  \Gamma$ for which there exists a point of $ \Gamma$ that moves toward $Q$ under the action of $\ti f$. 
The fact that $S \ne \emptyset$ follows by applying either Lemma~3.23 of \recognition\ or Lemma~4.36~(2) of \recognition\ (for the latter see also Fact~\trefGM{FactSingularRay}{ItemDirectionToAttractor}). By applying Lemma  3.16  of \cite{FeighnHandel:recognition} and the fact that $\Fix(\ti f \restrict  \Gamma) = \emptyset$ it follows that $S$ has at most one point and hence $S = \{Q\}$. 
It follows that if $L_a$ is any non-trivial  covering translation that commutes with $\ti f$ then either $L_a(Q) =  Q$ or $L_a(Q) \not \in \partial  \Gamma$. Letting $A \subgroup F_n$ be the subgroup of all $a \in F_n$ for which $L_a$ commutes with $\ti f$ and $L_a(\Gamma) = \Gamma$, it follows $A$ is trivial or infinite cyclic. By applying Lemma~3.15 of \recognition\ it follows that $\Fix_N(\hat f) \cap \partial  \Gamma$ is either one or two points, the former case occurs if and only if that one point is in $\Fix_+(\hat f)$ which is the conclusion of~\pref{ItemNonlinearPointsAtInfinity}, that the latter case occurs if and only if $A$ is infinite cyclic and $\Fix_N(\hat f) \cap \partial\Gamma = \bdy A$ which is the conclusion of~\pref{ItemLinearPointsAtInfinity}.
 
If $E$ is linear then $u$ is a closed Nielsen path for $f$ and the covering translation $L_a$ that maps the initial endpoint of $\ti u$ to the terminal endpoint of $\ti u$ commutes with $\ti f$ and preserves $\Gamma$, so $A$ is infinite cyclic and the conclusion of~\pref{ItemLinearPointsAtInfinity} follows. If $E$ is not linear then by Lemma~\trefGM{FactSingularRay}{ItemEGPrincipalRay} we have $Q \in \Fix_+(\hat f)$, implying that $Q \not\in \bdy\Fix(\Phi)$ and so $Q$ is not a fixed point of a covering translation that commutes with $\ti f$; the conclusion of~\pref{ItemNonlinearPointsAtInfinity} follows.
\end{proof}

Our last lemma gives a useful carrying property of eigenrays that will be used several times later in the paper.

\begin{lemma}\label{LemmaEigenrayCarried} If $\fG$ is a \ct\ and if $E_s$ is a superlinear \noneg\ edge satisfying $f(E_s) = E_s \cdot u_s$ then for each filtration element $G_r$ with $r \le s$, the eigenray $\Eigen(E_s)$ is carried by $[\pi_1 G_r]$ if and only if $u_s \subset G_r$.
\end{lemma}

\begin{proof}  The  if direction is obvious.  For the only if direction it suffices to assume that  $f^k_\#(u_s) \subset G_r$ for some $k \ge 1$ and prove that $f^{k-1}_\#(u_s) \subset G_r$.  The endpoint $x$  of $u_s$ is a fixed point because $u_s$ is a closed path.  Since $f \restrict G_r $ is a homotopy equivalence, there is a  closed  path $\tau \subset G_r$  based at $x$ such that $f_\#(\tau) = f^k_\#(u_s)$.  Since  $f$ is a homotopy equivalence,  $\tau$ is the unique path in $G$ such that $f_\#(\tau) = f^k_\#(u_s)$.  It follows that   $f^{k-1}_\#(u_s) = \tau \subset G_r$.
\end{proof}

\subsection{Complete splittings rel $G_r$} 
In many of our arguments we work with a \ct\ $f \from G \to G$ representing $\phi \in \Out(F_n)$ and a filtration element $G_r$ which carries every lamination in~$\L(\phi)$. In such a situation we often treat $G_r$ as a black box, in particular we do not wish to worry about details of splittings of paths in $G_r$. For this reason we introduce here the notion of a complete splitting of a path relative to $G_r$.

The next several facts are about properties of \cts\ above the highest \eg\ stratum. The conclusion of our next fact is built into the hypotheses of many of our statements. 
 
\begin{fact}
\label{all neg} 
Suppose that $\fG$ is a \ct\ representing $\phi$ and that $G_r$ carries $\L(\phi)$.  Then every stratum above $G_r$ is \noneg. 
\end{fact}

\begin{proof} Since $G_r$ carries $\L(\phi)$, by Fact~\refGM{FactLineRealizedCarried} the realization in $G$ of each line of each lamination in $\L(\phi)$ is contained in $G_r$, and so by Fact~\refGM{FactLamsAndStrata} each \eg\ stratum is contained in $G_r$. By property (Zero Strata) of Definition~\refGM{DefCT}, each zero stratum is also contained in $G_r$.
\end{proof}

The following basic splitting property of paths with \neg\ height applies for more general relative train track maps than we state it for here.

\begin{fact}[Lemma 4.1.4 of \BookOne]
\label{FactBasicNEGSplitting} 
Suppose that $\fG$ is a \ct, that $\sigma$ is a path or circuit with height $s$ and that $H_s$ is \neg.  Then the decomposition of $\sigma$ into subpaths obtained by subdividing at the initial vertex of each occurence of $E_s$ in $\sigma$ and at the terminal vertex of each occurence of $\overline E_s$ in $\sigma$ is a splitting.
\end{fact}

Recall the definition of \emph{complete splittings} of paths, from Definition~4.3 of \recognition\ (and see Definition~\refGM{DefCompleteSplitting}). Item~\pref{ItemRelSplitDef} of the following fact defines our relative version of complete splittings.

\begin{fact}[Definition and properties of complete splittings rel~$G_r$]
\label{FactRelSplit}
Suppose that $\fG$ is a \ct\  and that every edge with height $>r$ is \noneg. 
\begin{enumerate}
\item\label{ItemRelSplitDef}
Every completely split path or circuit has a splitting called the \emph{complete splitting rel~$G_r$}, each term of which is one of the following: a subpath in $G_r$; a single edge of height greater than~$r$; an indivisible Nielsen path of height greater than $r$; or an exceptional path of height greater than~$r$. Moreover, consecutive terms of this splitting are not both subpaths of~$G_{r}$.
\item\label{ItemRelSplitEdge}
For each edge $E$ of $G$ the path $f(E)$ has a complete splitting rel~$G_r$.
\item\label{ItemRelSplitIterate}
For each $\sigma$ which is a circuit or a path with endpoints at vertices, for all sufficiently large $k$ the circuit or path $f^k_\#(\sigma)$ has a complete splitting rel $G_r$. 
\end{enumerate}  
\end{fact}

\begin{proof} Each of these follows from the corresponding result about complete splittings by amalgamating consecutive terms of height at most $r$ in the appropriate completely split circuit or path: for item~\pref{ItemRelSplitDef} see Fact~\ref{all neg} and see Definition~4.4 of \recognition\ (or Definition~\refGM{DefCompleteSplitting}); for item~\pref{ItemRelSplitEdge} see (Completely Split) in the definition of a \ct\ (Definition~4.7 of \recognition\ or Definition~\refGM{DefCT}); and for item~\pref{ItemRelSplitIterate} see Fact~\refGM{FactEvComplSplit}.
\end{proof}



The distinction made between linear and superlinear \neg\ edges in Fact~\ref{FactIsATerm} plays an important role in section~\ref{SectionFCarriesAll}. 


\begin{fact}
\label{FactIsATerm} 
Suppose that $\fG$ is a \ct\ and that every edge with height $> r$  is \noneg.    Let $E_s$ be a non-fixed edge of height $s > r$. If $\sigma$ is a path or circuit with a complete splitting rel~$G_r$, if $E_s$ occurs as a subpath of $\sigma$, and if $\tau$ is the term in the splitting of $\sigma$ that contains~$E_s$, then (following Notation~\ref{DefTwistorCT}) we have:
\begin{enumerate}
\item  \label{FactSuperlinearEdgeTerm}  If $E_s$ is superlinear then $\tau = E_s$. 
\item \label{ItemLinearEdgeTerm}   If $E_s$ is linear and is not the initial edge of a subpath of $\sigma$ of the form $E_s w_s^p \overline E_s$ then either $\tau = E_s$ or $\tau$ is  an exceptional path  $\tau = E_sw_s^q \overline E_u$  where $E_s \ne E_u$ are linear edges in the same linear family.
\end{enumerate}
\end{fact}

\begin{proof}  By inspecting the possibilities for the term $\tau$ we need only prove the following:
\begin{itemize}
\item If $E_s$ is super-linear then $E_s$ is not a subpath of any Nielsen path~$\mu$~(c.f.~Fact~\refGM{FactNoSuperlinearNielsen}).
\item If $E_s$ is linear and is a subpath of a Nielsen path $\mu$ then $E_s$ is the initial edge of a subpath of $\mu$ of the form $E_s w_s^p \overline E_s$.
\end{itemize}
Supposing that $E_s$ is a subpath of a Nielsen path $\mu$, we proceed by induction on $i=\height(\mu)$. Without loss of generality we may replace $\mu$ by the term in its complete splitting that contains $E_s$, from which it follows that $\mu$ is indivisible, because $E_s$ is not a fixed edge. It follows by (\neg\ Nielsen Paths) in the definition of a \ct\ (\recognition\ Definition~4.7, or Definition~\refGM{DefCT}) that $\mu =  E_i w_i^p \overline E_i$ for some linear edge $E_i$ of height $i \ge s$. The case $i=s$ is clear. If $i > s$ then $E_s$ is a subpath of the Nielsen path~$w_i$, and the induction hypothesis completes the proof. 
\end{proof}

\subsection{Fixed subgroup systems}
\label{SectionFixedSubgroups}

The basic definitions and facts regarding principal automorphisms are reviewed in Section~\refGM{SectionPrincipalRotationless}, in particular Definition~\refGM{DefPrinicipalAndRotationless}. Recall in particular the notation $P(\phi) \subset \Aut(F_n)$ for the set of principal automorphisms representing a rotationless $\phi \in \Out(F_n)$. The \emph{fixed subgroup system} $\Fix(\phi)$ is defined to be the set of conjugacy classes of nontrivial subgroups of the form $\Fix(\Phi)$ for $\Phi \in P(\phi)$. Since there are only finitely many isogredience classes in the set $P(\phi)$, it follows that $\Fix(\phi)$ is a finite set, as required for a subgroup system (see Section~\refGM{SectionSSAndFFS}).

\begin{fact} \label{FactItsOwnNormalizer}  The fixed subgroup $\Fix(\Psi)$ of any $\Psi \in \Aut(F_n)$  is its own normalizer.
\end{fact}
 
\begin{proof}  If $d \in F_v$ normalizes $\Fix(\Phi)$ then by Fact~\refGM{FactFiniteRankSubgroup} we have $d^k \in \Fix(\Phi)$ for some integer~$k \ge 1$, but $(\Phi(d))^k = \Phi(d^k) = d^k$, and by uniqueness of $k^{\text{th}}$ roots in $F_v$ we have $\Phi(d)=d$, and so $d \in \Fix(\Phi)$. 
\end{proof}

\begin{lemma}
\label{LemmaFixedIsPeriodic}
For any $\phi,\sigma \in \Out(F_n)$ and any representatives $\Phi,\Sigma \in \Aut(F_n)$, respectively, we have:
\begin{enumerate}
\item\label{ItemFixPtsNatural}
(\cite{FeighnHandel:abelian} Lemma 2.6) \, \, $\Fix(\wh{\Sigma \Phi \Sigma^\inv}) = \wh\Sigma(\Fix(\wh\Phi))$ and $\Fix_N(\wh{\Sigma\Phi\Sigma^\inv}) = \wh\Sigma(\Fix_N(\wh\Phi))$. 
\item\label{ItemFixSbgpNatural} $\Fix(\Sigma\Phi\Sigma^\inv) = \Sigma \Fix(\Phi) \Sigma^\inv$
\end{enumerate}
If in addition $\phi$ is rotationless and $\sigma$ commutes with $\phi$ then
\begin{enumeratecontinue}
\item \label{ItemConjugatePrincipal}
If $\Phi \in P(\phi)$ then $\Sigma \Phi \Sigma^\inv \in P(\phi)$.
\item \label{ItemFixPsiPermuted}
$\sigma$ permutes the elements of the subgroup system $\Fix(\phi)$.
\end{enumeratecontinue}
\end{lemma}

\begin{proof} Item~\pref{ItemFixSbgpNatural} is obvious, \pref{ItemConjugatePrincipal} follows from definition of principal automorphisms using~\pref{ItemFixPtsNatural}, \pref{ItemFixSbgpNatural}, and the fact that $\Sigma\Phi\Sigma^\inv$ is a representative of~$\phi$, and~\pref{ItemFixPsiPermuted} follows from~\pref{ItemFixSbgpNatural} and~\pref{ItemConjugatePrincipal}.
\end{proof}

\paragraph{Conjugacy separation of fixed subgroups.} Define an equivalence relation on subgroups of $F_n$, where $A,B \subgroup F_n$ are \emph{conjugacy inseparable} if the set of $F_n$-conjugacy classes of elements of $A$ equals the set of $F_n$-conjugacy classes of elements of $B$. 

\begin{lemma}\label{LemmaInseparablyFixed}
For any rotationless $\phi \in \Out(F_n)$ and any $\Phi,\Phi' \in P(\phi)$, if $\Fix(\Phi),\Fix(\Phi')$ are nontrivial and conjugacy inseparable then $\Fix(\Phi),\Fix(\Phi')$ are conjugate subgroups.
\end{lemma}

\begin{proof} For proving this claim we may freely replace $\Phi$ by any $i^{\vphantom{-1}}_c \Phi i_c^\inv$ in its isogredience class, with the effect of replacing $\Fix(\Phi)$ by the conjugate subgroup $i_c(\Fix(\Phi)) = c \Fix(\Phi) c^\inv$. 


%

We prove the lemma in two cases, each of which uses Fact~\ref{FactTwistor}.

Suppose both $\Fix(\Phi)$ and $\Fix(\Phi')$ have rank~$1$ then. After replacing $\Phi$ in its isogredience class, we may assume that $\Fix(\Phi) \intersect \Fix(\Phi)$ is nontrivial. Applying Fact~\ref{FactTwistor} it follows that $\Fix(\Phi)=\Fix(\Phi')=\<a\>$ where the unoriented conjugacy class $[a]_u$ is an element of the finite set of twistors $\Twist(\phi)$ and we are done.

Suppose one of $\Fix(\Phi)$ or $\Fix(\Phi')$, say the former, has rank~$\ge 2$. Since $\Twist(\phi)$ is finite, it follows that there exists $c \in \Fix(\Phi)$ whose conjugacy class is not represented by any power of any $a$ for which $[a]_u \in \Twist(\phi)$. After replacing $\Phi$ in its isogredience class we may assume $c \in \Fix(\Phi) \intersect \Fix(\Phi')$. By Fact~\ref{FactTwistor} it follows that $\Phi=\Phi'$ and so $\Fix(\Phi)=\Fix(\Phi')$ and we are again done.
\end{proof}

\section{Reducing Theorem~\ref{relKolchin} to Theorem~\ref{ThmPeriodicFreeFactor} and Propositions \ref{PropFCarriesAll}, \ref{PropNielsenPairsExist}}
\label{SectionReduction}

In this section we state several lemmas and propositions and use them to prove Theorem~\ref{relKolchin}. The proofs of Theorem~\ref{ThmPeriodicFreeFactor} and Propositions \ref{PropFCarriesAll} and \ref{PropNielsenPairsExist}, as well other results underlying their proofs, appear later in the text.  As such, this section is really an outline of the logic of the proof. The reader will note that several results are framed and proved in terms of rotationless outer automorphisms so that the full assortment of facts from relative train track theory can be applied. Other results are framed in terms of $\IA_n(\Z/3)$ outer automorphisms for which we develop some new theory in Sections~\ref{SectionIA3Elements} and~\ref{SectionOneEdge}.

We begin with:

\begin{theorem}  \label{ThmPeriodicFreeFactor}  For any $\psi \in \IA_n(\Z/3)$ and any free factor system $\F$ in $F_n$, if $\F$ is $\psi$-periodic then $\F$ is fixed by~$\psi$.
\end{theorem}

Underlying this result is Theorem~\ref{ThmPeriodicConjClass} which says that if $\psi \in \IA_n(\Z/3)$ then every $\psi$-periodic conjugacy class in $F_n$ is $\psi$-fixed. The proof of Theorem~\ref{ThmPeriodicConjClass} is in Section~\ref{SectionIA3Elements} and the proof of Theorem~\ref{ThmPeriodicFreeFactor} is in Section~\ref{SectionOneEdge}. 

\smallskip\textbf{Remark.} By applying Theorem~\ref{ThmPeriodicFreeFactor} to a free factor system $\F$ as well as to its individual components, for any $\psi \in \IA_n(\Z/3)$ the equation $\psi(\F)=\F$ is equivalent to saying that $\F$ is $\psi$-periodic, and to saying that each component of $\F$ is $\psi$-fixed.

\medskip

Next we turn to results about the asymptotic data $\Lam(\phi) = \L(\phi) \union \Eigen(\phi) \union \Twist(\phi)$ of a rotationless $\phi \in \Out(F_n)$, described in Section~\ref{SectionAsymptotic}. The following lemma shows that a $\phi$-invariant  free factor system $\F$ carries $\Lam(\phi)$ if and only if $\phi$ is represented by a \ct\ in which $\F$ is realized by a filtration element $G_r$  and all strata above $G_r$ are single edges that can be attached simultaneously to~$G_r$.

\begin{lemma} \label{lem:cofinal}  Suppose that $\phi$ is rotationless and that $\F$ is a $\phi$-invariant proper free factor system. Then $\F$ carries $\Lam(\phi)$ if and only if there exists a \ct\ $\fG$ representing $\phi$ with a filtration element $G_r$ realizing $\F$ such that each stratum $H_i$ with $i > r$ is an \neg-edge $E_i$ satisfying either $f(E_i) = E_i$ or $f(E_i) = E_i \cdot u_i$ for some non-trivial closed path $u_i$ in $G_r$. 
\end{lemma}

\begin{proof}   Assume that $\F$  carries $\Lam(\phi)$. By \recognition\ Theorem~4.28 (or see Theorem~\refGM{TheoremCTExistence}) there exists a \ct\ $\fG$ representing $\phi$ in which $\F$ is realized by a filtration element~$G_r$. For any such \ct, since $\F$ carries $\L(\phi)$ it follows by Facts~\ref{all neg} and~\refGM{FactNEGEdgeImage} that each stratum $H_i$ above $G_r$ is \neg, and so either $f(E_i)=E_i$ or $f(E_i)=E_i \cdot u_i$ where the closed path $u_i$ is known only to be in~$G_{i-1}$. 

If $E_i$ is linear then by Fact~\ref{FactTwistor} the circuit formed by $u_i$ represents an element of $\Twist(\phi)$ and so is carried by $\F$, and since $G_r$ realizes $\F$ it follows that $u_i \subset G_r$.  If $E_i$ is non-linear then $E$ generates an eigenray of $f$ in $G$ of the form $R = E \cdot u_i \cdot f_\#(u_i) \cdot f_\#^2(u_i) \cdot \ldots$. By Fact~\ref{LemmaEigenrayDefs} the ray $R$ is a realization in $G$ of some element of $\Eigen(\phi)$, and since $\Eigen(\phi)$ is carried by $\F = [\pi_1 G_r]$ it follows by Lemma~\ref{LemmaEigenrayCarried} that $u_i \subset G_r$. This  completes the proof of  the only if direction.
 
Choose a \ct\ $\fG$ representing $\phi$ with a filtration element $G_r$ realizing~$\F$ such that the strata above $G_r$ satisfy the properties of the lemma. In particular every \eg\ stratum is in $G_r$ and so $\F=[\pi_1 G_r]$ carries $\L(\phi)$. If $E$ is any nonfixed \neg\ edge with height $>r$ and if $f(E) =E\cdot  u$ for some non-trivial closed path $u$ then by hypothesis both the circuit $u$ and the ray $u \cdot f_\#(u) \cdot f_\#^2(u) \cdot \ldots$ are contained in $G_r$: the former proves that $\F$ carries $\Twist(\phi)$ and the latter proves that $\F$ carries $\eigen(\phi)$.
\end{proof}

\subparagraph{Remark.} The proof shows that one may replace the existential quantifier with a universal quantifier to get yet another equivalent statement, namely: ``for all \cts\ $f \from G \to G$ representing $\phi$ with a filtration element $G_r$ realizing $\F$, each stratum $H_i$ with $i>r$ is \ldots''.

\begin{definition} \label{DefIrreducibleRelF}
Suppose that $\F$ is a proper free factor system. A subgroup $\h \subgroup \IA_n(\Z/3)$ that leaves $\F$ invariant is said to be \emph{irreducible rel $\F$} if there does not exist a proper free factor system that properly contains~$\F$ and that is invariant under every element of $\h$ (equivalently, by Theorem~\ref{ThmPeriodicFreeFactor}, ``that is periodic under every element of~$\h$''). 
\end{definition}

\begin{prop}\label{PropFCarriesAll} 
Suppose that $\h \subgroup \IA_n(\Z/3)$ is finitely generated, $\F$ is a proper $\h$-invariant free factor system, $\h \subset \PGF$, and $\h$ is irreducible rel $\F$. Then $\F$ carries $\Lam(\phi)$ for each rotationless $\phi \in \h$.
\end{prop}

Underlying this proposition (as well as Proposition~\ref{PropNielsenPairsExist} to follow) are the concepts of ``Limit trees'' developed in Section~\ref{SectionLimitTrees}. The proof of Proposition \ref{PropFCarriesAll} is given in section~\ref{SectionFCarriesAll}.

\medskip

We next recall from Definition~5.11 of \BookTwo\ the concept of a Nielsen pair of an outer automorphism. Given a simplicial $F_n$-tree $T$ with trivial edge stabilizers recall from Section~\ref{SectionFnTrees} that $\F(T)$ denotes the free factor system consisting of conjugacy classes of nontrivial vertex stabilizers.

  
\begin{definition} \label{DefNielsenPairs}  Consider $\phi \in \Out(F_n)$ and a free factor system~$\F$
each of whose components is fixed by~$\phi$ (for example $\phi \in \UPGF$). Let $\wt\V = \{V \suchthat [V] \in \F\}$ be the set of free factors whose conjugacy classes are components of~$\F$ and let~$\wt\V^{(2)}$ be the set of unordered pairs $(V,W)$ of distinct elements of $\wt\V$ (by an ``unordered pair of distinct elements'' we simply mean a two-element subset, and we abuse notation by writing $(V,W)$ instead of $\{V,W\}$). The action of $F_n$ on $\wt\V$ by inner automorphism determines a diagonal action of $F_n$ on $\wt\V^{(2)}$. The quotient set of this action is denoted $\V^{(2)}$ and the image of $(V,W)$ in $\V^{(2)}$ is denoted $[[V,W]]$. Each automorphism $\Phi$ representing $\phi$ induces a permutation of the elements of $\wt\V^{(2)}$ that respects the decomposition into $F_n$-orbits, and so there is an induced action of $\phi$ on $\V^{(2)}$. We say that $(V,W) \in \wt\V^{(2)}$ is a \emph{Nielsen pair for $\phi$} if $[[V,W]]$ is $\phi$-invariant. 

This definition can be formulated equivalently in any simplicial $F_n$-tree $T$ with trivial edge stabilizers for which $\F(T)=\F$. The vertices of $T$ with nontrivial stabilizers correspond bijectively to $\wt\V$, and the unoriented paths between such vertices correspond bijectively to~$\wt\V^{(2)}$. Then $(V,W) \in \wt\V^{(2)}$ is a Nielsen pair for $\phi$ if and only if for some (any) such $F_n$-tree~$T$ and for some (any) automorphism $\Phi$ representing $\phi$, the unoriented path in $T$ connecting the vertices corresponding to $\Phi(V)$ and $\Phi(W)$ is a translate, by some element of $F_n$, of the unoriented path $\ti \gamma \subset T$ connecting the vertices corresponding to $V$ and $W$.

Given a subgroup $\h \subgroup \Out(F_n)$ such that each component of~$\F$ is fixed by each~$\phi \in \h$, if the unordered pair $(V,W)$ is a Nielsen pair for each $\phi \in \h$ then we say that $(V,W)$ is \emph{a Nielsen pair for $\h$ associated to~$\F$}.
\end{definition}

In Definition~\ref{DefNielsenPairs} and in Definition~\ref{DefNPSS} to follow, the concepts and the notations being defined depend implicitly on~$\F$, but this dependence is suppressed in the notation because $\F$ is constant in applications. 
 
\begin{definition}\label{DefNPSS} Continuing with the notation of Definition~\ref{DefNielsenPairs}, given $(V,W) \in \wt\V^{(2)}$ the \emph{subgroup system $S(V,W)$ associated to~$\F$ and $(V,W)$} is defined as follows. Both $[V]$ and $[W]$ are elements of $\F$. If $[V] \ne [W]$ then $S(V,W)$ is obtained from $\F$ by removing $[V]$ and $[W]$ and replacing them with $[\<V,W\>]$.  If $[V] = [W]$ then $W=V^a$ for some $a \in F_n$ and  $S(V,W)$ is obtained from $\F$ by removing $[V]$ and replacing it with $[\langle V,a \rangle]$. In general, $S(V,W)$ is not a free factor system (but see Lemma~\ref{LemmaNielsenPairRelevance}). 
\end{definition} 

In the next proposition we shift attention away from a general finitely generated subgroup of $\IA_n(\Z/3)$ and towards a subgroup generated by finitely many rotationless elements.

\begin{prop}\label{PropNielsenPairsExist} For any subgroup $\k \subgroup \IA_n(\Z/3)$ generated by a finite number of rotationless elements and for any proper $\k$-invariant free factor system~$\F$, if $\F$ carries $\Lam(\phi)$ for each rotationless $\phi \in \k$ then there exists a Nielsen pair $(V,W)$ for $\k$ associated to~$\F$. Moreover, following the notation of Definition~\ref{DefNielsenPairs}, one may choose the tree $T$ and Nielsen pair $(V,W)$ so that $\ti \gamma$ is an edge of $T$.
\end{prop}

The proof of Proposition~\ref{PropNielsenPairsExist} is given in section~\ref{SectionNielsenPairsExist}. The moreover part of the proposition is motivated by the following lemma.


\begin{lemma}  \label{LemmaNielsenPairRelevance} 
For any subgroup $\k \subgroup \IA_n(\Z/3)$ and any proper $\k$-invariant free factor system~$\F$, if $(V,W)$ is a Nielsen pair for $\k$ associated to~$\F$ then the subgroup system $S(V,W)$ is \hbox{$\k$-invariant}. Moreover, if $T$ and $\ti \gamma$ are as in Definition~\ref{DefNielsenPairs} and if $\ti \gamma$ is an edge of $T$ then $S(V,W)$ is a free factor system that is a one-edge extension of $\F(T)=\F$.
\end{lemma}

\begin{proof}     Each $\phi \in \k$ is represented by an automorphism $\Phi$  that preserves both $V$ and $W$. The subgroup $\<V,W\>$ is therefore $\Phi$-invariant. If $W= V^a$ then $V^a = V^{\Phi(a)}$ and hence $\bar a \Phi(a)$ normalizes $V$. It follows that  $\bar a \Phi(a) \in V$ and hence that  $\Phi(a) \in \langle a,V\rangle$; in particular, $\langle a,V\rangle$ is $\Phi$-invariant. This proves that $S(V,W)$ is $\phi$-invariant for all $\phi \in \k$.

For the last sentence, let $T'$ be the tree obtained from $T$ by collapsing the orbit of the edge $\ti \gamma$.  Then $T'$ has trivial edge stabilizers and    $\F(T') =  S(V,W)$, proving that $S(V,W)$ is a free factor system, and clearly $S(V,W)$  is a one-edge extension of~$\F$. 
\end{proof}

\subparagraph{Proof of Theorem~\ref{relKolchin}.}  As in the proof of the equivalence of Theorems~\ref{relKolchin} and \ref{ThmTreeRelKolchin}, by applying the results of \cite{HandelMosher:distortion} Section~4.1 it suffices to show that $\{[F_n]\}$ is a one-edge extension of $\F_{m-1}$. Since the given increasing chain of $\h$-invariant free factor systems is maximal, $\h$ is irreducible rel~$\F_{m-1}$, in other words $\h$ is irreducible relative to the extension $\F_{m-1} \sqsubset \{[F_n]\}$. Let $\psi_1,\ldots, \psi_\gen$ be generators of $\h$, and let $\k$ be the subgroup generated by  rotationless iterates $\phi_1,\ldots,\phi_\gen$, respectively. Proposition~\ref{PropFCarriesAll} implies  that  $\F_{m-1}$ carries $\Lam(\phi)$ for each rotationless $\phi \in \h$ and hence also for each rotationless $\phi \in \k$. 
Applying Proposition~\ref{PropNielsenPairsExist} to the subgroup $\k$ we obtain a simplicial $F_n$-tree $T$ with trivial edge stabilizers such that $\F(T) = \F_{m-1}$  and a pair of vertices $v,w \in T$ with non-trivial stabilizers $V$ and $W$ such that $V,W$ is a Nielsen pair for $\k$ associated to~$\F_{m-1}$ and such that the path $\ti\gamma$ with endpoints $v,w$ is an edge of $T$. Lemma~\ref{LemmaNielsenPairRelevance} implies that $S(V,W)$ is a $\k$-invariant free factor system and is a one-edge extension of $\F_{m-1}$.  Theorem~\ref{ThmPeriodicFreeFactor} implies that $S(V,W)$ is $\psi_i$-invariant for each $i$ and hence $S(V,W)$ is $\h$-invariant. Since $\h$ is irreducible rel $\F_{m-1}$, it follows that $S(V,W)= \{[F_n]\}$.\qed

\section{Periodic conjugacy classes under some $\theta \in \IA_n(\Z/3)$}
\label{SectionIA3Elements}

This is the first of two sections devoted to results about the subgroup $\IA_n(\Z/3)$. The main result of this section is the following invariance property, which is an important component of many later arguments in this paper, in particular it is used in the proof of Theorem~\ref{ThmPeriodicFreeFactor} in next section:

\begin{theorem}\label{ThmPeriodicConjClass}  For each $\theta \in \IA_n(\Z/3)$, every $\theta$-periodic conjugacy class in $F_n$ is fixed by~$\theta$. 
\end{theorem}

Preparatory to the proof of Theorem~\ref{ThmPeriodicConjClass} we establish some simpler invariance properties of $\IA_n(\Z/3)$, each of which is already known for rotationless elements (\recognition\ Lemma~3.30).

\begin{lemma} \label{LemmaFFSComponent} 
If $\theta \in \IA_n(\Z/3)$ then:
\begin{enumerate}
\item\label{ItemIAThreeFFS}
For any $\theta$-invariant free factor system $\F= \{[F^1],\ldots, [F^k]\}$, each $[F^i]$ is fixed by~$\theta$.
\item\label{ItemIAThreeLams}
Each element of $\L(\theta)$ is fixed by $\theta$.
\end{enumerate}
\end{lemma}

\begin{proof}  To prove \pref{ItemIAThreeFFS}, the natural homomorphism $F_n \mapsto H_1(F_n;\Z/3)$ induces a natural map from conjugacy classes of free factors of $F_n$ to subspaces of the vector space $H_1(F_n;\Z/3)$, and the images of $[F^1],\ldots,[F^k]$ are pairwise distinct subspaces of the vector space $H_1(F_n;\Z/3)$, each of which is fixed by $\theta \in \IA_n(\Z/3)$. It follows that each of $[F^1],\ldots,[F^k]$ is fixed by $\theta$.

The proof of \pref{ItemIAThreeLams}, while more intricate, has a similar idea at its base. Picking $\Lambda \in \L(\theta)$ and letting $A$ be the period of $\Lambda$ under the action of~$\theta$, we prove \hbox{$A=1$}. Let $\Lambda=\Lambda^1,\ldots,\Lambda^A$ be the orbit of $\Lambda$ under $\theta$, with $\theta(\Lambda^a) = \Lambda^{a+1}$, where $a$ varies over $\Z/A\Z$. Picking any relative train track map $f \from G \to G$ representing $\theta$, there is an irreducible \eg-stratum $H_r$ that decomposes into $A$ distinct irreducible \eg-aperiodic strata $H^1_r,\ldots,H^A_r$ for the relative train track map $f^A_\# \from G \to G$, with indexing chosen so that $\Lambda^a$ is the lamination corresponding to $H^a_r$ (see \BookOne, Section~3). We have $H^{a+1}_r \subset f(H^a_r) \subset G_{r-1} \union H^{a+1}_r$. The subgraph $G_{r-1} \union H^a_r$ has a noncontractible component $G^a_r$ containing each generic leaf of $\Lambda^a$, and furthermore $H^a_r \subset G^a_r$. It follows that $H^{a+1}_r \subset f(G^a_r) \subset G^{a+1}_r$. Each of the restricted maps $f \from G^a_r \to G^{a+1}_r$ is $\pi_1$-injective, because $f$ is $\pi_1$-injective. These maps fit into a cycle of maps
$$G^1_r \xrightarrow{f} \cdots \xrightarrow{f} G^A_r \xrightarrow{f} G^1_r
$$
such that any complete trip around the cycle induces a $\pi_1$-injection $f^A_\# \restrict G^a_r \from G^a_r \to G^a_r$ and furthermore, by Scott's lemma (\BookOne\ Lemma~6.0.6), each $f^A_\# \restrict G^a_r$ is a $\pi_1$-isomorphism. It follows that each of the restricted maps $f \from G^a_r \to G^{a+1}_r$ is a $\pi_1$-isomorphism and therefore a homotopy equivalence. The subgraphs $\core(G^a_r)$ are pairwise distinct for $a \in \{1,\ldots,A\}$, because $\core(G^a_r) \intersect H_r = H^a_r \not\subset \core(G^b_r)$ if $a \ne b \in \{1,\ldots,A\}$. It follows that the subgraphs $G^a_r$ represent pairwise distinct subspaces of $H_1(F_n;\Z/3)$ which are transitively permuted by~$\theta$, but $\theta \in \IA_n(\Z/3)$ fixes each of these subspaces and so $A=1$.
\end{proof}

The rest of the section is devoted the proof of Theorem~\ref{ThmPeriodicConjClass}. We proceed by induction on the rank $n$, starting with the case $n=1$ which is obvious. We may therefore assume the induction hypothesis, that the proposition is true in all ranks $<n$, and we must prove it in rank~$n$.

For the rest of the proof we fix choices of $\theta \in \IA_n(\Z/3)$ and of a rotationless power $\phi = \theta^k$, $k \ge 1$. We shall proceed in steps, combining the hypothesis of Theorem~\ref{ThmPeriodicConjClass} and the induction hypothesis to slowly build up to the conclusions. 

\subsection{Reduction to: $F_n$ is filled by the $\theta$-periodic conjugacy classes}

\begin{IAStep}\label{IAStepNonfillingOrbits} We shall reduce to the case that the union of the $\theta$-periodic conjugacy classes fills~$F_n$. Under this reduction, it follows that for every \ct\ $f \from G \to G$ representing $\phi$, the graph $G$ is the union of the fixed edges and indivisible Nielsen paths of~$f$, and furthermore every point of $G$ is contained in a closed Nielsen path of~$f$.
\end{IAStep}

To carry out this reduction, let $\F$ be the free factor support of the $\theta$-periodic classes. Clearly $\F$ is $\theta$-invariant, and so by Fact~\ref{LemmaFFSComponent}~\pref{ItemIAThreeFFS} each component of $\F$ is fixed by~$\theta$. If $\F$ is a proper free factor system then, applying induction on rank to the restriction of $\theta$ to each component of $\F$ (see Fact~\refGM{FactMalnormalRestriction}), it follows that $\theta$ fixes each $\theta$-periodic class. We are therefore reduced to the case that $\F = \{[F_n]\}$. 

For proving the ``furthermore'' clause, given a \ct\ $f \from G \to G$ representing $\phi$, 
since $\F=\{[F_n]\}$ it follows that $G$ is the union of all circuits representing $\theta$-periodic classes, but each such circuit is fixed by $f_\#$ (\refGM{FactPNPFixed}), each $f_\#$ fixed circuit is a concatenation of fixed edges and indivisible Nielsen paths (\refGM{FactNielsenCircuit}), and so each $f_\#$ fixed circuit decomposes at some vertex as a closed Nielsen path. This completes Step~\ref{IAStepNonfillingOrbits}.

\medskip

Recall that \eg\ strata of a \ct\ are classified as \emph{geometric} and \emph{nongeometric} (\BookOne\ Definition~5.1.4, and in slightly restructured form in Definition~\refGM{DefGeometricStratum}). We will need various equivalent characterizations of geometricity and nongeometricity found in \PartOne; these are reviewed as needed.

\subsection{Each \ct\ $f \from G \to G$ representing $\phi$ is geometric/linear/fixed.}

\begin{IAStep}\label{IAStepGLF} That is, each stratum of $G$ is either \eg-geometric, \neg-linear, or a fixed edge.
\end{IAStep}
\noindent
For the proof, consider first an \eg-stratum $H_r \subset G$. If $H_r$ is nongeometric then, by Step~\ref{IAStepNonfillingOrbits}, there exists a height~$r$ indivisible Nielsen path $\rho_r$ which is not a closed path. By \BookOne\ Lemma~5.1.7 (or see Fact~\trefGM{FactEGNielsenCrossings}{ItemNongeomFFS}) there exists a proper free factor system $\F$ such that a line is carried by $\F$ if and only if its realization in $G$ is a concatenation of edges of $G \setminus H_r$ and copies of $\rho_r$. By Fact~\refGM{FactNielsenCircuit} this set of lines includes the bi-infinite iterates of all $f_\#$ fixed circuits, and so $\F$ carries all periodic conjugacy classes, contradicting Step~\ref{IAStepNonfillingOrbits}. Each \eg\ stratum must therefore be geometric. If $H_r = \{E_r\}$ is a zero stratum or an \neg-superlinear stratum then by Fact~\refGM{FactNoSuperlinearNielsen} no Nielsen path crosses $E_r$, again contradicting Step~\ref{IAStepNonfillingOrbits}. This completes Step~\ref{IAStepGLF}. 

\bigskip

For the rest of the proof of Theorem~\ref{ThmPeriodicConjClass} we consider two cases depending on whether the set $\cup\L(\theta) = \{$all leaves of all laminations in $\L(\theta)\}$ fills $F_n$. Step~\ref{IAStepGeometricDoesntFill} considers the case that this set does fill, and Steps~\ref{IAStepOneEdge} and~\ref{IAStepFinal} consider the case that it does not fill.

\subsection{The case that $\cup \L(\theta)$ fills.}

\begin{IAStep}\label{IAStepGeometricDoesntFill} If $\cup \L(\theta)$ fills then the conclusion of Theorem~\ref{ThmPeriodicConjClass} holds
\end{IAStep}

\noindent
Assuming that $\cup\L(\theta)$ fills $F_n$ we have the following as well:
\begin{itemize}
\item The top stratum $H_r$ of the \ct\ $f \from G \to G$ is \eg, and by Step~\ref{IAStepGLF} it is \eg-geometric. Applying Fact~\refGM{FactGeometricCharacterization} and Fact~\refGM{FactEGNielsenCrossings}~\prefGM{ItemEGNielsenPointInterior} there is a unique height~$r$ indivisible Nielsen path $\rho_r$, and this is a closed path with base point $p_r \in H_r - G_{r-1}$. 
\end{itemize}
We give a quick review of geometric strata and their models, referring the reader to Section~\refGM{SectionGeometric} for details. 

\smallskip \textbf{Geometric models of $H_r$ (review: Section~\refGM{SectionGeometricModelsAndStrata}).} Henceforth in Step 3 we fix a geometric model for $H_r$ as given in Definition~\refGM{DefGeomModel}. Since $H_r$ is the top stratum, a geometric model for $H_r$ is equivalent to a weak geometric model as given in Definition~\refGM{DefWeakGeomModel}. Here are some details.

For the ``static data'' of a weak geometric model of $H_r$ one is given a finite 2-complex $X$ expressed as the quotient of the graph $G_{r-1}$ and a compact surface $S$ whose boundary is nonempty and has components $\bdy S = \bdy_0 S \union\cdots\union \bdy_m S$ ($m \ge 0$), where $\bdy_0 S$ is the \emph{top boundary} and $\bdy_1 S,\ldots,\bdy_m S$ are the \emph{lower boundaries}. The quotient map $j \from G_{r-1} \disjunion S \to X$ is defined by gluing each lower boundary $\bdy_i S$ to $G_{r-1}$ via a closed, homotopically nontrivial edge path $\alpha_i \from \bdy_i S \to K$ ($1 \le i \le m$). Note that $j$ embeds the graph $G_{r-1}$ and circle $\bdy_0 S$ as disjoint subcomplexes of~$X$; the union of these two subcomplexes is referred to in \refGM{DefComplSubgraph} as the \emph{complementary subgraph} $K \subset X$ of the geometric model. One is also given an embedding $G=G_r \inject X$ which extends the embedding $G_{r-1} \inject X$, and a deformation retraction $d \from X \to G$. It is required that $G \intersect \bdy_0 S$ be a single point $p_r$ and that the restriction of $d$ to $\bdy_0 S$ with base point $p_r$ is a parameterization of a closed indivisible Nielsen path~$\rho_r$ for $f$; from this requirement it follows that $d \restrict \bdy_0 \from \bdy_0 \to G$ is an immersion, and so $d \restrict K \from K \to G$ is an immersion. As a consequence of these conditions one obtains a further static conditions saying that the interior of $H_r$ in $G$ equals $H_r - G_{r-1} = H_r \intersect (X - G_{r-1}) = H_r \intersect (\interior(S) \union \{p_r\})$. 

For the ``dynamic data'' of a weak geometric model of $H_r$ one is given a homotopy equivalence $h \from X \to X$ and a homeomorphism $\Psi \from S \to S$ with pseudo-Anosov mapping class, subject to the requirement that the composed maps $d \composed h, \, f \composed d \from X \to G_r$ are homotopic, and that the composed maps $j \composed \Psi, \, h \composed j \from S \to X$ are homotopic.

We shall assume that orientations have been chosen for each of the boundary components $\bdy_i S$, and when the surface $S$ is orientable we specify that these boundary orientations are induced by a chosen orientation of~$S$. Each $\bdy_i S$ therefore determines a well-defined conjugacy class in $F_n$ denoted $[\bdy_i S]$, the inverse of which is denoted $[\bdy_i S]^\inv$. Note that $[\bdy_i S] \ne [\bdy_i S]^\inv$ because no element in $F_n$ is conjugate to its inverse. We denote $[\bdy_i S]^\pm = \{[\bdy_i S], [\bdy_i S]^\inv\}$; and we denote $[\bdy S]^\pm = \union_{i=0}^m [\bdy_i S]^\pm$, called the set of \emph{peripheral conjugacy classes} of $X$. 

The immersion $d \restrict K \from K \to G$ is $\pi_1$-injective on each component, and the images of these injections define a subgroup system of $F_n$ denoted $[\pi_1 K]$. 

We will cite various results of Sections~\refGM{SectionGeometric} and~\refGM{SectionVertexGroups} regarding properties of the complementary subgraph $K$ and the immersion $d \from K \to G$. To start with, like any graph immersion, $d \from K \to G$ is $\pi_1$-injective on each component of $K$; the conjugacy classes of the images of these injections define a subgroup system of $F_n$ denoted $[\pi_1 K]$. By Lemma~\refGM{LemmaLImmersed}, for each noncontractible component of $K$ the corresponding subgroup is malnormal, and for distinct components the two subgroups are ``mutually malnormal'' in that any conjugates of those two subgroups intersect trivially. We describe this by saying that $[\pi_1 K]$ is a malnormal subgroup system with one component for each noncontractible component of $K$. 

\smallskip

The following lemma, besides its immediate application here in the geometric case, will be applied also in Section~\ref{SectionOneEdge} in both the geometric and nongeometric cases. The proof is an application of results from \BookOne\ Section~6 including the Weak Attraction Theorem, and the results from Section~\refGM{SectionVertexGroups} on vertex group systems.

\begin{lemma}\label{LemmaANAReview}
Consider any $\theta \in \IA_n(\Z/3)$, 
any rotationless power $\phi = \theta^k$, and any \ct\ $f \from G \to G$ representing $\phi$ whose top stratum $H_r$ is \eg. If $H_r$ is not geometric let $K=G_{r-1}$, whereas if $H_r$ is geometric let $K$ be the complementary subgraph of any geometric model for~$H_r$, and in either case let $[\pi_1 K]$ denote the corresponding subgroup system. The following hold:
\begin{enumerate}
\item\label{ItemANAPeriodic}
If $c$ is a $\theta$-periodic conjugacy class then $c$ is carried by $[\pi_1 K]$. If $H_r$ is nongeometric it follows that $c$ is carried by $G_{r-1}$. If $H_r$ is geometric it follows that either $c$ is carried by $G_{r-1}$ or $c$ is an iterate of $[\bdy_0 S]$ or $[\bdy_0 S]^\inv$.
\item\label{ItemANAKFixed}
The action of $\theta$ on subgroup systems fixes $[\pi_1 K]$.
\end{enumerate}
\end{lemma}

\begin{proof} Let $\Lambda_r \in \L(\phi)$ be the attracting lamination corresponding to $H_r$. If $H_r$ is geometric let $\rho_r$ be the closed, indivisible Nielsen path of height~$r$. Applying Lemma~\ref{LemmaFFSComponent}~\pref{ItemIAThreeLams} we have $\theta(\Lambda_r)=\Lambda_r$.

The proofs of both items~\pref{ItemANAPeriodic} and~\pref{ItemANAKFixed} use the Weak Attraction Theorem, \BookOne\ Theorem~6.0.1, and Remark~6.0.2 following that theorem, which together imply that a conjugacy class $c$ in $F_n$ is not weakly attracted to $\Lambda_r$ under iteration by $\phi$ if and only if the circuit in $G$ representing $c$ is carried by $G_{r-1}$ or (in the geometric case) is an iterate of $\rho$ or $\bar\rho$; by construction of $K$ this is equivalent to saying that $c$ is carried by $[\pi_1 K]$.

Item~\pref{ItemANAPeriodic} follows immediately, noting that each $\theta$-periodic conjugacy class is fixed by $\phi$ and so is not weakly attracted to $\Lambda_r$.

To prove item~\pref{ItemANAKFixed}, we use the fact that the subgroup system $[\pi_1 K]$ is a vertex group system as defined in Section~\refGM{SectionVertexGroupSystems}: if $H_r$ is geometric this follows by applying Proposition~\refGM{PropGeomVertGrSys}; and if $H_r$ is nongeometric then $[\pi_1 K] = [\pi_1 G_{r-1}]$ is a free factor system which is a special case of a vertex group system. By Lemma~\refGM{LemmaVSElliptics}, a vertex group system is characterized by the conjugacy classes that it carries. It therefore suffices to observe that the vertex group systems $[\pi_1 K]$ and $\theta^\inv[\pi_1 K]$ carry the same conjugacy classes: a conjugacy class $c$ is carried by $[\pi_1 K]$ if and only if $c$ is not weakly attracted to $\Lambda^+_\phi$ under iteration of $\phi$, if and only if $\theta(c)$ is not weakly attracted to $\theta(\Lambda_r)=\Lambda_r$ under iteration by $\theta\phi\theta^\inv=\phi$, if and only if $\theta(c)$ is carried by $[\pi_1 K]$, if and only if $c$ is carried by $\theta^\inv[\pi_1 K]$.

\end{proof}

\subparagraph{Free boundary circles (review: end of Section~\refGM{SectionGeomModelComplement}).} Define $\bdy_i S$ to be a \emph{free boundary circle} of $X$ if there exists an open collar neighborhood $U \subset S$ of $\bdy_i S$ such that the map $j \from S \to X$ restricts to an embedding of $U$ onto an open subset of $X$. Each free boundary circle is identified homeomorphically by $j$ with its image in $X$. The top boundary $\bdy_0 S$ is a free boundary circle (this uses that $H_r$ is the top stratum). Each free boundary circle is a component of~$K$, and so any lower boundary $\bdy_i S$ which is a free boundary circle is a component of $G_{r-1}$.

Fix a subsurface $S' \subset S$ called the \emph{free subsurface} which is characterized up to ambient isotopy by saying that $S-S'$ is the union of a pairwise disjoint collection of open collar neighborhoods of the nonfree boundary circles of $S$, so $S'$ is identified homeomorphically by~$j$ with its image in~$X$. Note that $S'$ contains each free boundary circle, and that the inclusion $S' \inject S$ is homotopic to a homeomorphism relative to the free boundary circles. 

Let $\bdy^\free X \subset X$ be the union of free boundary circles, each of which may be regarded simultaneously as a component of $\bdy S$ and of $\bdy S'$, amongst which is included $\bdy_0 S$. After reordering the components of $\bdy S$ by a permutation of $\{0,\ldots,m\}$ that fixes~$0$, there exists $l \in 0,\ldots,m$ such that
$$\bdy^\free X = \bigcup_{0 \le i \le l} \bdy_i S = \bigcup_{0 \le i \le l} \bdy_i S' 
$$
and we define the \emph{nonfree boundary circles} to be 
$$\bdy^\nonfree X = \bdy S' - \bdy^\free S = \bigcup_{l < i \le m} \bdy_i S'
$$
Let $[\bdy^\free X]^\pm, [\bdy^\nonfree X]^\pm \subset [\bdy S]^\pm$ be the corresponding sets of conjugacy classes. 

Free boundary circles are addressed in Lemma~\refGM{LemmaFreeDefRetr}. The hypothesis of that lemma is that $\theta$ preserves $[\pi_1 K]$ which we know to be true by Lemma~\ref{LemmaANAReview}~\pref{ItemANAKFixed}. From the conclusion of that lemma we have:

\begin{fact}\label{FactFreePreserved}
There is a homotopy equivalence $\Theta \from X \to X$ representing $\theta$ such that $\Theta$ preserves the $K$, the free subsurface $S'$, and its complement $X \setminus S'$, the restriction $\Theta \restrict K$ is a self-homotopy equivalence, and the restriction $\Theta \restrict S'$ is a self-homeomorphism. In particular, $\Theta$ preserves the images in $X$ of the free boundary circles of~$S$.
\end{fact}

We now turn to the proof of Step~\ref{IAStepGeometricDoesntFill}. Consider a $\theta$-periodic conjugacy class $c$. If the $\theta$-orbit of $c$ does not include an iterate of $[\bdy_0 S]^\pm$ then by Lemma~\ref{LemmaANAReview}~\pref{ItemANAPeriodic} the entire $\theta$-orbit of $c$ is carried by the proper free factor system $[\pi_1 G_{r-1}]$, and so $\theta$ fixes $c$ by Step~\ref{IAStepNonfillingOrbits}. If the $\theta$-orbit of $c$ does include an iterate of $[\bdy_0 S]^\pm$ then by Fact~\ref{FactFreePreserved} the $\theta$-orbit of $c$ is contained in the set $[\bdy^\free X]^\pm$. All that remains is therefore to prove that $\Theta$ preserves each free boundary circle of $S'$ and preserves the orientation on that circle. 

We claim that one of the following holds:
\begin{description}
\item[Case (i):] The map $[\bdy^f S']^\pm \to H_1(X;\Z/3)$ is injective.
\item[Case (ii):] $S'$ is orientable and $l=0$.
\end{description}
For the proof of (and later application of) this claim, consider the subcomplex $N = X \setminus S'$ which deformation retracts to $G_{r-1}$, and consider the decomposition $X = S' \union N$ where $S' \intersect N = \bdy^\nonfree S' = \bdy_{l+1} S' \union \cdots \union \bdy_m S'$. The Mayer-Vietoris sequence of this decompositition, an exact sequence of vector spaces over the field $\Z/3$, has the form
$$H_1(\bdy^\nonfree S') \mapsto H_1(S') \oplus H_1(G_{r-1}) \mapsto H_1(X) \mapsto \wt H_0(\bdy^\nonfree S') \mapsto \wt H_0(N) \oplus \wt H_0(S')=0
$$
where the coefficient field $\Z/3$ is understood in the notation. The kernel of the inclusion induced homomorphism $H_1(S') \to H_1(X)$ is therefore contained in the span $V' \subset H_1(S')$ of the homology classes of $\bdy_{l+1} S',\ldots,\bdy_m S'$. 

If $S'$ is nonorientable then the homology classes of boundary components $\bdy_0 S',\ldots,\bdy_m S'$ are linearly independent in $H_1(S')$, and after modding out by $V'$ it follows that the homology classes of $\bdy_0 S',\ldots,\bdy_l S'$ remain linearly independent in $H_1(X)$, and Case~(i) follows.

If $S'$ is orientable then the kernel of the inclusion induced homomorphism 
$$(\Z/3)^{m+1} \approx H_1(\bdy_0 S') \oplus\cdots\oplus H_1(\bdy_m S') = H_1(\bdy S') \to H_1(S')
$$
is the 1-dimensional vector subspace $W$ generated by the homology class of the cycle $\bdy_0 S' + \cdots +\bdy_m S'$. If $\bdy_0 S',\ldots,\bdy_l S'$ are linearly independent in $H_1(X)$ then Case~(i) follows. 

Suppose that $\bdy_0 S',\ldots,\bdy_l S'$ are not linearly independent in $H_1(X)$. Let $V \subset H_1(\bdy S')$ be the subspace spanned of the homology classes of $\bdy_{l+1}S',\ldots,\bdy_m S'$, so the kernel of $H_1(\bdy S') \to H_1(X)$ is $W+V$. Consider any coefficient sequence $a_0,\ldots,a_l \in \Z/3$, not all zero, such that $\sum_{i=0}^l a_i \, \bdy_i S'$ is trivial in $H_1(X)$. It follows that there exists an element of $V$ of the form $\sum_{i=l+1}^m b_i \, \bdy_k S'$, where $b_{l+1},\ldots,b_m \in \Z/3$, and there exists $c \in \Z/3$, such that in $H_1(\bdy S')$ we have the following equation:
$$a_0 \, \bdy_0 S' + \cdots + a_l \, \bdy_l S' = c(\bdy_0 S' + \cdots + \bdy_l S') + b_{l+1} \, \bdy_{l+1} S' + \cdots + b_m \, \bdy_m S'
$$
But this implies that $a_0=\ldots=a_l = c$ and $b_{l+1}=\ldots=b_{m}=-c$ in $\Z/3$. We have therefore shown that the \emph{only} linear relation amongst $\bdy_0 S',\ldots,\bdy_l S'$ in $H_1(X)$ is that their sum can equal zero. If $l \ge 1$ it follows that Case~(i) holds; if $l=0$ then Case~(ii) holds. This completes the proof of the claim.

Using the claim we now prove that $\Theta$ preserves each free boundary circle and its orientation, breaking the proof into the Cases (i) and (ii) described above.

\textbf{Case~(i):} The set $[\bdy^\free S']^\pm$ is invariant by $\Theta$, and the map from this set to $H_1(X;\Z/3)$ is $\Theta$-equivariant and injective. Since $\Theta$ acts as the identity on the range, it follows that $\Theta$ acts as the identity on the domain $[\bdy^\free S']^\pm$, and we are done. 

\textbf{Case~(ii):} It follows that $\Theta$ permutes the two-element set $[\bdy_0 S']^\pm$, and this permutation is the identity if and only if $\Theta$ preserves orientation of $S'$. To prove the latter we break into subcases depending on the value of~$m$.

\textbf{Case~(ii)(a): $m=0$.} It follows that $\bdy^\nonfree S' = \emptyset$ which implies that $X=S$ deformation retracts to $S'$ by removing a collar neighborhood of $\bdy_0 S$. We therefore have identifications $F_n \approx \pi_1(X) \approx \pi_1(S')$. The homeomorphism $\Theta \from S' \to S'$ therefore acts as the identity on $H_1(S';\Z/3)$. Since $l=0$ we have a natural isomorphism $H_1(S') \approx H_1(S',\bdy S')$ and so $\Theta$ acts as the identity on both sides of this isomorphism. It follows that $\Theta$ preserves the intersection pairing $H_1(S';\Z/3) \oplus H_1(S',\bdy S';\Z/3) \to \Z/3$, which implies that $\Theta$ preserves orientation of~$S'$.

\textbf{Case~(ii)(b): $m=1$.} It follows that $\bdy^\nonfree S' = \bdy_1 S'$, and so the action of $\theta$ on conjugacy classes in $F_n$ preserves the two element set $\{[\bdy_1 S'], [\bdy_1 S']^\inv\}$. But these two elements are not conjugate in $F_n$ and so by induction $\theta$ fixes $[\bdy_1 S']$. It follows that $\Theta$ preserves the orientation of $\bdy_1 S'$ and so $\Theta$ preserves the orientation of $S'$.

\textbf{Case~(ii)(c): $m \ge 2$.} We claim that $\Theta \from S \to S$ preserves the lower boundary circle $\bdy_i S'$ for each $i=1,\ldots,m$. Once this claim is established then we are done, for then the action of $\theta$ preserves the two element set of conjugacy classes $\{[\bdy_i S'],[\bdy_i S']^\inv\}$ and the proof is completed exactly as in Case~(ii)(b).

Arguing by contradiction, suppose that $\Theta$ maps some lower boundary circle $\bdy_i S'$ to some other lower boundary circle $\bdy_j S'$. Permuting indices we may assume that $i=1$ and $j=2$ so $\Theta(\bdy_1(S'))=\bdy_2(S')$. We now use the map $H_1(X) \mapsto \wt H_0(\bdy^\nonfree S')$ from the Mayer-Vietoris sequence. This map is natural with respect to the action of $\theta$ on the domain $H_1(X)$, which is the identity, and the action of $\Theta$ on the range $\wt H_0(\bdy^\nonfree S')$, which is therefore also the identity. For each $i=1,\ldots,m$ pick $x_i \in \bdy_i S'$, and so each element of $\wt H_0(\bdy^\nonfree S')$ is represented uniquely by a 0-cycle $\sum_{i=1}^m a_i x_i$ where $a_i \in \Z/3$ and $\sum_{i=1}^m a_i=0$. Consider in particular the cycle $x_1 - x_2$. Choosing $k$ so that $\Theta(\bdy_2 S') = \bdy_k S'$, we have $\Theta[x_1-x_2]=[x_2-x_k] \ne [x_1-x_2]$, contradicting that $\Theta$ acts as the identity on $\wt H_0(\bdy^\nonfree S')$, and we are done. Tracing back through the Mayer-Vietoris theorem, the nonfixed element of $H_1(X)$ which gives the contradiction is represented by a circle in $X$ which starts from the point~$x_1$, goes through $S'$ to the point $x_2$, and then goes back through $N$ to the point $x_1$.

This completes Step~\ref{IAStepGeometricDoesntFill}.


\subsection{The case that $\cup\L(\theta)$ does not fill: a one-edge extension.}

\begin{IAStep}\label{IAStepOneEdge} If $\cup\L(\theta)$ does not fill $F_n$ then $F_n$ is a one-edge extension of some \break $\theta$-invariant free factor system that carries $\L(\theta)$.
\end{IAStep}

For the proof we note that the free factor support of $\cup\L(\theta)$ is a $\theta$-invariant proper free factor system. Let $\M$ be a \emph{maximal} $\theta$-invariant, proper free factor system that supports~$\L(\theta)$. We must show that $F_n$ is a one-edge extension of $\M$.

We need some results from \recognition\ regarding existence of relative train track maps for arbitrary outer automorphisms such that the zero strata and \neg\ strata share certain important features of \cts. In our context with $\M$ as described above,
by \recognition\ Theorem~2.19 and Lemma~2.20, there exists a relative train track representative $h \from K \to K$ of $\theta$ with strata $L_r = K_r \setminus K_{r-1}$ having the following properties:
\begin{enumerate}
\item\label{ItemNoEG}
$\M$ is represented by a core filtration element $K_s$ (see Theorem~2.19, particularly (F) and the final clause). 
Since $\M$ carries $\union\L(\theta)$ it follows that there is no \eg\ stratum of height~$>s$.
\item There is no zero stratum of height~$>s$ (see Theorem~2.19 (Z)).
\item\label{ItemPeriodicValenceOne}
For each periodic stratum $L_r$ of height $r > s$, if $L_r$ is a forest then every valence~$1$ vertex of $L_r$ is in $K_s$ (see Lemma~2.20~(1)).
\item\label{ItemTerminalIsIn}
For each \neg\ nonperiodic stratum $L_r$ of height $r>s$, the terminal endpoint of each edge of $L_r$ is in $K_s$ (see Theorem~2.19~(NEG)).
\end{enumerate}
We deduce one further property in our context:
\begin{enumeratecontinue}
\item\label{ItemEachTouchesKs}
Each edge $E$ of $K \setminus K_s$ has at least one endpoint in $K_s$. Furthermore, up to reversing orientation in each periodic stratum we may assume that the terminal endpoint of $E$ is in $K_s$.
\end{enumeratecontinue}
To prove \pref{ItemEachTouchesKs}, assuming $E$ has no endpoint in $K_s$ it follows from the previous properties that $E$ is a periodic edge. The union of periodic edges disjoint from $K_s$ is invariant under~$k$, and so there is a stratum $L_r$ of such edges with $r>s$. If $L_r$ is not a forest then $[L_r]$ is a nontrivial free factor system and $\M \union [L_r]$ is a proper free factor system invariant by $\theta$ that properly contains $\M$, contradicting maximality of $\M$. And if $L_r$ is a forest then item~\pref{ItemPeriodicValenceOne} is contradicted. This proves the first sentence of \pref{ItemEachTouchesKs}, and the second sentence follows from~\pref{ItemTerminalIsIn}.

We now break the proof of Step~\ref{IAStepOneEdge} into two cases depending on whether or not there is a unique stratum of height~$>s$.

\smallskip

\textbf{Case 1: $L_{s+1}$ is the unique stratum of height~$>s$.} Recall from Definition~\refGM{DefRelTT} that the edges of $L_{s+1}$ can be oriented and listed as $E_1,\ldots,E_I$ so that for each $i \in \Z / I\Z$ we have $h(E_i) = E_{i+1} u_i$ for some (possibly empty) path $u_i$ in $K_s$. Let $p_i, q_i$ be the initial and terminal points of $E_i$, so $h(p_i)=p_{i+1}$, and we may regard $u_i$ as a (possibly constant) path in $K_s$ from $q_{i+1}$ to $h(q_i)$. 

If $I=1$ then we are done, for in that case $K=K_s \union E_1$ and $p_1 \in K_s$, so $F_n$ is a one-edge extension of~$\M$. Henceforth in Case 1 we assume $I \ge 2$, so $E_1$ and $E_2$ exist and are distinct, and we work towards a contradiction.

Note that since $\theta$ fixes the components of the free factor system~$\M$, the map $h$ preserves each component of the graph $K_s$, and so all of the points $q_i$ and paths $u_i$ are in the same component of $K_s$, which we denote $\Gamma_q$. There are two subcases to consider: either all of the $p_i$ are in $K_s$, or none of them are. 

\smallskip

\textbf{Subcase 1:} None of the $p_i$ are in $K_s$. The period of the $h$-orbit $p_1,\ldots,p_I$ is an integer $J \ge 1$ dividing $I$, so for $1 \le i \le I$ we have $h^J(p_i) = p_{i+J} = p_i$, and furthermore if $1 \le j < J$ we have $h^j(p_i) = p_{i+j} \ne p_i$. Furthermore we must have $J < I$, otherwise each $p_i$ has valence~$1$ in the core graph $K$, an impossibility. It follows that $E_1 \ne E_{1+J}$ and that $p_1 = p_{1+J}$, and so we may concatenate to get a path $\overline E_1 * E_{1+J}$. Since $q_1,q_{1+J} \in \Gamma_q$, there is a path $\delta$ in $\Gamma_q$ from $q_{1+J}$ to $q_1$ and we obtain a circuit $\gamma = \overline E_1 * E_{1+J} * \delta$ whose straightened image may be written as
$$h_\#(\gamma) = \overline E_2 * E_{2+J} * \underbrace{[u_{1+J} * h_\#(\delta) * \bar u_1]}_{\text{in $K_s$}}
$$
Clearly $\gamma$ and $h_\#(\gamma)$ have different relative homology classes in $H_1(K,K_s;\Z/3)$. Using the natural homomorphism $H_1(K;\Z/3) \to H_1(K,K_s;\Z/3)$ it follows that $\gamma$ and $h_\#(\gamma)$ have different classes in $H_1(K;\Z/3)$, contradicting that~$\theta \in \IA_n(\Z/3)$. 

\smallskip
 
\textbf{Subcase 2:} All of the $p_i$ are in $K_s$, and as above they are in the same component which we denote $\Gamma_p$. There is a path $\epsilon$ in $\Gamma_p$ from $p_1$ to $p_2$ and a path $\delta$ in $\Gamma_q$ from $q_2$ to $q_1$ and we obtain a circuit $\gamma = \overline E_1 * \epsilon * E_2 * \delta$ with straightened image
$$h_\#(\gamma) = \overline E_2 * \underbrace{h_\#(\epsilon)}_{\textbf{in $K_s$}} * E_3 * \underbrace{[u_2 * h_\#(\delta) * \bar u_1]}_{\textbf{in $K_s$}}
$$ 
(If $I=2$ then $3=1$ and $E_3=E_1$). Again we get the contradiction that $\gamma$ and $h_\#(\gamma)$ have different homology classes in $H_1(K,K_s;\Z/3)$ and so also in $H_1(K;\Z/3)$. 

\smallskip

\textbf{Case 2: There is more than one stratum of height~$>s$.} After some preliminaries we will be able to arrange a picture quite similar to Case~1, leading to a similar conclusion.

For each stratum $L_r$ with $r>s$, it follows from \pref{ItemEachTouchesKs} that each component of $K_s \union L_r$ intersects $K_s$ nontrivially, and from \pref{ItemNoEG}--\pref{ItemTerminalIsIn} that $K_s \union L_r$ is $h$-invariant. By maximality of $\M$ it follows that $K_s \union L_r$ deformation retracts to $K_s$. Since by \pref{ItemEachTouchesKs} each edge of $L_r$ has its terminal endpoint on $K_s$, its initial endpoint has valence~$1$ in $K_s \union L_r$. Since $K$ is a core graph, there exist two strata $L_r,L_{r'}$ of heights $r \ne r' > s$ such that some edge of $L_r$ and some edge of $L_{r'}$ share a common initial vertex. The subgraph $K_s \union L_r \union L_{r'}$ therefore does not deformation retract to $K_s$, but this subgraph is $h$-invariant. By maximality of $\M$ it follows that $\{r,r'\}=\{s+1,s+2\}$ and that $K = K_s \union L_{s+1} \union L_{s+2}$. Furthermore, for $a=1,2$ the sets of initial vertices of $L_{s+a}$, each of valence~$1$ in $K_s \union L_{s+i}$ and valence~$\ge 2$ in $K$, must be bijectively identified. It follows that there is an integer $I \ge 1$ and for each $a=1,2$ an enumeration of the edges of $L_{s+a}$ as $E_{a,1},\ldots,E_{a,I}$ so that for $i = \Z / I\Z$ the edges $E_{1,i}$ and $E_{2,i}$ have the same initial vertex $p_i$, the edge $E_{a,i}$ has terminal vertex $q_{a,i} \in K_s$, and $h(E_{a,i}) = E_{a,i+1} u_{a,i}$ for some (possibly trivial) path $u_{a,i}$ in $K_s$ from $q_{a,i+1}$ to~$h(q_{a,i})$. 

If $I=1$ then we have completed Step 4. 

Assuming $I \ge 2$ we shall derive a contradiction. Since $h$ preserves each component of~$K_s$, all the points $q_{1,i}$ are in the same component $\Gamma_1$ of $K_s$, and all the points $q_{2,i}$ are in the same component $\Gamma_2$. Choose a path $\epsilon$ in $\Gamma_1$ from $q_{1,1}$ to $q_{1,2}$ and a path $\delta$ in $\Gamma_2$ from $q_{2,2}$ to $q_{2,1}$ and so we have a closed path $\gamma = \overline E_{2,1} \, E_{1,1} \, \epsilon \, \overline E_{1,2}  \, E_{2,2} \, \delta$ with straightened image
$$h_\#(\gamma) = \overline E_{2,2}  \, E_{1,2} \, \underbrace{[u_{1,1} \, h_\#(\epsilon) \, \bar u_{1,2} ]}_{\text{in $K_s$}} \, \overline E_{1,3} \, E_{2,3} \,  \underbrace{[u_{2,2} \, h_\#(\delta) \, \bar u_{2,1}]}_{\text{in $K_s$}}
$$
(where again $3=1$ if $I=2$) and as before we get the contradiction that $\gamma$ and $h_\#(\gamma)$ have different homology classes in $H_1(K;\Z/3)$. This completes step~5.

\subsection{The case that $\cup\L(\theta)$ does not fill: conclusion.}

\begin{IAStep}\label{IAStepFinal} If $\cup\L(\theta)$ does not fill $F_n$ then the conclusion of Theorem~\ref{ThmPeriodicConjClass} holds.
\end{IAStep}
\noindent
%
%
Applying Step~\ref{IAStepOneEdge}, there is a $\theta$-invariant \ffs\ $\M$ that carries $\L(\theta) = \L(\phi)$ and so that $F_n$ is a one-edge extension of $\M$.  Let $\f \from G \to G$ be a \ct\ representing $\phi$ in which $\M$ is realized by a core filtration element $G_r$, and let $G_s$ be the highest proper core filtration element. By Lemma~\ref{LemmaOneEdgeVersusLollipop}, the graph $G$ is either a one-edge extension or a lollipop extension of $G_r$. By Step~\ref{IAStepGLF} each stratum above $G_r$ is a fixed or linear edge. Combining this with (Periodic Edges) and (Linear Edges) in the definition of a \ct, either $G_s = G_r$ and $G$ is a one-edge extension of $G_r$ with $E = G \setminus G_r$, or $G_s=G_r \union C$ and $G$ is a lollipop extension of $G_r$ where $C = H_{r+1}$ is the loop edge of the lollipop, a fixed circle disjoint from $G_r$, and $E = H_{r+2} = G \setminus G_s$ is the stem of the lollipop. In either case $E = G \setminus G_s$ is a single topological arc with endpoints $y, z \in G_s$. Furthermore, we have $f(E) = \bar u E v$ where $u$, $v$ are each either trivial or a closed Nielsen path. Note that in the lollipop case we may orient $E$ and choose the notation so that $y \in C$ is the initial endpoint of $E$, in which case the Remark following Fact~\refGM{FactPrincipalVertices} implies that $u$ is trivial.
 
We choose $x \in \Fix(f) \cap E$ as follows. If both $u$ and $v$ are trivial then $E$ is a single fixed edge of $G$ that we may orient to have initial endpoint $y$, and we set~$x = y$; in this case $f(E)=E$. If exactly one of $u$ and $v$ is non-trivial then $E$ is a single linear edge of $G$ that we assign its \neg-orientation (see Definition~\refGM{DefRelTT}), we choose the notation so that $y$ is the initial vertex, and we set~$x=y$; in this case we have $f(E)=Ev$. If both $u$ and $v$ are non-trivial then, choosing either orientation of $E$, it follows that $E= \overline E_1 E_2$ is a union of two linear edges of $G$ that we assign their \neg-orientations, and we set $x$ to be their common initial vertex in the interior of $E$; in this case we have $f(E_1)=E_1u$ and $f(E_2) = E_2 v$.

Under the marking of $G$ we may identify $\pi_1(G,x) \approx F_n$ in a manner well-defined up to inner automorphism. Induced by the map $f \from (G,x) \to (G,x)$ we obtain an automorphism $\Phi \in \Aut(F_n)$ representing $\phi$ given by 
$$\Phi \from F_n \approx \pi_1(G,x) \xrightarrow{f_*} \pi_1(G,x) \approx F_n
$$
Let $A_x = \Fix(\Phi)$. The endpoints $y,z$ of $E$ are fixed by $f$ and belong to $f$-invariant components $C_y$, $C_z$ of $G_s$ respectively. Identify $\pi_1(C_y,y)$ with a free factor $B_y$  of $\pi_1(G,x)\cong F_n$ using the subpath of $E$ connecting $x$ to $y$. Then $B_y$  is $\Phi$-invariant and $\Phi \restrict B_y$ is the automorphism of $\pi_1(C_y,y)$ induced by $f$ and the path $u$ from $y$ to $f(y)=y$. Similarly identify $\pi_1(C_z,z)$ with a free factor $B_z$ of $F_n$ using the subpath of $E$ connecting $x$ to $z$, note that $B_z$ is $\Phi$-invariant, and note that $\Phi \restrict B_z$ is the automorphism of $\pi_1(C_z,z)$ induced by $f$ and the path $v$ from $z$ to $f(z)=z$.  Define $A_y \subgroup B_y$ and $A_z \subgroup B_z$ to be the fixed subgroups of $\Phi \restrict B_y $ and $  \Phi \restrict B_z$ respectively. Note that $A_y, A_z \subgroup A_x$. 

By Fact~\ref{IAStepNonfillingOrbits}, there exists a $\phi$-invariant conjugacy class 
that is not carried by $[\pi_1 G_s]$. Any such class is realized by a circuit in $G$ that crosses~$E$. Applying Fact~\ref{FactBasicNEGSplitting}, it follows that any circuit in $G$ that crosses $E$ splits into closed subpaths based at $x$, each of which must be a Nielsen path. If $G_s = G_r \cup C$ then the same is evidently true for any circuit contained in~$C$. Thus any $\phi$-invariant conjugacy class that is not carried by $[\pi_1 G_r] = \M$ is represented by some element of the subgroup $A_x$, and in particular $A_x$ is nontrivial.
     
\begin{lemma}  \label{Ax is invariant} Assuming notation as above, $[A_x]$ is $\theta$-invariant and $\theta \restrict A_x$ is well defined.
\end{lemma} 
  
\begin{proof} Fact~\refGM{FactUsuallyPrincipal} implies that $\Phi$ is principal, and combined with nontriviality of $A_x$ it follows that $[A_x] \in \Fix(\phi)$. By Fact~\ref{LemmaFixedIsPeriodic} we have $\theta[A_x] \in \Fix(\phi)$. Choose any root-free $\phi$-invariant conjugacy class $[c]$ that is not carried by $\M$. As noted above, $[c]$ is carried by~$[A_x]$. Since $\M$ is $\theta$-invariant, $\theta[c]$ is not carried by $\M$ and so is also carried by $[A_x]$. Thus $\theta[c]$ is carried by both $[A_x]$ and $\theta[A_x]$. To prove that $[A_x] = \theta[A_x]$ it suffices by Fact~\ref{FactTwistor}~\pref{ItemIntersectingFixed} to check that the unoriented conjugacy class $\theta[c]_u$ is not a twistor. This follows from the fact that $\theta[c]$ is not carried by $\M = [\pi_1 G_r]$ and the fact that $G_r$ contains the terminal endpoint of every linear edge in $G$ and so carries all twistors (by Fact~\ref{FactTwistor} and Definition~\ref{DefTwistorCT}). Finally $\theta \restrict A_x$ is well defined by Fact~\ref{FactItsOwnNormalizer} and Fact~\refGM{FactMalnormalRestriction}. 
\end{proof}
  
Conjugacy classes that are carried by $\M$ are $\theta$-invariant by the inductive hypothesis.  All other conjugacy classes are carried by $[A_x]$ so  we need only show that   $\theta \restrict [A_x]   $ is trivial.   The following criterion for triviality will be useful.
   
\begin{lemma} \label{finite order}  For any finite rank free group~$A$, if $\psi \in\Out(A)$ has finite order and if the action of $\psi$ on $H_1(A, \Z)$ is unipotent then $\psi$ is trivial.
\end{lemma}  

\begin{proof} The outer automorphism $\psi$ may be realized as a homeomorphism of an $A$-marked graph (see \cite{Vogtmann:OuterSpaceSurvey} and the references therein).  Lemma 4.47 of \BookTwo\ implies that this homeomorphism is isotopic to the identity and hence that $\psi$ is trivial.  
\end{proof}

Since $(\theta \restrict A_x)^k = \theta^k \restrict A_x = \phi \restrict A_x$ is the identity element of $\Out(A_x)$, we are reduced to showing that the induced action of $\theta$ on $H_1(A_x, \Z)$ is unipotent.  

Most of the work is done  in Lemmas~\ref{LemmaAxProduct} and~\ref{LemmaInvariantGroups}. The proofs of these two lemmas could be cast in a general setting, using the machinery of the proof of the Scott Conjecture in \cite{BestvinaHandel:tt}, specifically the graph $\Sigma$ and the map $p \from \Sigma \to G$ used in Proposition~6.3 of that paper, which gives an algorithmic description of all fixed subgroups of all automorphisms representing an outer automorphism. In our present situation, however, rather than employing that machinery, we give simpler ad hoc arguments by exploiting the fact that the edges playing the primary role in our argument are \neg\ edges and using their ``basic splitting property'' recounted above in Fact~\ref{FactBasicNEGSplitting}.
  
\begin{lemma} \label{LemmaAxProduct} 
Assuming notation as above,  $A_y$ and $A_z$ are non-trivial and  $A_x = A_y \ast A_z$.
\end{lemma}

\begin{proof} Once we prove that $A_x = A_y \ast A_z$, non-triviality of $A_y$ and $A_z$ follows from the fact that every conjugacy class represented by an element of $A_y$ or $A_z$ is carried by $\M$, but $A_x$ has a conjugacy class that is not carried by~$\M$.
 

We consider cases depending on the triviality or non-triviality of $u$ and $v$. In each case, each element of $A_x$ is represented by a closed Nielsen path $\sigma$ based at~$x$.

For the first case, assume that $u$ and $v$ are trivial, equivalently $E$ is a fixed edge.  We claim that there does not exist a Nielsen path $\mu \subset G_s$ connecting $z$ to $y$.   Suppose to the contrary that such a $\mu$ exists. Then $C_y =  C_z  = G_s$ and the closed path $E\mu$ determines an $f_\#$-invariant element $\alpha \in \pi_1(G,x)$.  Moreover, since $E \mu$ crosses $E$ exactly once, there is a basis for $\pi_1(G,x)$ consisting of $\alpha$ union some basis for $\pi_1(G_s,x)$.  But then $\{[\<\alpha\>]\} \union [\pi_1 G_s]$ defines a $\phi$-invariant proper free factor system that properly contains $[\pi_1 G_s]$, in contradiction to our choice of $G_s$ and the (Filtration) property of a \ct. This contradiction verifies the claim.   
    
Continuing with the first case and applying Fact~\ref{FactBasicNEGSplitting}, any path $\sigma$ as above can be written as a concatenation $\sigma =\sigma_1 \cdots \sigma_m$ of Nielsen subpaths~$\sigma_i$, each of which is either  $E$ or $\overline E$ or is contained in $G_s$. By the above claim, we can amalgamate terms in this decomposition to write  $\sigma$  as an alternating concatenation of closed Nielsen paths $\alpha_i \subset C_y$ based at $x$ and closed Nielsen paths $\beta_i$ based at $x$ and of the form $\beta_i = E \tau_i \overline E$ for some closed Nielsen path $\tau_i \subset C_z$ based at $z$. The $\alpha_i$'s determine elements of $A_y$ and the $\beta_i$'s determine elements of $\A_z$ so every element of $A_x$ can be written as an alternating concatenation of elements of $A_y$ and elements of $A_z$. Since $E$ is contained in the complement of $C_y$, any alternating concatenation of nontrivial Nielsen paths $\alpha_j \subset C_y$ and $\beta_j = E  \tau_j  \overline E$, with $\tau_j \subset C_z$, is immersed in $G$ and so represents a nontrivial element of~$A_x$. This completes the proof that $A_x = A_y \ast A_z$. 
    
For the second case, $u$ is trivial, $E$ is a linear edge with initial vertex $x=y$ and $f(E) =  Ev = E v'{}^d$ where $v'$ is a root free closed Nielsen path in $C_z$ and $d \ne 0$. By the basic splitting property (Fact~\ref{FactBasicNEGSplitting}) and the property (\neg\ Nielsen Paths) from the definition of a \ct, any $\sigma$ as above can be written as a concatenation of closed Nielsen paths based at $x$ each of which is either contained in $C_y$ or has the form $E{v'}^t \overline E$. Letting $a$ be the element of~$A_z$ represented by $Ev'\bar E$, we have that every element of $A_x$ can be written as an alternating concatenation of elements of $A_y$ and elements of $ \langle a \rangle$. As in the first case, the fact that $E$ is contained in the complement of  $C_y$ completes the proof that $A_x = A_y \ast \langle a \rangle $. It is an immediate consequence of the definitions that a closed path $\tau \subset C_z$ based at $z$ determines an element of $A_z$ if and only if $E\tau\overline E$ determines an element of $\pi_1(G,x)$ that is fixed by $f_\#$. Since the latter is equivalent to $E \tau \overline E$ being a Nielsen path, (\neg\ Nielsen Paths) implies that $A_z = \langle a \rangle $ and so $A_x = A_y \ast A_z$.
        
The final case is that $E  =  \overline E_1 E_2$ where $E_1$ and $E_2$ are linear edges with initial endpoint $x$ and satisfying $f(E_1) = E_1   u$ and $f(E_2) = E_2v$. Define $v'$ as in the previous case and define $u'$ similarly with respect to $u$.   Let $a_1$ and $a_2$ be the elements of $A_x$ represented by $E_1  u'\overline E_1$ and $ E_2 v' \overline E_2$ respectively.  Applying the basic splitting property (Fact~\ref{FactBasicNEGSplitting}) and the property  (\neg\ Nielsen Paths) as in the previous case and observing that $E_1 \ne E_2$, we conclude that $A_y = \langle a_1 \rangle  $, that $A_z = \langle a_2 \rangle $ and that $A_x = A_y \ast A_z$.
\end{proof}

\begin{lemma} \label{LemmaInvariantGroups} 
With notation as above, 
\begin{enumerate}
\item \label{ItemRestrictWDAndTrivial}
If $G_s = G_r$ then $[A_y],[A_z]$ are fixed by $\theta$, and the restrictions $\theta \restrict A_y \in \Out(A_y)$ and $\theta \restrict A_z \in \Out(A_z)$ are well defined and trivial.
\item \label{ItemRestrictWDAndTrivialLollipop}
If $G_s = G_r \union C$ then $A_y = \pi_1(C,y)$, $[A_z]$ is fixed by $\theta$, and the restriction $\theta \restrict A_z \in \Out(A_z)$ is well defined and trivial.
\end{enumerate}
\end{lemma}

\begin{proof}   We work first in case~\pref{ItemRestrictWDAndTrivial} where $G_s = G_r$.  We will prove the desired conclusion for~$A_y$ giving an argument that applies equally well to $A_z$, based on the fact that $[B_y], [B_z] \in [\pi_1 G_r] = \M$. 
Since $\M$ is $\theta$-invariant, Lemma~\ref{LemmaFFSComponent}~\pref{ItemIAThreeFFS} implies that $[B_y]$ is  $\theta$-invariant. Fact~\refGM{FactMalnormalRestriction} therefore implies that  $\psi  = \theta \restrict B_y \in \Out(B_y)$ is well defined and that $A_y$ is $\theta$-invariant if and only if it is $\psi$-invariant.  Moreover, Fact~\refGM{FactMalnormalRestriction}  and the inductive hypothesis imply that every conjugacy class in $A_y$ is $\psi$-invariant. If $A_y$ has rank~$1$ then $[A_y]$ is $\psi$-invariant because the conjugacy class of the generator of $A_y$ is $\psi$-invariant.   

If  $A_y $ has rank~$\ge 2$ then $ \Phi \restrict B_y \in \Aut(B_y)$ is a principal lift  of $\phi \restrict B_y \in \Out(B_y)$ by Remark~3.3 of \cite{FeighnHandel:recognition} and  so $[A_y] \in \Fix(\phi \restrict B_y)$. By Fact~\ref{LemmaFixedIsPeriodic}~\pref{ItemFixPsiPermuted} it follows that $\psi[A_y] \in \Fix(\phi \restrict B_y)$.  Since $[A_y]$ and $\psi[A_y]$ carry the same conjugacy classes, Lemma~\ref{LemmaInseparablyFixed} implies that $[A_y]$ is $\psi$-invariant.  Fact~\refGM{FactMalnormalRestriction} and Fact~\ref{FactItsOwnNormalizer}  imply that the $\psi \restrict A_y = (\theta \restrict B_y) \restrict A_y$ is well defined and hence that $\theta \restrict A_y$ is well defined.

If $a_1, a_2 \in A_y$ are root-free elements that are conjugate in $F_n$ but not conjugate in $A_y$ then $a_2 = i_c(a_1)$ for some $c \not \in A_y$.  Since $B_y$ is malnormal, $c \in B_y$ and so $\Phi(c) \ne c$.  The automorphism $\Phi' =i_c \Phi i_c^{-1} = i_{c\Phi(c^{-1})}\Phi \ne \Phi$ represents $\phi$ and fixes $a_2$.    Since $a_2$ is fixed by distinct automorphisms representing $\phi$, $[a_2]_u$ is a twistor of $\phi$ (by Definition~\ref{DefTwistorInvariant}).  As there are only finitely many twistors (by Fact~\ref{FactTwistor}), and since each conjugacy class in $F_n$ is represented by only finitely many $A$-conjugacy classes (because $A$ is realized by an immersion of a finite graph $K$  into a marked graph $G$ and $K$ has only finitely many circuits of a given length), it follows that there is a cofinite set $S$ of root-free conjugacy classes in $A_y$ that represent distinct conjugacy classes in $F_n$.  Since $\theta$ preserves the $F_n$-conjugacy classes of elements of $A_y$, it preserves the $A_y$-conjugacy classes of elements of $S$. Since $\rank(A_y) \ge 2$ one may choose a free basis for $A_y$ so that their conjugacy classes miss any given finite set, and in particular to be contained in $S$; it follows that  $\theta \restrict A_y$ acts trivially on $H_1(A, \Z)$. By construction, $\theta \restrict A$ has finite order $k$ so Lemma~\ref{finite order} implies that  $\theta \restrict A$ is trivial.  This completes the proof of~\pref{ItemRestrictWDAndTrivial}.

Consider now case~\pref{ItemRestrictWDAndTrivialLollipop} where we have a disjoint union $G_s = G_r \cup C$ and $C$ is a circle.  In this case, $x=y \in C$ and $E$ is either a single fixed edge or a single linear edge with initial vertex  $y \in C$.
By construction, $A_y$ is the subgroup of $\pi_1(C,y)$ consisting of elements that are represented by a Nielsen path based at $y$, and this is evidently all of $\pi_1(C,y)$. Since $[B_z] \in [\pi_1 G_r] = \M$, we ,ay use the same arguments as in case~\pref{ItemRestrictWDAndTrivial} to conclude that the restriction $\theta \restrict A_z$ is well-defined and trivial. This completes the proof of~\pref{ItemRestrictWDAndTrivialLollipop}.
\end{proof}

\textbf{Remark.} In case~\ref{ItemRestrictWDAndTrivialLollipop} of the preceding proof, since $[B_y] \not\in[\pi_1 G_r] = \M$ we \emph{may not} use the previous arguments to conclude that $\theta$ preserves~$[A_y]$.

\smallskip

We can now complete the proof of Step 5.  If $G_s= G_r$ then $A_x = A_y \ast A_z$ by Lemma~\ref{LemmaAxProduct} and the restriction of $\theta$ to both $A_y$ and $A_z$ is trivial by Lemma~\ref{LemmaInvariantGroups}.    It follows that the action induced by $\theta$ on $H_1(A_x, \Z)$ is trivial so Lemma~\ref{finite order} completes the proof.

If $G_s= G_r \cup C$ then Lemma~\ref{LemmaAxProduct} and  Lemma~\ref{LemmaInvariantGroups} imply that  $A_x =   \langle c \rangle \ast A_z$ where $F_n = \langle c \rangle \ast B_z$.   Let $M$ be the matrix determined by the action   induced by $\theta$ on $H_1(A_x,\Z)$ with respect to a basis whose first element is  $c$ and whose remaining elements are contained in $A_z$.  By Lemma~\ref{finite order} it suffices to show that $M$ is lower triangular.  The square submatrix associated to the elements of $A_z$ is the identity so it suffices to show that the entry in the first row and column is $1$.  

Since $\M = [\pi_1 G_r]$ is $\theta$-invariant, there exists a homotopy equivalence $h \from G \to G$ representing $\theta$ such that $h(G_r)\subset G_r$. Since $G_r \union E$ deformation retracts to $G_r$ we may further assume that $h(G_r \union E) \subset G_r \union E$. Let $\gamma$ be a closed path based at $x$ going once around $C$ and representing the generator $c$ of $A_y$. We may assume that $h(\gamma)$ is a path, and applying Corollary~3.2.2 of \BookOne\ we have $h(\gamma) = w_1\gamma^{\pm} w_2$ for some paths $w_1,w_2 \subset  G_r \cup E$. Since $\theta$ induces the trivial action on $H_1(F_n,\Z_3)$ we must have the plus sign, $h(\gamma) = w_1\gamma w_2$. Straightening $h(\gamma)$ as a circuit we have $h_\#(\gamma) = \gamma w'$. Since $\theta$ preserves $[A_x]$, the circuit $h_\#(\gamma)$ splits into Nielsen paths for $f$, one of which is $\gamma$ and so represents $c$ and the rest of which is $w'$ which is a closed Nielsen path continued in $G_r \cup E$ and so determines an element of $A_z$.   This proves that the homology class in $H_1(A_x,\Z)$ determined by $\theta[c]$ is the sum of the homology class of $[c]$ and a homology class in $A_z$ as desired.  

\smallskip

This completes Step~\ref{IAStepFinal} and so completes the proof of Theorem~\ref{ThmPeriodicConjClass}. \qed



\section{Periodic free factors under some $\theta \in \IA_n(\Z/3)$: Proof of Theorem~\ref{ThmPeriodicFreeFactor}}  
\label{SectionOneEdge}

In this section we prove:

\medskip
\noindent{\bf Theorem~\ref{ThmPeriodicFreeFactor} }\ \  \emph{ If $\psi \in \IA_n(\Z/3)$ then every $\psi$-periodic free factor system  is $\psi$-invariant.  
}

\subsection{Reduction to one-edge extensions: Proposition~\ref{PropSpecialCase}} 

We prove Theorem~\ref{ThmPeriodicFreeFactor} by first reducing it to a special case, and then for the remainder of the section we prove the proposition in that case.

\begin{prop} \label{PropSpecialCase} Suppose that $\psi \in\IA_n(\Z/3)$, that $\phi$ is a rotationless iterate of $\psi$, and that $\F' \sqsubset \F$ are free factor systems, such that the following hold:
\begin{enumerate}
\item \label{ItemOneEdge} $\F' \sqsubset \F$ is a one-edge extension.
\item \label{ItemSomeInvariance}
$\F'$ is $\psi$-invariant and $\F$ is $\phi$-invariant and proper.
\item \label{ItemCarriesAsym} $\F'$ carries $\Lam(\phi)$.
\item \label{ItemReducedExtension} No $\phi$-invariant free factor system is properly contained between $\F'$ and~$\F$.
\item \label{ItemNoNewInvConjClass} Each $\phi$-invariant conjugacy class that is carried by $\F$ is carried by $\F'$.
\end{enumerate}
Then $\F$ is $\psi$-invariant.
\end{prop}

\subparagraph{Remark.} We note that \pref{ItemReducedExtension} is not a formal consequence of \pref{ItemOneEdge}, because by Lemma~\ref{LemmaOneEdgeVersusLollipop} it is possible that $\F' \sqsubset \F$ are realized in a filtered marked graph $G$ by filtration elements $G_{i-2} \subset G_i$ where $H_{i-1} = G_{i-1} \setminus G_{i-2}$ is a loop disjoint from $G_{i-2}$, and $G_i$ is a lollipop extension of $G_{i-2}$.

%

\begin{lemma} 
\label{LemmaFreeFactorReduction}
Proposition~\ref{PropSpecialCase} implies Theorem~\ref{ThmPeriodicFreeFactor}.
\end{lemma}

\begin{proof} Theorem~\ref{ThmPeriodicFreeFactor} is obvious for $n=1$. Assuming the induction hypothesis that Theorem~\ref{ThmPeriodicFreeFactor} holds in all ranks~$<n$, we must prove it for $\psi \in \IA_n(\Z/3)$. Let $\phi = \psi^k$ be a rotationless power,~$k \ge 1$. Any component of a $\psi$-periodic free factor system is $\phi$-periodic and so is $\phi$-invariant, by \recognition\ Lemma~3.30. 

Fixing henceforth a proper, nontrivial, free factor $F \subgroup F_n$ such that $[F]$ is $\phi$-invariant, it remains to show that $[F]$ is $\psi$-invariant.

\smallskip
\textbf{Case 1: $\L(\phi) \cup \{[F]\}$ fills~$F_n$.}  
Choose a \ct\ $\fG$ with top filtration element $G=G_r$ in which $[F]$ is realized by a connected core filtration element~$G_s$, $s \le r-1$, and so $[F] \sqsubset [\pi_1 G_{r-1}]$. Since $\L(\phi) \union \{[F]\}$ fills~$F_n$ it follows that the top stratum $H_r$ is \eg.

We claim that $\psi [\pi_1 G_{r-1}] = [\pi_1 G_{r-1}]$. Using this claim it follows by Lemma~\ref{LemmaFFSComponent}~\pref{ItemIAThreeFFS} that each component of $[\pi_1 G_{r-1}]$ is $\psi$-invariant. The entire $\psi$-orbit of $[F]$ is therefore supported in a single component of $[\pi_1 G_{r-1}]$. It follows that the free factor support of the $\psi$-orbit of $F$ is a proper free factor system~$\F'$. Furthermore $\F'$ has a single component because each of its components is $\psi$-invariant (by Lemma~\ref{LemmaFFSComponent}~\pref{ItemIAThreeFFS}) and so the component of $\F'$ carrying $[F]$ carries the entire $\psi$-orbit of $[F]$. Restricting $\psi$ to $\F'$ and applying induction on rank it follows that $[F]$ is $\psi$-invariant, finishing the proof of Lemma~\ref{LemmaFreeFactorReduction} in Case~1.

The proof of the claim in the case that the stratum $H_r$ is nongeometric is an immediate application of Lemma~\ref{LemmaANAReview}. Consider the case that $H_r$ is geometric, and so there exists a closed, indivisible, height~$r$ periodic Nielsen path $\rho$ (\BookOne\ Theorem~5.1.5 item~eg-(iii), or see Fact~\refGM{FactGeometricCharacterization}). Let $[\<\rho\>]$ denote the conjugacy class in $F_n$ of the infinite cyclic subgroup generated by $\rho$. By applying Lemma~\ref{LemmaANAReview} we conclude that $\psi$ preserves the vertex group system $[\pi_1 G_{r-1}] \union \{[\<\rho\>]\}$. The conjugacy class in $F_n$ represented by the circuit $\rho$ is $\phi$-invariant and therefore $\psi$-periodic, and so by applying Theorem~\ref{ThmPeriodicConjClass} it is fixed by $\psi$. It follows that $[\<\rho\>]$ is fixed by $\psi$, and so $\psi[\pi_1 G_{r-1}] = [\pi_1 G_{r-1}]$.

\smallskip
\textbf{Case 2: $\L(\phi) \cup \{[F]\}$ does not fill~$F_n$.} 
Choose a \ct\ $f \from G \to G$ with $G=G_r$ in which $\L(\phi) \cup \{[F]\}$ is carried by a proper filtration element, and so in particular is carried by $G_{r-1}$. The top stratum $H_r$ must therefore be \neg. 

Following Section~\ref{SectionAsymptotic}, in addition to $\L(\phi)$ the other two subsets comprising the asymptotic data $\Asym(\phi) = \L(\phi) \union \Eigen(\phi) \union \Twist(\phi)$ are also carried by $[\pi_1 G_{r-1}]$: by Fact~\ref{FactTwistor} each element of $\Twist(\phi)$ is the twistor corresponding to some \neg-linear edge $H_s = E_s$, and by Definition~\ref{DefTwistorCT} that twistor is carried by $[\pi_1 G_{s-1}] \sqsubset [\pi_1 G_{r-1}]$; also, by Fact~\ref{LemmaEigenrayDefs} each element of $\Eigen(\phi)$ is the eigenray generated by some \neg-superlinear edge $H_s=E_s$, and by Definition~\ref{DefEigenrayCT} that eigenray is again carried by $[\pi_1 G_{s-1}] \sqsubset [\pi_1 G_{r-1}]$. Note that $\Asym(\phi) \union \{[F]\}$ has proper free factor support, being carried by $[\pi_1 G_{r-1}]$. It follows that the free factor support of $\Asym(\phi)$ is a proper free factor system that we shall denote~$\F_0$. 
Since $\psi$ commutes with $\phi$, the set $\Asym(\phi)$ is $\psi$-invariant, and so $\F_0$ is $\psi$-invariant, by Fact~\trefGM{FactFFSPolyglot}{ItemFFSPolyNatural}). Each component of $\F_0$ is therefore $\psi$-invariant (Fact~\ref{LemmaFFSComponent}~\pref{ItemIAThreeFFS}), and so by induction on rank we may assume that $\F_0$ does not carry $[F]$.   

Extend $\F_0$ to a $\phi$-invariant filtration by free factor systems $\F_0 \sqsubset \F_1 \sqsubset \ldots \sqsubset \F_K = \{[F_n]\}$ such that the free factor support of $\Asym(\phi) \cup \{[F]\}$ equals $\F_j$ for some $1 \le j <K$, and such that this filtration is maximal in that there is no $\phi$-invariant free factor system strictly between $\F_{i-1}$ and $\F_i$ for any $i=1,\ldots,K$. We shall show that $\F_i$ is $\psi$-invariant for each $0 \le i \le j$; as in Case 1, applying induction on rank to the component of $\F_j$ supporting $[F]$ then finishes the proof of Lemma~\ref{LemmaFreeFactorReduction} in Case~2.

Knowing already that $\F_0$ is $\psi$-invariant, assume by induction that $\F_{i-1}$ is $\psi$-invariant for $1 \le i \le j$. 

Suppose at first that there exists a $\phi$-invariant conjugacy class $[a]$ that is supported in $\F_i$ but not $\F_{i-1}$. It follows that $\F_i$ is the free factor support of $\F_{i-1}$ and $[a]$. Since $\F_{i-1}$ is $\psi$-invariant, as is $[a]$ (by Theorem~\ref{ThmPeriodicFreeFactor}), so is $\F_{i}$.

We have reduced to the case that every $\phi$-invariant conjugacy class supported by $\F_i$ is supported by $\F_{i-1}$. We want to complete the proof by applying Proposition~\ref{PropSpecialCase} with $\F' = \F_{i-1}$ and \hbox{$\F=\F_i$}. Items \pref{ItemSomeInvariance}--\pref{ItemNoNewInvConjClass} are satisfied so it suffices to show that $\F_{i-1} \sqsubset \F_{i}$ is a one edge extension. By \recognition\ Theorem~4.28 (or see Theorem~\refGM{TheoremCTExistence}), there is a \ct\ $\fG$  representing $\phi$ in which $\F_{i-1}$ and $\F_{i}$ are realized by core filtration elements $G_{k(i-1 )}$ and $G_{k(i)}$.  Facts \ref{all neg} and \refGM{FactNEGEdgeImage} imply that each stratum between $G_{k(i-1)}$ and $G_{k(i)}$ is a single edge. Each time one these edges is attached, the free factor system represented is either unchanged or changes by a one-edge extension. Letting $G_{k'}$ be the lowest filtration element  whose corresponding free factor system $[\pi_1 G_{k'}]$ properly contains $\F_{i-1}$, we have $k(i-1) < k' \le k(i)$, 
$[\pi_1 G_{k'}] \sqsubset \F_{i}$, and $[\pi_1 G_{k'}]$ is a one-edge extension of $\F_{i-1}$.  Since $[\pi_1 G_{k'}]$ can not be properly contained between $\F_{i-1}$ and $\F_{i}$, it follows that $[\pi_1 G_{k'}]=\F_{i}$ and so $\F_{i-1} \sqsubset \F_{i}$ is a one edge extension as desired.
\end{proof}

\subsection{Relative Nielsen classes and the path set $\Gamma$}
\label{SectionNielsenAndGamma}

After some preliminaries regarding relative Nielsen classes, following Definition~\ref{DefEStar} we give a brief idea of the proof of Proposition~\ref{PropSpecialCase}, which will motivate the description of a special set of $\phi$-invariant bi-infinite lines denoted $\Gamma$. We then establish a series of lemmas regarding~$\Gamma$, which lays the foundation for the proof to be carried out in the following section. 

\begin{definition}[Nielsen classes relative to subgraphs]
\label{DefGrNielsenClass}  
Given $f \from G \to G$ a homotopy equivalence of a connected graph, recall that $x,y \in \Fix(f)$ are in the same \emph{Nielsen class} if they are endpoints of a Nielsen path of $f$. Also, if $\ti f \from \wt G \to \wt G$ is a lift of $f$ and if $\ti x \in \Fix(\ti f)$ is a lift of $x \in \Fix(f)$ then $\Fix(\ti f)$ projects to exactly the Nielsen class of~$x$.    

Consider an $f$-invariant subgraph $K \subset G$.  If $x,y \in \Fix(f) \cap K$  are in the same component $K_0$ of $K$ and in the same Nielsen class of $f \restrict K_0$---equivalently, if $x,y$ are connected by a Nielsen path in $K$---then we say that $x,y$ are in the \emph{same $K$-Nielsen class}. If $\ti f, \ti x, x$ are as above and if $\wt C \subset \wt G$ is the component of the full pre-image of $K$ that contains $\ti x$ then $\Fix(\ti f \restrict \wt C) $ projects to exactly the $K$-Nielsen class of $x$. 
\end{definition}  

The following notation will be assumed for the rest of Section~\ref{SectionOneEdge}.

\begin{notn}[A \ct\ representative of $\phi$] 
\label{NotationSection8}
From the hypotheses of Proposition~\ref{PropSpecialCase}, fix $\psi \in \IA_n(\Z/3)$, a rotationless power $\phi = \psi^k$, and an inclusion of free factor systems $\F' \sqsubset \F$ satisfying hypotheses \pref{ItemOneEdge}--\pref{ItemNoNewInvConjClass}. By Theorem~\refGM{TheoremCTExistence}, Fact~\ref{all neg} and Fact~\refGM{FactNEGEdgeImage} there is a \ct\  $\fG$  representing $\phi$ such that  $\F' \sqsubset \F$ are realized by core  filtration elements $G_r \subset G_s$ respectively, and such that each stratum $H_i$ with $i > r$ is a single oriented edge $E_i$ satisfying either $f(E_i) = E_i$ or $f(E_i) = E_iu_i$ for some closed path $u_i \subset G_r$. Recal some properties of the definition of a \ct\ (Definition 4.7 of \recognition\ or see Definition~\refGM{DefCT}): the (Periodic Edges) property implies that if $E_i$ is fixed but not a loop then both endpoints of $E_i$ are contained in $G_{i-1}$; and the (Linear Edges) property implies that $u_i$ is a closed Nielsen path of height~$< i$, in which case $E_i u_i \overline E_i$ is a closed Nielsen path of height~$i$.

\medskip

By items \pref{ItemOneEdge} and~\pref{ItemReducedExtension} in the hypotheses of Proposition~\ref{PropSpecialCase}, combined with the definition of a \ct, we obtain the following case analysis for the proof of Proposition~\ref{PropSpecialCase}, with additional conclusions that shall be justified below.
\begin{description}
\item[One stratum:] $s = r+1$ and both endpoints of  $E_{r+1}$ are contained in~$G_r$. 

\emph{Additionally:} if $E_{r+1}$ is a fixed or linear edge then its endpoints are distinct and are contained in distinct $G_r$-Nielsen classes and  at least one of these endpoints is not the base point of a closed Nielsen path in $G_r$. 

 \item[Two strata:]  $s=r+2$, neither $E_{r+1}$ nor $E_{r+2}$ is fixed, and $ \overline E_{r+2} E_{r+1}$ is a topological arc with both endpoints in $G_r$. 
 
\emph{Additionally:} At least one of $E_{r+1}$, $E_{r+2}$ is non-linear.

\item[Disjoint loop stratum:]  $s = r+1$  and  $E_{r+1}$ is a fixed loop that is a component of~$G_s$. 

\emph{However:} This case ``Disjoint loop stratum'' does not occur, by item \pref{ItemNoNewInvConjClass} in the hypotheses of Proposition~\ref{PropSpecialCase}.
\end{description}
\end{notn}
 
\noindent We verify the additional conclusions of the cases ``One stratum'' and ``Two strata'' by a separate case analysis. 

\textbf{Case (i):} Suppose first that $E_{r+1}$ is a fixed edge.  If there is a  (possibly trivial) Nielsen path   $\rho \subset G_r$ connecting the terminal endpoint of $E_{r+1}$ to the initial endpoint of $E_{r+1}$, let   $\sigma = E_{r+1} \rho$.  If there are closed Nielsen paths $\tau_1$ and $\tau_2$ based at the initial and terminal endpoints of $ E_{r+1}$ respectively, let $\sigma =  \tau_1  E_{r+1} \tau _2 \bar  E_{r+1}$.  In both cases $\sigma$ is a closed Nielsen path that is contained in $G_s$ but not $G_r$.  This contradicts item~\pref{ItemNoNewInvConjClass} of the hypotheses of Proposition~\ref{PropSpecialCase}. The additional conclusions of ``One stratum'' follow when $E_{r+1}$ is fixed.  
 
 \textbf{Case (ii):} Suppose next that $E_{r+1}$ is a linear edge.  If $\rho$ exists as in Case~(i), let  $\sigma =  E_{r+1} u_{r+1} \overline E_r \bar \rho u_{r+1} \rho$.    If $\tau_1$ exists as in Case~(i), let $\sigma =  \tau_1  E_{r+1} u_{r+1} \bar  E_{r+1}$.  The additional conclusions of ``One stratum'' follow by contradiction as in Case~(i).
 
 \textbf{Case (iii):} Suppose finally that  $s=r+2$ and that both $E_{r+1}$ and  $E_{r+2}$ are linear. In this case   $\sigma = E_{r+1}u_{r+1}\overline E_{r+1} E_{r+2}u_{r+2}\overline E_{r+2}$ provides the contradiction that proves the additional conclusions of ``Two strata''.
 
\begin{definition}[The arc $E^*$]
\label{DefEStar} 
In the case ``One stratum'' let $E^* = E_{r+1}$, and in the case ``Two strata'' let $E^* = \overline E_{r+2} E_{r+1}$. In both cases $E^*$ is a topological arc with both endpoints in $G_r$ and interior disjoint from $G_r$.  
\end{definition}
 
\smallskip
  
The idea of the proof of Proposition~\ref{PropSpecialCase} is simple. We will extend $E^*$ to a certain line $\gamma$ by adding rays in $G_r$ (Corollary~\ref{LemmaExistenceOfGamma}) and then show that $\psi_\#(\gamma)$ is also an extension of $E^*$ by rays in $G_r$.  Since $\F$ is both the free factor support of $\F' \cup \{\gamma\}$ and the free factor support of $\F' \cup \{\gamma'\} = \psi_\#(\F' \cup \{\gamma\})$,  it follows that $\F$ is $\psi$-invariant. The hard work is to prove $\psi_\#$-invariance of an appropriate and sufficiently rich class of lines~$\gamma$.

\begin{notn}[The class of lines $\Gamma$]
\label{principal ray}  Let $\Gamma$ be the set of $\phi$-invariant lines $\gamma$ such that 
\begin{enumerate}
\item \label{item:height} $\gamma$ has height greater than $r$ and both ends of $\gamma$ have height less than $ r$.
\item  \label{item:not repelling} For any lift $\ti \gamma$ there exists an automorphism $\Phi$ representing $\phi$ such that the endpoints of $\ti \gamma$ are contained in $\Fix_N(\Phi)$ (such an automorphism $\Phi$ is unique by by Lemma~\ref{uniqueness of Phi}).
\end{enumerate}
Let $\ti \gamma$ and $\Phi$ be as in \pref{item:not repelling}. Let $\ti f \from \wt G \to \wt G$ be the lift of $f$ corresponding to~$\Phi$. Let $\wt V$ be the set of endpoints of fixed edges and initial points of \neg-edges, over all edges of height~$>r$ in $\ti\gamma$.  By \pref{item:height} the set $\wt V$ is finite and non-empty. Lemma 4.1.4 of \BookOne\ implies that  the decomposition of $\ti \gamma$ obtained by subdividing at each element of $\wt V$ is  a splitting.  Since $\ti \gamma$ is $\ti f_\#$-invariant, it follows that  $\wt V \subset \Fix(\ti f)$. Let $\ti \rho$ be the (possibly trivial) subpath of $\ti \gamma$ connecting the first and  last elements of $\wt V$ and let $\wt R_-$ and $\wt R_+$ be the complementary rays in~$\ti \gamma$.  Obviously, $\ti \rho$  is a Nielsen path for $\ti f$. Its projection $\rho$ is a Nielsen path for $f$ called the \emph{connecting Nielsen path} of $\gamma$.   We say that  $\ti \gamma = \wt R_-^{-1} \ti \rho \wt R_+$  is \emph{the highest edge splitting} of~$\ti \gamma$ and that the projected splitting $\gamma = R_-^{-1} \rho R_+$ is \emph{the highest edge splitting of~$\gamma$}. Finally, we say that the maximal subpath $\sigma$ of $\gamma$ that begins and ends with edges of height greater than $r$ is the \emph{central subpath} of $\gamma$. Thus $\rho\subset\sigma$ and $\sigma$ is obtained from $\rho$ by adding at most one initial and one terminal edge. More precisely, the following three conditions are equivalent: the initial edge of $R_+$ is appended to the terminal end of $\rho$ to become the terminal edge of $\sigma$; the initial edge of $R_+$  equals $E$ (not $\overline E$) for some non-fixed edge $E$ of height greater than $r$; $R_+ \not\subset G_r$. Similarly the following are equivalent: the inverse of the initial edge of $R_-$ is appended to the initial end  of $\rho$ to become the initial edge of $\sigma$; the initial edge of $R_-$   equals $E$ (not $\overline E$) for some non-fixed edge of height greater than $r$;   $R_- \not\subset G_r$.
\end{notn}
   
\begin{lemma}  \label{uniqueness of Phi} For each $\gamma \in \Gamma$ and each lift $\ti \gamma= \wt R_-^{-1} \ti \rho \wt R_+$ the automorphism $\Phi$ satisfying Notation~\ref{principal ray}\pref{item:not repelling} is unique and  principal.   The lift $\ti f$ corresponding to $\Phi$ is the unique lift that fixes some (and hence every) element of $\wt V$.  
\end{lemma}
   
\begin{proof} A line whose endpoints are fixed by distinct automorphisms representing the same outer automorphism is fixed by an inner automorphism and so is periodic. Item  \pref{item:height} of Notation~\ref{principal ray} therefore implies that  $\Phi$ is unique.   
   
As noted in Notation~\ref{principal ray}, the lift $\ti f$ corresponding to $\Phi$ fixes every point in $\wt V$.  Since each of these points project to an endpoint of an \neg\ edge, and since each such endpoint is principal (by Fact~\refGM{FactUsuallyPrincipal}), it follows that $\Phi$ is principal.
\end{proof}
    
\begin{lemma} \label{psi invariance}     $\Gamma$ is $\psi_\#$-invariant.  
 \end{lemma}
 
  \begin{proof}   Let $\Psi$ be any automorphism representing $\psi$.   Given $\gamma \in \Gamma$,    a lift $\ti \gamma$ and $\Phi$ as in item \pref{item:not  repelling} of Notation~\ref{principal ray}, let $\gamma' = \psi_\#(\gamma)$, $\ti \gamma' = \Psi_\#(\ti \gamma)$ and $\Phi' = \Psi \Phi \Psi^{-1}$.    Since $G_r$ realizes $\F'$  and $\psi$ preserves $\F'$, $\gamma'$  satisfies item \pref{item:height} of Notation~\ref{principal ray}.  Item \pref{item:not repelling} follows from the fact that $\Fix_N(\Phi') = \Psi(\Fix_N(\Phi))$.  
 \end{proof}
 
%

\begin{lemma}[Constructing paths in $\Gamma$] \quad
\label{LemmaExistenceOfGamma} 
\begin{enumerate}
\item \label{ItemGammaFromNielsenPath}
Every non-trivial Nielsen path $\mu$ whose endpoints are contained in $G_r$ and whose first and last edges have height greater than $r$ is the central subpath of some element $\gamma \in \Gamma$.
\item \label{ItemGammaFromEStar}
The arc $E^*$ (Definition~\ref{DefEStar}) is the central subpath of some element of $\gamma \in \Gamma$.
\end{enumerate}
\end{lemma}

\begin{proof}   The endpoints $x$ and $y$ of $\mu$ are principal fixed points in $G_r$ and so are the basepoints of fixed directions $d$ and $d'$ in $G_r$.  Choose a lift $\ti \mu$ of $\mu$ and let $\ti f$ be the lift of $f$ that fixes the initial endpoint $\ti x$ and the  terminal endpoint $\ti y$ of $\ti \mu$. Lifting $d$ and $d'$ to directions based at  $\ti x$ and $\ti y$, and applying Fact~\refGM{FactSingularRay}~\prefGM{ItemFixedDirectionRay} and~\prefGM{ItemRayHeight}, there exist rays $\wt R_-$ based at $\ti x$ and $\wt R_+$ based at $\ti y$ that project into $G_r$ and that have ideal endpoints in  $\Fix_N(\hat f)$. The line $\gamma=\overline R_- \mu R_+$ satisfies the conclusions of~\pref{ItemGammaFromNielsenPath}. 

For~\pref{ItemGammaFromEStar} we may assume by~\pref{ItemGammaFromNielsenPath} that $E^*$ is not a single fixed edge.  Let $\wt E_{r+1}$ be a lift of $E_{r+1}$,  let $\ti z$ be its initial  endpoint, let  $z$ be the initial endpoint of $E_{r+1}$, and let $\ti f$ be the lift of $f$ that fixes $\ti z$. Following Notation~\ref{NotationSection8}, consider first the subcase ``One stratum'', so $E^*=E_{r+1}$ is a nonfixed \neg\ edge. Let $d$ be a fixed direction in $G_r$ that is based at $z$ and let $d'$ be the initial direction of $E_{r+1}$.  Applying Fact~\refGM{FactSingularRay}~\prefGM{ItemFixedDirectionRay} and~\prefGM{ItemRayHeight} to the lift of $d$ based at $\ti z$ we obtain $\wt R_-$ based at $\ti z$ that projects into $G_r$. Applying Fact~\refGM{FactSingularRay}~\prefGM{ItemRayHeight} to the lift of $d'$ based at $\ti z$ we obtain a ray $\wt R_+$ with initial edge $\wt E_{r+1}$ that projects into $G_{r+1}$ and has subray $\wt R_+ \setminus E_{r+1}$ projecting into~$G_r$, and so $\gamma = \overline R_- R_+$ satisfies the conclusions of~\pref{ItemGammaFromEStar}. Consider next the subcase ``Two strata'', so $E^* = E_{r+1} \union E_{r+2}$. By (Periodic Edges) in the definition of a \ct, neither $E_{r+1}$ nor $E_{r+2}$ is a fixed edge. We may take $d'$ and $\wt R_+$ as in the previous subcase. Taking $d$ to be the initial direction of $E_{r+2}$, we may apply Fact~\refGM{FactSingularRay}~\prefGM{ItemRayHeight} to obtain a ray $\wt R_-$ with initial edge $E_{r+2}$ such that $\wt R_-$ projects into $G=G_{r+2}$ and has subray $\wt R_- \setminus E_{r+2}$ projecting into $G_r$, and again $\gamma = \overline R_- R_+$ satisfies the conclusions of~\pref{ItemGammaFromEStar}. 
\end{proof}

Recall $G_r$-Nielsen classes as defined in Definition~\ref{DefGrNielsenClass}.

\begin{definition} \label{DefInducedNielsen} Let $\cN$ be the set of $G_r$-Nielsen classes of principal vertices of $f$ contained in $G_r$. There is a $\psi$-induced permutation $N \mapsto \psi N$ of $\cN$ defined as follows. Choose $x \in N$ and a lift $\ti x \in\wt G$. Let $C$ be the component of $G_r$ that contains $x$, let $\wt C$ be the component of the full pre-image of $C$ that contains $\ti x$, let $\ti f : \wt G \to \wt G$ be the principal lift of $f$ that fixes $\ti x$ and hence preserves $\wt C$ and let $\Phi \in P(\phi)$ be the automorphism corresponding to $\ti f$. Fact~\ref{LemmaFFSComponent}~\pref{ItemIAThreeFFS} implies that $\psi$ preserves $[\pi_1(C)]$. We may therefore chose an automorphism $\Psi$ representing $\psi$ that preserves $\partial \wt C$. The automorphism $\Phi' = \Psi \Phi \Psi^{-1}$ represents $\phi$, is principal by Fact~\ref{LemmaFixedIsPeriodic}~\pref{ItemConjugatePrincipal} and preserves $\partial \wt C$. The corresponding lift $\ti f'$ of $f$ preserves $\wt C$ and is principal. Define $\psi N$ to be the $G_r$-Nielsen class that is the image in $C$ of $\Fix(\ti f') \cap \wt C$. 
\end{definition}

\begin{lemma}   With notation as in Definition~\ref{DefInducedNielsen}, $\psi N$ is a well defined $G_r$-Nielsen class whose elements are principal points. If $\Fix(\Phi)$ is non-trivial then $\psi N  = N$.
\end{lemma}

\begin{proof}   It is straightforward to check that $\psi N$ is independent of the choices of $x, \ti x$ and $\Psi$. For example, if $x$ is replaced by another element $y \in N$ then there is a Nielsen path $\rho \subset C$ connecting $x$ to $y$.  Lift this to $\ti \rho \subset \wt C$ connecting $\ti x$ to a lift $\ti y \in \Fix(\ti f)$.   This shows that $\Phi$ is unchanged and that   $\Psi$ need not be changed so $\ti f'$ need not be changed.   We leave the verification that $\psi N$ does not depend on the choices of $\ti x$ and $\Psi$ to the reader.

It remains only to show that $\Fix(\ti f') \cap \wt C \ne \emptyset$ and that if $\Fix(\Phi)$ is non-trivial then $\psi N  = N$. Suppose first that $\Fix(\Phi) \ne \emptyset$ and choose a non-trivial element $a \in F_n$. By Theorem~\ref{ThmPeriodicConjClass}, we may choose $\Psi$ so that $\Psi$ fixes~$a$. Thus $\Phi' = \Psi \Phi \Psi^{-1}$ fixes $a$ and $\Phi' = i_a^k\Phi$ for some $k$. Choosing $m > 0$ so that $\psi^m = \phi$, we have  $\Psi^m \Phi \Psi^{-m} = i_a^{mk} \Phi$. On the other hand, $\Psi^m = i_a^l \Phi$ for some $l$ because $\Psi^m  $ fixes $a$ and represents $\Phi$. This implies that $\Psi^m$ commutes with $\Phi$ and hence that $mk = 0$. Thus $k = 0$ and $\Phi' = \Phi$.  Equivalently $\ti f' = \ti f$ and we are done.   

Suppose now that $\Fix(\Phi)$ is trivial. By \recognition\ Lemma~2.20~(4), there exist directions $d_1 \ne d_2$ in $C$ based at $x$ that are $f$-periodic, and since $x$ is a principal vertex of $f$ it follows by (Rotationless) in the definition of a \ct\ that $d_1,d_2$ are each fixed by~$f$. They lift to directions $\ti d_1 \ne \ti d_2$ in $\wt C$ based at $\ti x$ that are fixed by $\ti f$. Since there does not exist a nontrivial $\gamma \in F_n$ such that $\Fix(\wh\Phi) \intersect \Fix(\hat\gamma) \ne \emptyset$ (by Fact~\refGM{FactFPBasics}), we may apply Fact~\refGM{FactSingularRay}, with the conclusion that $\ti d_1,\ti d_2$ are the initial directions of rays in $\wt C$ that are fixed by $\ti f_\#$ and that end at points of $\Fix_N(\wh\Phi) \cap\partial\wt C$. Since $\Fix(\Phi)$ is trivial, these two endpoints are in $\Fix_+(\wh\Phi)$ by Fact~\refGM{LemmaFixPhiFacts}.
It follows that $\Fix_N(\wh\Phi') \cap\partial\wt C = \wh\Psi(\Fix_N(\wh\Phi) \cap\partial\wt C)$ contains two points of $\Fix_+(\wh\Phi)$, and applying \recognition\ Lemma~3.16 it follows in turn that $\Fix(\ti f') \cap \wt C \ne \emptyset$. Since $\ti f'$ is principal, Fact~\refGM{FactPrincipalLift} implies that points of $\Fix(\ti f') \cap \wt C$ are principal.
 \end{proof}

\begin{lemma}\label{fixed class}  Suppose that the $G_r$-Nielsen class $N \in \cN$ determined by $x$ is $\psi$-invariant and that $\ti f, \ti x, \wt C$ and $\Phi$ are as in Definition~\ref{DefInducedNielsen}.  Then there exists an automorphism $\Psi_0$ representing $\psi$ that commutes with $\Phi$ and such that $\wh\Psi_0(\partial \wt C) =\partial \wt C$.  \end{lemma}

\begin{proof}  Begin with any automorphism $\Psi$ representing $\psi$ that preserves $\partial \wt C$.  Following 
Definition~\ref{DefInducedNielsen} let $\Phi' = \Psi \Phi \Psi^\inv$ and let $f'$ be the lift of $f$ corresponding to $\Phi'$.  
Since $N$ is $\psi$-invariant, there is a covering translation $L_a$ such that $L_a(\Fix(\ti f' \restrict \wt C)) = \Fix(\ti f \restrict \wt C)$.  It follows that $i_a(\partial \wt C) = \partial \wt C$ and  that $L_a \ti f ' L_a^\inv $ fixes each element of $\Fix(\ti f \restrict \wt C)$ and so equals $\ti f$.  Translating back to the language of automorphisms, $i_a \Phi' i_a^\inv  =\Phi$.     Thus $\Psi_0 := \ i_a\Psi$ satisfies  $\Phi = \Psi_0 \Phi \Psi_0^{-1}$ and $\wh\Psi_0(\partial \wt C) = \partial \wt C$.
\end{proof}

Recall from Lemma~\ref{psi invariance} that $\Gamma$ is $\psi$-invariant.

\begin{lemma}\label{action of psi}    Suppose that   $\gamma = R_-^{-1} \rho R_+$ is the highest edge splitting of $\gamma \in \Gamma$ (see Notation~\ref{principal ray}) and that $\gamma' = {R_-'}^{-1}\rho' R_+$   is the   highest edge splitting of  $\gamma' = \psi_\#(\gamma)$.   If $R_-  \subset G_r$ and the basepoint $x$ of $R_-$ is contained in the $G_r$-Nielsen class $N$ then   $R'_-  \subset G_r$ and the basepoint  of $R'_-$ is contained in  the $G_r$-Nielsen class $ \psi N$.  The analogous statement holds for $R_+$.
\end{lemma}

\begin{proof}  Assume  notation as in Definition~\ref{DefInducedNielsen}.   The terminal endpoint of $\wt R'_-$ is contained in $\partial \wt C$ and $\Fix(\ti f' \restrict \wt C) \ne \emptyset$.      It suffices to prove that $\wt R'_- \subset \wt C$ for its initial endpoint will then be in $\Fix(\ti f' \restrict \wt C)$ and so project to $\psi N$.  If this fails then the initial edge of $\wt R'_-$ is $\wt E_i$ for some non-fixed edge $E_i$ above $G_r$.  This contradicts Lemma~\ref{LemmaPointsAtInfinity} and the fact that $\Fix(\ti f' \restrict \wt C) \ne \emptyset$ and  so completes the proof.  
  \end{proof}

In what follows $\<\alpha, E\>$ is the algebraic intersection number of an edge  path $\alpha$ with an edge $E$.  

\begin{lemma}\label{first intersection number}   Suppose that $E$ is an   oriented edge  of $G$ with height greater than $r$.
\begin{enumerate}
\item $\<\rho, E\>= 0$  for every \iNp\  $\rho$ of height greater than $r$   
\item If $E$ is non-fixed then $\<\rho, E\>= 0$  for every Nielsen path  $\rho$.  
\item If $\tau \subset G$ is a circuit then  $\<\tau ,E\> = \<\psi_\#(\tau),E\>$.
\end{enumerate}
\end{lemma} 

\begin{proof} If $\rho$ is an \iNp\ with height $s > r$ then $\rho=E_s \beta \overline E_s$ for some path $\beta_s \subset G_r$, by (\neg\ Nielsen Paths) in the definition of a \ct\ (\recognition\ Definition~4.7, or Definition~\refGM{DefCT}). This proves (1). Every Nielsen path decomposes as a concatenation of \iNp s of height greater than $r$, \iNp s of height at most $r$ and fixed edges.  Since the second and third types do not cross a non-fixed edge of height greater than $r$, (2) follows from (1). Item (3) follows from the assumption that $\psi$ acts trivially on $H_1(F_n,\Z_3)$.   
\end{proof}


\begin{lemma}  \label{intersection number} Suppose that $E$ has height greater than $r$, that $\gamma \in \Gamma$ and that $\sigma, \sigma'$ and $\sigma''$ are the central subpaths of $\gamma, \gamma'= \psi_\#(\gamma)$ and  $\gamma''= \psi_\#(\gamma')$ respectively.  Then
\begin{enumerate}
\item \label{item:intersection equality} If $E$ is non-fixed or if both ends of $\sigma$ are contained in the same component of $G_r$ then $ \<\sigma,E\> =\<\sigma',E\> = \<\sigma'',E\>  \mod 3$.
\item \label{item:intersection sum}If  the endpoints of $\sigma$ are contained in distinct components of $G_r$ then $2 \<\sigma',E\>  =\<\sigma,E\>+\<\sigma'',E\>\mod 3$.
\end{enumerate}
\end{lemma}

\begin{proof}  Choose a homotopy equivalence $g \from G \to G$ that represents $\psi$ and that preserves each component of $G_r$. Suppose that $\ell$ is a line in $G$ that has height greater than $r$ and whose ends each have height $< r$.  Let $\alpha$ be the maximal subpath of $\ell$ that begins and ends with an edge of height greater than $r$. We claim that the maximal subpath of $g_\#(\ell)$ that begins and ends with an edge of height  greater than $r$  is the same as the maximal subpath of $g_\#(\alpha)$ that begins and ends with an edge of height  greater than $r$. To see this, note that the tightening of  $g(\ell)$ to $g_\#(\ell')$ can be done in three stages. First tighten $g(\alpha)$ to $g_\#(\alpha)$, then tighten the $g$-images of the rays that are complementary to $\alpha$ and then cancel edges at the two juncture points. Since no edges with height $>r$ are cancelled during the second and third phases of this tightening, the claim follows.

Let $\beta$ be the maximal subpath of $g_\#(\sigma)$ that begins and ends with an edge of height greater than $r$. Applying the above claim  with $\ell = \gamma$ we conclude that $\sigma' = \beta$. If both endpoints of $\sigma$ are contained in the same component  $C$ of $G_r$, connect them by a path in $G_r$ to form a circuit $\tau$. A straightforward modification of the  above tightening argument shows that $\psi_\#(\tau)$ is the concatenation of $\beta$ and a path in $G_r$. Thus
$$
\< \sigma,E\> =  \<\tau,E\>  =\< \psi_\#(\tau),E\> = \< \beta,E\>  = \< \sigma',E\>  \qquad  \text{ mod(3)}
$$
where the second equality follows from Lemma~\ref{first intersection number}.   This same argument, with $\gamma$ replaced by $\gamma'$, yields $\<\sigma',E\> = \<\sigma'',E\>$ mod(3). This completes the proof of the second part of (1).

 Let $\beta'$  be the maximal subpath of $g_\#(\sigma')$ that begins and ends with an edge of height $> r$. If the endpoints of $\sigma$ are    contained in distinct components  of $G_r$ let $\tau$ be a circuit that factors as $\tau =\sigma \mu \bar \sigma'\nu$ where $\mu,  \nu \subset G_r$.  This is always possible because each component of $G_r$ is non-contractible and $g$-invariant.  Arguing as above,  
$$\<\sigma,E\> -\<  \sigma' ,E\>=\< \tau,E\> =\<   \psi_\#(\tau),E\>  =   \< \beta,E\> -  \< \beta' ,E\> = \<\sigma',E\> -\<\sigma'',E\> \qquad \text{ mod(3)}
$$
and hence 
$$ 2 \< \sigma',E\>= \< \sigma,E\> +\<  \sigma'',E\>  \text{ mod(3)}$$
which verifies (2).

Finally, suppose that $E$ is non-fixed and that the components $C_1$ and $C_2$ of $G_r$ that contain the initial and terminal endpoints of $\sigma$ respectively are distinct.   As a first subcase, assume that the   terminal endpoint of $E$ is not in $C_1$.  Our second claim is that for any line $\ell  \in \Gamma$ whose central subpath $\alpha$ has initial point in $C_1$ and terminal point in $C_2$, the quantity $\<\alpha,E\>$ takes values in $\{0,1\}$.  From the description of central subpaths in Notation~\ref{principal ray}, it follows that $\alpha$  is obtained from a Nielsen path $\rho$ by adding at most one initial edge and  one terminal edge.  Moreover, an added 
initial edge cannot be $E$ and an added terminal edge cannot be $\overline E$.   Since the initial endpoint of $\sigma$ is in $C_1$ and the terminal endpoint of $E_1$ is not in $C_1$, an added initial 
edge cannot be $\overline E$.   If $E$ is added to $\rho$ as a terminal edge then $\<\alpha, E\>  =  \<\rho, E\>+1$; otherwise, $\<\alpha, E\> =  \<\rho, E\>$.  The second claim therefore follows from the fact (Lemma~\ref{first intersection number})   that $\<\rho,E\> = 0$.  Since 
$\psi_\#$ preserves $\Gamma, C_1$, and $C_2$, the second claim applies to each of the 
lines $\gamma, \gamma'$ and $\gamma''$, and so $\<\sigma,E\>, \<\sigma',E\>$ and 
$\<\sigma'',E\>$ each have value in $\{0,1\}$. Combined with \pref{item:intersection sum}  it follows that $\<\sigma,E\> = \<\sigma',E\> = \<\sigma'',E\>$ 
as desired.
  
   In the second and final subcase,  the terminal endpoint of $E$ is not in $C_2$, in which case a 
completely symmetric argument holds with $\{0,1\}$ replaced by $\{0,-1\}$.
\end{proof}

\begin{cor} \label{higher edge case}  Let  $\gamma = R^{-1}_- \rho R_+$ be the highest edge splitting of $\gamma \in \Gamma$ and let  $\gamma' = {R'_- }^{-1}\rho' R'_+$ be the highest edge splitting of $\gamma' = \psi_\#(\gamma)$. If  $R_+$ begins with a (necessarily non-fixed) edge $E$  of height greater than $r$ and $R_-$ does not begin with $E$  then $R_+ = R_+'$.   Moreover, $\ti \gamma$ is a lift of $\gamma$ and   $\ti f$ is the unique lift such that $\ti f_\#(\ti \gamma) = \ti \gamma $ then the terminal endpoint of $R_+$ is in $\Fix_+(\hat f)$ if and only if $E$ is non-linear. The analogous result also holds for $R_-$. 
\end{cor}

\begin{proof}     Let $\sigma$ and $\sigma'$ be the central subpaths of $\gamma$ and $\gamma'$ and write   $\sigma = \mu \rho E$ where $\mu$ is either trivial or  $\overline E_1$  for a non-fixed edge $E_1$ of height greater than $r$.  By hypothesis $E_1 \ne E$.   Since $\<\rho,E\>=0$ by Lemma~\ref{first intersection number},\  $\<\sigma,E\>=1$.    Lemma~\ref{intersection number}~\pref{item:intersection equality}  implies that $\<\sigma',E\>=1$ mod(3).  Reversing the argument we conclude that $E$ is the terminal edge of $\sigma'$ and hence the initial edge of  $R_+'$.  

Let $\wt R_+ \subset \ti \gamma$ be the lift of $R_+$ into $\ti \gamma$, let $\wt E$ be the initial edge of $\wt R_+$ and let $\wt R_+'\subset \ti \gamma'$  be the lifts  of $R_+'\subset  \gamma'$ such that the  initial edge of $\wt R_+'$ is $\wt E$.
   Let $Q$ and $Q'$ be terminal endpoints of $\wt R_+$ and $\wt R_+'$ respectively.  Lemma~\ref{uniqueness of Phi} implies that $\ti f$ preserves  $\wt R_+'$   and so  $Q,Q' \in \Fix_N(\hat f)$.    To prove that $R_+ = R_+'$ it suffices  to show that $Q = Q'$.  
            
Let $\wt C$ be the component of the full pre-image of $G_r$ that contains the terminal vertex of $\wt E$ and hence contains $\wt R_+$  and $\wt R_+'$.   In particular, $Q, Q' \in  \Fix_N(\hat f) \cap\partial \wt C$.   Let $\Psi$ be an automorphism representing $\psi$ such that $\wh\Psi(Q) =Q'$. Lemma~\ref{LemmaPointsAtInfinity} implies that if $E$ is non-linear then $\Fix_N(\hat f) \cap\partial \wt C$ is a single point in $\Fix_+(\hat f)$ and that if $E$ is linear then $\Fix_N(\hat f) \cap\partial \wt C$ is the endpoint set of a non-trivial covering translation $L_a$. In the former case we are done so assuming the latter, we must show that $Q \in \Fix(\Psi)$. If this fails then  $\Psi$  interchanges the endpoints of the axis of $L_a$, which implies that $[a]$ is not $\psi$-invariant. This contradiction to Theorem~\ref{ThmPeriodicConjClass} completes the proof.    
  \end{proof}

    \begin{lemma} \label{invariant Nielsen class}   If a Nielsen class  contains more than one element of $\cN$ then each  element of $\cN$  that it contains is  $\psi$-invariant.  
\end{lemma}

\begin{proof}  Let $\E$ be the set  of fixed edges of $G$ of height $>r$ and whose endpoints, which may or may not be in $G_r$, are distinct and do not belong to the same $G_r$-Nielsen class.

Given distinct $M,N \in \cN$  in the same Nielsen class  choose a  shortest Nielsen path $\rho$ with one endpoint $x \in N$ and the other $y \in M$. Property (\neg\ Nielsen Paths) in the definition of a \ct\ (\recognition\ Definition~4.7, or Definition~\refGM{DefCT}) implies that $\rho$ begins and ends with edges in $\E$. Lemma~\ref{LemmaExistenceOfGamma} implies that $\rho$ is the central subpath of some $\gamma \in \Gamma$.  Let $\gamma = R_-^\inv \rho R^+$ be the highest edge spitting and let  $\gamma'= {R'}_-^{-1}\rho' R'_+ $ be the highest edge splitting of  $\gamma' = \psi_\#(\gamma)$. Lemma~\ref{action of psi}  implies  that $\rho'$ is the central subpath of $\gamma'$  and that the initial vertex $x'$ of $\rho'$ belongs to $\psi N$.   It suffices to show that  $x' \in N$.  

 Towards this end we construct an auxillary graph $K$.  Each vertex of $G \setminus G_r$ and each element of $\cN$ determine a vertex of $K$.  There is one edge for each element of $\E$. The edges of $K$ are attached to the vertices of $K$ in the obvious way.  Note that no edge of $K$ is a loop and that some vertices of $K$ may have valence zero.   
 
If a Nielsen path $\mu_k$ is either contained in $G_r$ or has the form $E_j \alpha \overline E_j$  where $E_j$ is a linear edge  above $G_r$ and $\alpha \subset G_r$ then both endpoints of $\mu_k$ correspond to the same vertex in $\Gamma$. Property (\neg\ Nielsen Paths) in the definition of a \ct\ (\recognition\ Definition~4.7, or Definition~\refGM{DefCT}) therefore implies  that each Nielsen path $\mu \subset G$  induces a path $\mu_{K} \subset K$ whose endpoints in $K$ correspond to the endpoints of $\mu$. Moreover, the algebraic crossing number of $\mu_{K}$ with the edge in $K$ corresponding to $E \in \E$ equals $\<\mu, E\>$.
 
Orient the edges of $\E$ so that the vertex $z$ of $K$ corresponding to $\psi N$  is the initial vertex of every edge in $\Gamma$ that contains it.   If $\mu_{K}$ has distinct endpoints then $S(z,\mu_{K})$, the sum   mod $3$  of the  algebraic crossing numbers of $\mu_{K}$ with the edges of $K$  incident to $z$, is $0$ if $z$ is not an endpoint of $\mu_{K}$, is $+1$ is $z$ is the initial endpoint of $\mu_{K}$ and is $-1$ if $z$ is the terminal endpoint of $\mu_{K}$.
 
If $x$ and $y$ belong to the same component of $G_r$ then (recalling from above that $\rho$ and $\rho'$ are the central subpaths of $\gamma$ and $\gamma'$) we have $\<\rho,E\> = \<\rho',E\>$ mod $3$ for all edges $E$ with height greater than $r$ by Lemma~\ref{intersection number}.    It follows that $S(z,\rho_{K}) = S(z,\rho'_{K})$.  Since $z$ is the initial endpoint of $\rho'_{K}$ it is also the initial endpoint of $\rho_{K}$ which means that $\psi N =  N$ as desired.

Suppose then that  $x$ and $y$ belong to distinct components of $G_r$.   Let $\gamma''= {R''}_-^{-1}\rho'' R''_+ $ be the highest edge splitting of  $\gamma'' = \psi^2_\#(\gamma)$.  Lemma~\ref{action of psi}  implies  that $\rho''$  is the central subpath of $\gamma''$.       Lemma~\ref{intersection number} implies that 
$2 S(z,\rho'_{K}) = S(z,\rho_{K})+ S(z,\rho''_{K})$.
 Since $z$ is the initial endpoint of $\rho'_{K}$ it is also the initial endpoint of both $\rho_{K}$  and $\rho''_{K}$ and so again  $\psi N =  N$.
 \end{proof}

We need one more lemma before proving Proposition~\ref{PropSpecialCase}.

\begin{lemma} \label{Theta}  Suppose that $\ti f$ is a principal lift of $f$, that $\Phi \in P(\phi)$ is the automorphism corresponding to $\ti f$ and that $\Psi$ is an automorphism representing $\psi$ such that $\Psi^m = \Phi$. Then the following hold.
\begin{enumerate}
\item \label{item:equal fix} $\Fix(\Psi) = \Fix(\Phi)$.
\item \label{item:fixing P} If $P \in \Fix_+(\wh\Phi)$  and the $F_n$-orbit of $P$ is fixed by $\psi$ then $P \in \Fix_N(\wh\Psi)$.  
\item \label{item:preserving tilde C}Suppose that  the $G_r$-Nielsen class $N \in \cN$ of $x$ is preserved by $\psi$, that $\ti x \in \Fix(\ti f)$ and that $\wt C$ is the component of the full pre-image of $G_r$ that contains $\ti x$.  Then $\partial \wt C$ is $\wh\Psi$-invariant. 
\end{enumerate}
\end{lemma}

\begin{proof}  Obviously, $\Fix(\Psi) \subset \Fix(\Psi^m) = \Fix(\Phi)$ so to prove \pref{item:equal fix} it suffices to prove $\Fix(\Phi) \subset \Fix(\Psi)$.  Since $\Psi^m = \Phi$ it follows that $\Psi$ preserves $\Fix(\Phi)$. Theorem~\ref{ThmPeriodicConjClass} implies that $\psi$ restricts to the trivial outer automorphism of $\Fix(\Phi)$ and hence that $\Psi$ acts on $\Fix(\Phi)$ by~$i_c$ for some $c \in\Fix(\Phi)$. If $\rank(\Fix(\Phi)) <2$ then all inner automorphisms act trivially and we are done.  Suppose then that $\rank(\Fix(\Phi)) \ge 2$. Since the action of $\Phi$ on $\Fix(\Phi)$ is given by $i_{c^m}$ it must be that $c^m$ is trivial. It follows that $c$ is trivial and that the action of $\Psi$ on $\Fix(\Phi)$ is trivial. In other words, $\Fix(\Phi) \subset \Fix(\Psi)$. This completes the proof of \pref{item:equal fix}.

Suppose next that $P$ is as in \pref{item:fixing P}. Since $\wh\Psi$  preserves $\Fix_N(\wh\Phi)$ and the $F_n$-orbit of $P$ is fixed by $\psi$, there exists $a \in F_n$ such that   $i_aP = \wh\Psi(P)  \in \Fix_N(\wh\Phi)$. Thus 
$$i_{\Phi(a)} P = i_{\Phi(a)}\wh\Phi(P)=\wh\Phi(i_a P) = i_aP
$$
and so $i_{\bar a \Phi(a)}(P) = P$.  Since $P \in \Fix_+(\wh\Phi)$, Facts~\refGM{FactFPBasics} and~\refGM{LemmaFixPhiFacts} together imply that $P$ is not fixed by any non-trivial inner automorphism. Thus  $a \in \Fix(\Phi) = \Fix(\Psi)$. By induction, 
$$\wh\Psi^k(P) = \wh\Psi^{k-1}\wh\Psi(P) = \wh\Psi^{k-1}i_aP = i_a \wh\Psi^{k-1}(P) = i_a i_a^{k-1}P = i_a^kP
$$
for all $k \ge 1$ so   
$$i_a^m P = \wh\Psi^m(P) = \wh\Phi(P) = P
$$ 
It follows that $a$ is trivial. This completes the proof of \pref{item:fixing P}.
  

It remains to prove (3). By Lemma~\ref{fixed class} there exists $\Psi_0$ representing $\psi$ that commutes with $\Phi$ and satisfies $\wh\Psi_0(\partial \wt C) = \partial \wt C$.  There exists $b \in F_n$ such that $\Psi = i_b \Psi_0$.   Since $\Phi$ commutes with both $\Psi$ and $\Psi_0$ it also commutes with $i_b$ which implies that $b \in \Fix(\Phi) = \Fix(\Psi)$ and hence that $b \in \Fix(\Psi_0)$ and $i_b$ commutes with $\Psi_0$.    Thus $$\partial \wt C =  \wh\Phi(\partial \wt C)= \wh\Psi^m(\partial \wt C) = (i_b\Psi_0)^m\partial \wt C  = i_b^m \Psi_0^m(\partial \wt C) = i_b^m \partial \wt C $$
Since an inner automorphism preserves $\partial \wt C$ if and only if its fixed points are contained in $\partial \wt C$ and since $i_b$ and $i_{b^m} = i_b^m$ have the same fixed points, 
$$\Psi(\partial \wt C) = i_b \Psi_0(\partial \wt C) = i_b(\partial \wt C) = \partial \wt C
$$ 
\end{proof}

\subsection{Proof of Proposition~\ref{PropSpecialCase}}    
Following Notation~\ref{NotationSection8} and Definition~\ref{DefEStar}, $\F' \sqsubset \F$ are realized by core filtration elements $G_r \subset G_s$ respectively;   $E^* = G_s \setminus G_r$  is a topological arc with endpoints in $G_r$ and interior disjoint from $G_r$; either $s=r+1$ and $E^* = E_{r+1}$ or $s=r+2$ and $E^* = \overline E_{r+2}E_{r+1}$. Following Definition~\ref{DefInducedNielsen}, $\cN$ is the set of $G_r$-Nielsen classes of principal vertices of $G$ contained in~$G_r$. We may assume without loss of generality that one of the following is satisfied  
\begin{description}
  \item [(a)]  $E^*=E_{r+1}$ is a  fixed or linear edge whose endpoints are  
contained in distinct  elements of $\cN$  and  whose initial endpoint   is not  the base point of a closed Nielsen path in $G_r$.    
\item [(b)]    $E^* = \overline E_{r+2}E_{r+1}$ where  $E_{r+1}$ and  $E_{r+2}$ are non-fixed and $E_{r+2}$ is non-linear.
  \item [(c)]  $E^*=E_{r+1}$ is neither fixed nor linear. 
 \end{description}  
 (Deriving this from the cases ``One stratum'' and ``Two strata'' of Notation~\ref{NotationSection8}   may require reversing the orientation on $E_{r+1}$ if it is a fixed edge in the case ``One stratum'' or interchanging $E_{r+1}$ and $E_{r+2}$ in the case ``Two strata''.)  
 
  Let $m>0$ satisfy $\psi^m = \phi$. 
 
By Corollary~\ref{LemmaExistenceOfGamma},  there exists $\gamma \in \Gamma$ with central subpath $E^*$.   We will show that  $\gamma' = \psi_\#(\gamma)$ is carried by $\F$ or equivalently that  $\gamma' \subset G_s$.   The smallest free factor system carrying $\F'$ and $\gamma$ is $\phi$-invariant and so is equal to $\F$.  It follows that   $\psi_\#\F$ is the smallest free factor system carrying $\F'$ and $\gamma'$ and hence that $\psi_\#\F \sqsubset  \F$.  This implies   that  $\psi_\#\F=  \F$ as desired.

Let $\gamma = R_-^{\inv}\rho R_+$ and $\gamma' = {R_-'}^\inv \rho' R_+'$ be the highest edge splittings of $\gamma$ and $\gamma' = \psi_\#(\gamma)$. Choose a   lift $\ti \gamma =  {\wt R_-}^\inv \ti \rho\wt R_+$ of $\gamma$,  let  $\ti f$ be the lift of~$f$ satisfying $\ti f_\#(\ti \gamma) = \ti \gamma$ and let $\Phi$ be the automorphism corresponding to $\ti f$. Let  $P_-$ and $P_+$  be the initial and terminal endpoints of $\ti \gamma$ respectively.   If $E^* = E_{r+1}$ then the initial endpoints of $R_+$ and $R_-$ are contained in $G_r$ and we  let $\wt C_\pm$ be the component of the full pre-image of $G_r$ that contains the initial endpoint of $\wt R_\pm$.  We break into cases depending on~$E_{r+1}$.

\textbf{Case~1: $E_{r+1}$ is nonfixed and nonlinear.} Thus $E_{r+1}$ is the initial edge of $R_+$ and  (a) does not hold.  We divide the argument into two subcases, one  in which (b) holds and one in which (c) holds. In both subcases, Corollary~\ref{higher edge case} implies that $P_+ \in \Fix_+(\wh\Phi)$ and that there exists $\Psi$ representing $\psi$ so that   $\ti \gamma' = \wh\Psi_\#(\ti \gamma)$ has highest edge splitting $\ti \gamma' =\wt R'_- \ti \rho' \wt R_+$. In particular, $P_+$ is fixed by~$\wh\Psi$. Also, since $E_{r+1}$ is nonfixed and nonlinear, it follows from Lemma~\ref{LemmaEigenrayDefs} that the point of $\bdy F_n / F_n$ represented by $P_+ \in \bdy F_n$ is an eigenray of $\phi$. Since $\Psi^m$, $\Phi$ are both representatives of $\phi$, both of which fix $P_+ \in \bdy F_n$, it follows by Definition~\ref{DefEigenraysInvariant} that $\Psi^m=\Phi$.
     
    If   (b) is satisfied then Corollary~\ref{higher edge case} implies that  that $P_- \in \Fix_+(\wh\Phi)$ and that    $R_-' = R_-$; the latter  implies that   the $F_n$-orbit of $P_-$ is fixed by $\psi$.     Lemma~\ref{Theta}\pref{item:fixing P} therefore implies that $P_- \in \Fix(\wh\Psi)$.  In this subcase  $\gamma' = \gamma$ and we are done.     
    
For the second subcase, (c) is satisfied.  Since $E_{r+1}$ is the central subpath of $\gamma$, $\rho$ is trivial.  Let $N$ be the element of $\cN$ determined by  the initial vertex of $E_{r+1}$, which is also the common initial vertex of $R_+$ and $R_-$.  Lemma~\ref{action of psi} implies that $\psi N$ contains the initial vertex of $R'_-$.  Since $\rho'$ is a Nielsen path connecting the initial vertex of $R_-'$ to the initial vertex of $R_+$, $N$ and $\psi N$ belong to the same Nielsen class.  Lemma~\ref{invariant Nielsen class} implies that $\psi N = N$ and  Lemma~\ref{Theta}\pref{item:preserving tilde C} implies that  $\wh\Psi(\partial \wt C_-) = \partial \wt C_-$. It follows that  the line connecting the initial endpoint $P_-'$ of $\ti \gamma'$ to the initial endpoint $P_-$ of $\ti \gamma$ projects into $G_r$. Since $P_+$ is the   terminal endpoint of both $\ti \gamma$ and $\ti \gamma'$ and since $\ti \gamma$ projects into $G_s$ it follows that $\ti \gamma'$ projects into $G_s$. This completes the proof in the case that the initial edge of $R_+$ has height greater than $r$ and is non-linear.
  
\textbf{Case 2: $E_{r+1}$ is linear.} As in Case~1, $E_{r+1}$ is the initial edge of $R_+$ and $R_+ = R_+'$. As a first subcase we assume that (a) is satisfied.  The rest of this paragraph and the next paragraph make no reference to   $E_{r+1}$ and so will also apply when (a) is satisfied and   $E_{r+1}$ is fixed. Since $\rho$ is trivial, the initial vertex  $v$ of $R_+$ is also the initial vertex of $R_-$.  If $N$ is the element of $\cN$  that contains  $v$ then by Lemma~\ref{action of psi}, $\psi N$ is the  element of $\cN$ that contains the initial endpoint $v'$ of $R_-'$.    Since $v$ and $v'$ are connected by the   Nielsen path $\rho'$, they are in the same Nielsen class and   Lemma~\ref{invariant Nielsen class} implies that $\psi N =N$.   
  
  By  Lemma~\ref{fixed class}, there exists $\Psi$ representing $\psi$ that commutes with $\Phi$ and preserves $\wt C_-$.    Since $\Phi$ also  preserves $\wt C_-$, $\Phi$ and $\Psi^m$  differ by an inner automorphism $i_c$ that   preserves $\wt C_-$.  Since $\Phi$ commutes with $\Psi^m$ it also commutes with $i_c$.  We claim that $i_c$ is trivial.  If not then the covering translation $L_c$ commutes with $\ti f$ and preserves $\wt C_-$.  The path connecting $\ti v$ to $L_c(\ti v)$ projects to a closed Nielsen path in $G_r$ that is based at $v$. This contradicts the hypotheses of case (a) and so verifies the claim. Thus $\Psi^m  = \Phi$.   
  
    Lemma~\ref{LemmaPointsAtInfinity} implies that $\Fix_N(\wh\Phi) \cap \partial \wt C_+$, which contains $P_+$, is a pair of points, each of which is contained in $\partial \Fix(\Phi)$.  Lemma~\ref{Theta}\pref{item:equal fix}  therefore implies that $P_+ \in \partial \Fix(\Psi)$.  We now know that $\Psi(P_-) \in \wt C_-$ and $P_+ \in \wt C_+$ are the endpoints of $\wh\Psi_\#(\ti \gamma)$.  It follows that $\gamma' \subset G_s$ as desired.  This completes the first subcase of the case that $E_{r+1}$ is linear.
  
     The second and last subcase is that (b) is satisfied.  Arguing as in the first case, with the roles of $E_{r+1}$ and $P_+$ being replaced by $E_{r+2}$ and $P_-$, we conclude that $P_- \in \Fix_+(\Phi)$ and that there exists $\Psi$ representing $\psi$ that fixes $P_-$ and that satisfies $\Psi^m = \Phi$.  As in the first subcase of this second case, Lemma~\ref{LemmaPointsAtInfinity}  and  Lemma~\ref{Theta}\pref{item:equal fix} imply that $\Psi$ fixes $P_+$ and hence that $\gamma' = \gamma$.  This completes the proof of the case that the initial edge of $R_+$ is linear.
     
\textbf{Case 3: $E_{r+1}$ is fixed,} and so (a) is satisfied.   The elements of $\cN$ containing the endpoints of $E_{r+1}$ are distinct and belong to the same Nielsen class.   Lemma~\ref{invariant Nielsen class} implies that both of the $G_r$-Nielsen classes are $\psi$-invariant. Arguing as in the first subcase of the second case, we conclude that there is an automorphism $\Psi$ representing $\psi$ that preserves $\partial \wt C_-$ and satisfies    $\Psi^m  = \Phi$.  Lemma~\ref{Theta}~\pref{item:preserving tilde C} implies that $\partial \wt C_+$  is preserved by $\Psi$.   Thus  $\Psi_\#(\ti \gamma)$ has one endpoint in $\partial \wt C_-$ and the other in    $\partial \wt C_+$.  The projected image $\gamma'$ is therefore contained in $G_s$. \qed

\section{Limit Trees} 
\label{SectionLimitTrees}

This section is devoted to material needed for the proofs of Propositions~\ref{PropFCarriesAll} and~\ref{PropNielsenPairsExist} given in later sections. Given a proper free factor system $\F$, we consider a certain rotationless $\phi \in \Out(F_n)$ that fixes $\F$, and we study the actions of $\phi$ on certain simplicial $F_n$-trees $T$ for which each element of $F_n$ carried by $\F$ is elliptic, with the goal of finding $\phi$-invariant trees by taking limits. We assume at all times that $\F$ carries both $\L(\phi)$ and $\Eigen(\phi)$, but we make no assumption that $\Twist(\phi)$ are carried by $\F$, and so twistors of $\phi$ are allowed to have positive translation length on~$T$ (throughout this section we use without comment the results on twistors in Section~\ref{SectionLaminationsAndTwistors}, namely Definitions~\ref{DefTwistorCT},~\ref{DefTwistorInvariant} and Fact~\ref{FactTwistor}). One of our chief goals is to get greater control over twistors. We gain this control by taking limits of trees: limits with scaling in Proposition~\ref{prop:grower} ``Iteration of a grower'', and limits without scaling in Proposition~\ref{prop:nongrower} ``Iteration of nongrower'', each of which gives information regarding how twistors behave under limits. Both of these results play a role in the proof of Proposition~\ref{PropFCarriesAll}; see  Notation~\ref{notn:generators} and  Lemma~\ref{LemmaFinitelyMany}. Proposition~\ref{prop:nongrower} is also used implicitly in the proof of Proposition~\ref{PropNielsenPairsExist}, towards the end of Section~\ref{SectionNielsenPairsExist} where we borrow from \BookTwo\ Sections~5.5 and~5.6, but to make this borrowing work we must apply Proposition~\ref{prop:nongrower} appropriately.
 
\paragraph{Standing Notation for Section~\ref{SectionLimitTrees}:}  We assume throughout this section that  $\phi$ is rotationless, that $\F$ is a proper $\phi$-invariant free factor system that carries both $\L(\phi)$ and $\eigen(\phi)$ and that  $\fG$ is a \ct\ representing $\phi$ in which  $\F$ is represented by a core filtration element~$G_r$.  Let $\ti f \from \wt G \to \wt G$ be a lift of $f$ to the universal cover $\wt G$ of $G$.

By Fact~\refGM{FactNEGEdgeImage},  each stratum $H_s$ above $G_r$ is a single \neg\ edge $E_s$ which is either fixed or satisfies $f(E_s) = E_s \cdot u_s$ for some closed path $u_s \subset G_{s-1}$.  Lemma~\ref{LemmaEigenrayCarried} and our assumption that $\F$ carries $\eigen(\phi)$ imply that if $u_s$ is not contained in $G_r$ then $E_s$ is fixed or linear.   When $E_s$ is linear we follow Definition~\ref{DefTwistorCT}, writing $f(E_s) = E_s w_s^{d_s}$ where $w_s$ is a Nielsen path that determines a twistor of $f$ in $G$ and $d_s \ne 0$; exceptional paths associated to $E_s$ have the form $E_s w_s^p \overline E_t$ where $p \in \Z$,  $E_t$ is a linear edge in the same linear family as $E_s$ meaning that $w_s=w_t$ and $f(E_t) = E_t w_t^{d_t}$.
     
Let $\T_\F$ be the space of very small simplicial $F_n$-trees in which each element carried by $\F$ is elliptic. Given $T \in \T_\F$, let $L_T$ to denote both the length function on paths in $T$ and the induced translation length function on conjugacy classes and on unoriented conjugacy classes, let $\A_+(T;\phi)$ denote the set of all twistors $[c]_u \in \Twist(\phi)$ such that $L_T[c]_u > 0$, and let $\ti h = \ti h_T \from \wt G \to T$ denote an equivariant map that takes vertices to vertices and maps each component of the full pre-image of  $G_r$ to a point. If $\A_+(T;\phi) \ne \emptyset$ then we refer to $T$ as a \emph{grower} and in Proposition~\ref{prop:grower} we prove convergence of the normalized sequence $\frac{1}{k} T \phi^k$. If $\A_+(T;\phi) = \emptyset$ then $T$ is a \emph{nongrower} and in Proposition~\ref{prop:nongrower} we prove convergence of the unnormalized sequence $T \phi^k$.

\subsection{Iteration of growers}

The following lemma is true regardless of whether $T$ is a grower, although if not then the quantity $L_G(\sigma)$ is zero and the lemma gives less useful information.

 
\begin{lemma} \label{LemmaGrowerFormula}  Assume the Standing Notation for Section~\ref{SectionLimitTrees}.  For any $T \in \T_\F$ the following hold:  
\begin{enumerate}
\item\label{ItemLGsigma}
For each path $\sigma \subset G$ and lift $\ti \sigma \subset \wt G$ the following limit exists:
$$L_G(\sigma) =\lim_{k\to \infty} \frac{1}{k}L_T(\ti h_\#\ti f^k_\#(\ti \sigma))
$$
\item \label{ItemLGFormula}
If $\sigma$ has a complete splitting rel $G_r$ 
then 
$$L_G(\sigma) =   \sum_s M_s \abs{d_s} L_T([w_s]) + \sum_s N_s \abs{d_s- d_t} L_T([w_s]) 
$$
where both sums are taken over the set of all $s$ such that $[w_s]_u \in \A_+(T;\phi)$, and for each such $s$ the complete splitting of $\sigma$ has $M_s$ linear edge terms of the form $\sigma_i = E_s$ or $\overline E_s$ and $N_s$ exceptional path terms of the form $\sigma_i = E_s w^p_s \overline E_t$.
\end{enumerate}
\end{lemma}
  
\begin{proof}  After replacing $\sigma$ by some $f^l_\#(\sigma)$, we may assume by Fact~\ref{FactRelSplit}~\pref{ItemRelSplitIterate} that $\sigma$ has a complete splitting  rel $G_r$,   $\sigma =\sigma_1 \cdot \ldots \cdot \sigma_m$. 

For each $\sigma_i$ and $k > 0$, we choose $J_i \ge 1$, independent of $k$, and a decomposition of $f^k_\#(\sigma_i)$ into $J_i$  subpaths as follows:
\begin{itemize}
\item If $\sigma_i$ has height $\le r$ or is a fixed edge or Nielsen path of 
height $> r$, then $J_i=1$ and the decomposition is the trivial one.
\item If  $\sigma_i = E$ [resp.\ $\overline E]$ for some oriented non-fixed   non-linear edge $E$ of height $>r$ then $J_i = 2$. One of the terms is $E$ [resp. $\overline E$] and the other is a (possibly trivial) path in $G_r$.
\item If $\sigma_i = E_s$ [resp.\ $\overline E_s]$ for some oriented linear edge $E_s$ of height $>r$ then $J_i = 2$. One of the terms is $E_s$ and the other is $w_s^{\pm k d_s}$.
\item If $\sigma_i = E_s w^p_s \overline E_t$ is an exceptional path of height $>r$  then $J_i = 4$ and the decomposition is $f^k_\#(\sigma_i) =  (E_s) (w^p_s) (w_s^{(d_s-d_t)k}) (\overline E_t)$.
\end{itemize}

It follows that $f^k_\#(\sigma)=f^k_\#(\sigma_1) \cdot \ldots \cdot 
f^k_\#(\sigma_m)$ can be decomposed into $J=J_1+...+J_m$ subpaths 
$\mu_{1,k},...,\mu_{J,k}$ such that   for each $j=1,\ldots,J$  one of the following holds for all $k \ge 0$:
 \begin{description}
  \item [($i$)]   $\mu_{j,k} \subset G_r$.
 \item [($ii$)]  $\mu_{j,k}=E$ of $\overline E$  for some oriented edge $E$.   
 \item [($iii$)] $\mu_{j,k} = w_s^{\pm kd_s}$ for some   $s$. 
 \item [($iv$)] $\mu_{j,k}= w_s^{(d_s -d_t)k}$ for some $s,t$.
 \item[($v$)]  $\mu_{j,k} = \rho$ for some Nielsen path $\rho$.
  \end{description}
Moreover, there is one type ($iii$) term for each $\sigma_j $ that, up to a change of orientation, is a single linear edge $E_s$ with $w_s \not \subset G_r$ and one type ($iv$) term for each $\sigma_j$ that is an exceptional path $E_sw_s^p \overline E_t$ with $w_s \not \subset G_r$.

The bounded cancellation lemma (Fact~\ref{FactBCCVerySmall})  implies that 
$$  \lim_{k\to \infty} (\frac{1}{k}L_T(\ti h_\#\ti f^k_\#(\ti \sigma)) -\sum_{j=1}^J  \frac{1}{k}L_T(\ti h_\#(\ti \mu_{j,k})))=0
$$
and that
$$  \lim_{k\to \infty} \frac{1}{k}L_T(\ti h_\#(\ti w_s^{kq})) = \abs{q} L_T([w_s])
$$
Since $h$ maps components of $\wt G_r$ to vertices of $T$,  
$$  \lim_{k\to \infty} \frac{1}{k}L_T(\ti h_\#(\tau_k)) = 0
$$
for any sequence of paths $\tau_k \subset G_r$. The last displayed equation is also true if $\tau_k$ is independent of $k$,  for example a constant edge $E$ or Nielsen path $\rho$. The lemma follows immediately.
\end{proof}

\begin{proposition}[Iteration of a grower]
\label{prop:grower} Assume the Standing Notation for Section~\ref{SectionLimitTrees}. For any $T \in \T_\F$, if $\A_+(T;\phi)$ is nonempty then 
\begin{enumerate}
\item $T \phi^\infty := \lim_{k\to \infty} \frac{1}{k}T\phi^k$ is a well defined $\phi$-invariant element of $\T_\F$. 
\item If a cyclic subgroup $\langle e\rangle $ is the stabilizer of some edge in $T \phi^\infty$ then $[e]_u \in \A_+(T;\phi)$.
\item If a conjugacy class $[c]$ is represented by a circuit with a complete splitting rel $G_r$ and if $\sigma$ is a closed path obtained from this circuit by subdividing at a vertex between terms of the splitting, then (in the notation of Lemma~\ref{LemmaGrowerFormula}) $L_{T\phi^\infty}[c] = L_G(\sigma)$.
\end{enumerate}
\end{proposition}

\begin{proof} If $[c]$ and $\sigma$ are as in (3) then the bounded cancellation lemma (Fact~\ref{FactBCCVerySmall}) implies  
$$\lim_{k\to \infty} \left( \left( \frac{1}{k}L_{T}(\phi^k[c]) \right) -  \frac{1}{k}L_T\!\left(\ti h_\#(\ti f^k_\#(\ti \sigma))\right) \right)=0
$$
and hence
$$\lim_{k\to \infty} \frac{1}{k}L_{T\phi^k}[c] = \lim_{k\to \infty} \frac{1}{k}L_{T}(\phi^k[c]) = \lim_{k\to \infty}\frac{1}{k}L_T\!\left(\ti h_\#(\ti f^k_\#(\ti \sigma))\right) = L_G(\sigma)
$$
where the last equality follows from  Lemma~\ref{LemmaGrowerFormula}. This completes the proof of (3).


Given an arbitrary conjugacy class $[c]$, applying Fact~\ref{FactRelSplit}~\pref{ItemRelSplitIterate} choose $k_0 >0$ so that $\phi^{k_0}[c]$ is represented by a circuit that has a complete splitting rel $G_r$ and let $\sigma$ be a closed path obtained from this circuit by subdividing at a vertex between terms of the splitting. Item (3) implies that 
$$\lim_{k\to \infty} \frac{1}{k}L_{T}(\phi^k[c]) = L_G(\sigma) \qquad\qquad(*)
$$
In particular,  the non-zero values taken by these limits is a non-empty set that is bounded below by the minimum $T$-length of an element of $\A_+(T;\phi)$. Applying Fact~\ref{LimitIsSimplicial}, the limiting tree $T \phi^\infinity$ exists satisfying $L_{T\phi^\infinity}[c]=L_G[c]$ for all $[c]$, and is clearly $\phi$-invariant. If $[c]$ is supported by~$\F$ then so is $\phi^k[c]$ for all $k$, because $\phi(\F)=\F$; it follows that $L_{T}(\phi^k[c])=0$ for all~$k$, implying that $L_G(\sigma)=0$. This proves that $T\phi^\infinity \in \T_\F$, finishing the proof of~(1).

If $\L(\phi) = \emptyset$ then (2) is Remark 4.38 of \BookTwo, which follows from an explicit construction of  $T \phi^\infty$ (pages 37 and 38 of that paper).    The construction of the tree and the verification that it is $T \phi^\infty$ carries over to our context without change.     

For a  reference to a statement rather than  a construction,   define    $g : G \to G$ by $g(E_s) = f(E_s) = E_sw_s^{d_s}$  if $E_s$ is a linear edge for $f$ and the twistor determined by $w_s$ is contained in $\A_+(T;\phi)$ and by $g(E) = E$ for all other edges of $G$.    Then $g$ is    a homotopy equivalence by Lemma 6.7 of \cite{FeighnHandel:abelian} and each $w_s$ as above is a Nielsen path for $g$ by Lemma~6.13~(1) of \cite{FeighnHandel:abelian}.    Applying Lemma~\ref{LemmaGrowerFormula} and items (1) and (3) of this lemma to the outer automorphism $\psi$ determined by $g$, we see that $T \psi^\infty$ is well defined and equal to $T \phi^\infty$.   Hence there is no loss in replacing $\phi$ by $\psi$.  Since $\L(\psi) = \emptyset$, Remark 4.38 of \BookTwo\   can be quoted directly to prove (2). 
\end{proof}

\subsection{Iteration of nongrowers}

We now turn to the case that $T$ is a nongrower, that is, $\A_+(T;\phi) = \emptyset$ and so each twistor of $\phi$ is elliptic in~$T$. For a simplicial tree $T$, let $\Arc(T)$ denote the set of conjugacy classes of stabilizers of edges in $T$.

\begin{lemma} \label{arc stabilzers} Suppose that $\phi$ is rotationless, that $T \in \T_\F$ and that $T \phi^\infty := \lim_{n\to \infty} T\phi^k$ is a well defined element of $T_{\F}$.  Then $\Arc(T \phi^\infty ) \subset \Arc(T)$ and elements of $\Arc(T \phi^\infty)$ are $\phi$-invariant.
\end{lemma}

\begin{proof}  This is Lemma 5.9 of \BookTwo\ with slightly different hypotheses. The proof uses the fact that a conjugacy class that is periodic under the action of $\phi$ is fixed by $\phi$. In \BookTwo\ this is established by Proposition 3.16 of that paper which does not apply in our context because we do not assume the hypothesis that $\phi$ is \upg. Instead we use our hypothesis that $\phi$ is rotationless and we quote \recognition\ Lemma~3.30 which gives the same conclusion.  Other than that, the proof carries over to our context without change.  
\end{proof}


\begin{proposition}[Iteration of a nongrower]
\label{prop:nongrower}  
Assume the Standing Notation for Section~\ref{SectionLimitTrees}. Given $T \in \T_\F$, if $\A_+(T;\phi)$ is empty and if no element of $\Arc(T)$ is carried by~$\F$ then:
\begin{enumerate}
\item \label{ItemLengthEventuallyConstant}
For all $c \in F_n$, \ $L_T(\phi^k[c])$ is eventually constant.
\item \label{ItemTInfinityInvariant}
$T \phi^\infty := \lim_{n\to \infty} T\phi^k$ is a well defined $\phi$-invariant element of $\T_\F$.
\item \label{ItemArcsStabilizersDontGrow}
$\Arc(T\phi^\infty) \subset \Arc(T)$.
\end{enumerate}  
\end{proposition}
 
The proof of the proposition follows some preliminary notation and a sublemma.
 
Assume the Standing Notation for Section~\ref{SectionLimitTrees}. In particular, $\ti h: \wt G \to T$ is an equivariant map that sends vertices to vertices and sends each component of  the full pre-image of $G_r$ to a vertex. Subdivide the edges of $\wt G$ into \emph{edgelets} that are mapped by $\ti h$ to either vertices  in $T$ or single edges in $T$.  Paths with endpoints in this subdivision are called \emph{edgelet paths}. An edgelet path $\ti \nu$ is \emph{elliptic} if $\ti h_\#(\ti \nu)$ is trivial or equivalently if both endpoints of $\ti \nu$ have the same image under $\ti h$. An edgelet is \emph{horizontal} if its image under $\ti h$ is an edge. The edgelet subdivision is equivariant, as are the notions of elliptic and horizontal, and so by projection these notions pass to~$G$.
 
Given a linear edge $E_s$ with height greater than $r$, let $\ti w_s$ be a  lift of $w_s$ and let $L_a$ be the covering translation that maps the initial endpoint of $\ti w_s$ to  the terminal endpoint of~$\ti w_s$. Thus $\ti w_s$ is a fundamental domain for action of $L_a$ on its axis. The action of $a$ on $T$ has a nontrivial fixed point set $Z$  because $[w_s]$ is elliptic in $T$. Since the $\ti h_\#$-image of the axis of $L_a$ is invariant under  the action of $a$ on $T$,  $\ti h_\#(\ti w_s)$  intersects $Z$.  Subdivide $\ti w_s$ at the first point   in  $\ti w_s \cap \ti h^{-1}(Z)$ and let $\ti w_s  =\ti w_s' \ti w_s''$ be the induced decomposition.  Define $\ti w^*_s = \ti w''_s L_a(\ti w'_s)$. Thus $\ti w^*_s$ is  an elliptic edgelet subpath that is a  fundamental domain for action of $L_a$ on its axis. 

 
   \begin{sublem} \label{nongrower decomposition} Assuming notation as above,  suppose that $\sigma \subset G$ is a circuit and that  there exist arbitrarily large $k$ such that  $ f_\#^k(\sigma)$ is not elliptic in $T$.   Then for all sufficiently large~$k$ the circuit $ f_\#^k(\sigma)$ has a cyclic decomposition into edgelet subpaths $f_\#^k(\sigma) = H_1\nu_{1,k}H_2 \nu_{2,k} \ldots H_J\nu_{J,k}$  so that the following hold:
  \begin{enumerate}
 \item $J$ is independent of $k$.
 \item  each $H_j$ is a horizontal edgelet that  is independent of $k$.
 \item for each $1 \le j \le J$ there exist arbitrarily large $k$ such that $ H_j\nu_{j,k}H_{j+1}$
is not elliptic, where $j$ is taken mod $J$.
  \item  each  $\nu_{j,k}$ is (a possibly trivial) elliptic edgelet subpath that  is a concatenation $\nu_{j,k} = \eta_{j,k,1}\dots\eta_{j,k,P_j}$ of edgelet subpaths  where the number of terms $P_j$ in this concatenation  depends only on $j$ and each term $\eta_{j,k,l}$ satisfies one of the following:
      \begin{enumerate}
  \item   $\eta_{j,k,l}$ does not depend on $k$; or
  \item  $\eta_{j,k,l} \subset G_r$ for all $k$ and $ \lim_{k \to \infty}\abs{\eta_{j,k,l}} = \infty$; or
  \item $\eta_{j,k,l}$ is an iterate of some $w^*_s$  for all $k$ where $s$ does not depend on $k$ and  where $\lim_{k \to \infty}\abs{\eta_{j,k,l}} = \infty$.
\end{enumerate} 
\end{enumerate}
\end{sublem}

In the proof and later application of Sublemma~\ref{nongrower decomposition} we say that an edgelet path or parameter is \emph{constant} if it is independent of $k$.

\begin{proof} After replacing $\sigma$ with some $f^m_\#(\sigma)$ we  may assume by Fact~\ref{FactRelSplit}~\pref{ItemRelSplitIterate} that $\sigma$ has a complete splitting rel $G_r$, \  $\sigma = \sigma_1 \cdot \ldots \cdot \sigma_m$.      
The first step in the proof is to show that  for all $i$  and for all sufficiently large $k$, $ f_\#^k(\sigma_i)$ decomposes as a concatenation of a constant number of edgelet subpaths, each of which is either a constant horizontal edgelet or is elliptic and satisfies   (a), (b) or (c).  This is accomplished by a simple case analysis.

If $\sigma_i$ is a Nielsen path then  $f^k_\#(\sigma_i)$ is constant and so has a constant decomposition into constant horizontal edgelets and constant elliptic edgelet subpaths.  We may therefore assume that $\sigma_i$ is not a Nielsen path and hence (\recognition\ Lemma~4.13, Fact~\refGM{FactPNPFixed}) not a periodic Nielsen path. In particular, $\abs{f_\#^k(\sigma_i)} \to \infty$.
If $\sigma_i \subset G_r$   then (b) holds so we may assume that $\sigma_i$ has height greater than $r$.   If $\sigma_i =  E_s w_s^p \overline E_t$ is an exceptional path  then   $f^k_\#(\sigma_i) = E_sw_s^{p+(d_s-d_t)k} \overline E_t = E_s w_s' {w^*_s}^{p+(d_s-d_t )k-1} w_s'' E_t$ decomposes as a concatenation of two constant edgelet subpaths and one elliptic edgelet subpath that satisfies (c).  We are now reduced to the case that $\sigma_i$ is a single non-fixed edge $E_s$.    If $E_s$ is  linear  then  $f^k_\#(\sigma_i) =  E_s w_s' {w^*_s}^{kd_s-1} w_s'' $ decomposes as a concatenation of two constant edgelet subpaths and one elliptic edgelet subpath that satisfies (c).      If   $E_s$ is non-linear than  $f^k_\#(\sigma_i)$ decomposes as $E_s$ followed by an initial segment of its eigenray in $G_r$ and so decomposes as  a concatenation of a constant edgelet subpath and an elliptic edgelet subpath that satisfies (b).    This completes the first step in the proof.

Concatenate the decompositions of the $\sigma_i$'s to give a decomposition of $\sigma$. Amalgamate adjacent elliptic edgelet subpaths into a single elliptic edgelet subpath.   If there are adjacent horizontal edgelets insert a trivial elliptic edgelet subpath between them.   We now have a decomposition $\sigma_k = H_1\nu_{1,k}H_2 \nu_{2,k} \ldots H_J\nu_{J,k}$ satisfying (1), (2) and (4).  If there exists $j $ such that $ H_j\nu_{j,k} H_{j+1}$ is elliptic for all sufficiently  large $k$,  rename this as a single elliptic term and increase the threshold value of $k$ for which our decomposition is defined.  After finitely many amalgamation and renaming steps we obtain the desired decomposition.
\end{proof}

\begin{proof}[Proof of Proposition~\ref{prop:nongrower}.]  Assuming for now that (1) holds, we prove (2) and (3).
 
If $a^1,\dots,a^n$ is a basis for $F_n$ then there exists $K$ so that for all $i,j$ the quantity $L_T(\phi^k[a^i a^j])$ is constant, independent of $k \ge K$. By a result of Serre \cite{Serre:trees}, Section 6.5, Corollary 2, since the action of $F_n$ on $T$ has no global fixed point, there exist $i,j$ so that the stable value of $L_T(\phi^k[a^i a^j])$ is nonzero. It follows that as the conjugacy class $[c]$ varies, the values of the expression $\lim_{n\to \infty} L_T(\phi^k[c])$ are not all zero. Furthermore, the smallest positive value of this expression is greater than or equal to the smallest positive length of a conjugacy class in $T$. Fact~\ref{LimitIsSimplicial} implies  that $T\phi^k$ converges to a very small simplicial tree $T \phi^\infty$. The inclusion $T \phi^\infinity \in \T_\F$ is proved exactly as in Proposition~\ref{prop:grower}. This completes the proof of (2). Item (3) follows from  Lemma~\ref{arc stabilzers}.  
   
We now turn to the proof of (1). Given  $c \in F_n$ and $k \ge 0$ let  $[c_k] =  \phi^k_\#[c]$.    We may assume that there are arbitrarily large values of $k$ with $L_T([c_k]) \ne 0$. For all sufficiently large $k$, let  $\ldots \ti H_{j,k} \ti \nu_{j,k} \ti H_{j+1,k}\ldots$ be the decomposition of the axis of $L_{c_k}$ obtained  by lifting the decomposition of $f^k_\#(\sigma_k)$ given by Sublemma~\ref{nongrower decomposition}.   
 
We show below that for each $j$, if $k$ is sufficiently large  (depending on $j$), then $\ti h( \ti H_{j,k}) \ne \ti h(\ti H_{j+1,k})^{-1}$. This implies that  $\ti h( \ti H_{j,k}) \ne \ti h(\ti H_{{j'},k})^{-1}$ for any $j \ne j'$ and hence that $ \ti h(\ti H_{1,k})\ldots  \ti h(\ti H_{J,k})$ is a fundamental domain for the action of $c_k$ on $T$.  This proves that     $L_T([c_k]) $ is eventually constant with the stable value being   $\sum_{j=1}^J L_T(\ti h(\ti H_{j,k}))$ and so completes the proof of (1).  
 
 It remains to fix $j$ and show that $\ti h( \ti H_{j,k}) \ne \ti h(\ti H_{j+1,k})^{-1}$ for all sufficiently large $k$.   After possibly replacing $c_k$ with another element in its conjugacy class, we  may assume that   $\ti H_{j,k}$ is independent of $k$; denote this edgelet $\ti e$ and denote   $\ti H_{j+1,k}$ by $\ti e'_k$.     Let $\E$ be the set of edgelets that have the same $\ti h$-image as $\ti e^{-1}$ and note that  $H_{j,k}  \nu_{j,k}  H_{j+1,k}$ is elliptic if and only if    $\ti e'_k \in \E$.    Since there are only finitely many $F_n$-orbits of edgelets in $\wt G$, \ $\E$ is finite if the stabilizer of $\ti h(\ti e)$ is trivial and  is a finite union of $L_b$-orbits    if the stabilizer of $\ti h(\ti e)$ is a non-trivial cyclic group $\<b\>$ and $L_b$ is the covering translation of $\wt G$ corresponding to $b$. 
      
We must show that $\ti e'_k \not \in \E$ for sufficiently large $k$. By  Sublemma~\ref{nongrower decomposition}~(3) we may assume that $\nu_{j,k}$ is not constant and by (4) it follows that the length of $\nu_{j,k}$ goes to infinity.    We may therefore also assume that $\E$ is infinite and that  the stabilizer of $\ti h(\ti e)$ is a  non-trivial cyclic group $\<b\>$.    
 We argue by contradiction, assuming that   there exist $k_q\to \infty$  and $m_q$ such that  $\ti e'_{k_q} = L_{b^{m_q}} \ti e'_{k_1}$ and deriving a contradiction.    After replacing $b$ by $b^{-1}$ and passing to a further subsequence we may assume that $m_q \to \infty$.

     The path $\ti \mu_k =  [\ti \nu^{-1}_{j,k_1} \ti  \nu_{j,k}] $ connects  $\ti e'_{k_1}$ to $\ti e'_{k}$.   Denote its projection to $G$ by  $\mu_k$.     In the special case that $k = k_q$,   $\mu_{k_q}$ factors into  initial and terminal edgelet subpaths that do not depend on $k_q$ and a central edgelet subpath equal to  $m_q$ fundamental domains of $L_b$.  
   By hypothesis, $[b]$ is not carried by $\F$   and so there is an  upper bound to the length of an edgelet subpath of the axis of $L_b$ that projects into $G_r$.  It follows that there are no type 4(b) paths in the decomposition of  $ \nu_{j,k_q}$ provided by Sublemma~\ref{nongrower decomposition}.  If $\eta_{j,k,l}$ is a type 4(c) term,  then the axes of $\omega_s$ and $b$ have arbitrarily large overlaps  and so must be equal.   
    
    To summarize,  for all sufficiently large $k$, $\ti \mu_k$ has a decomposition into edgelet subpaths with a constant number of terms so that for certain arbitrarily large $k$ (the $k_q$'s) each  term in this decomposition whose projection into $G$ is not constant, is contained in the axis of $L_b$ and equals an integral number of fundamental domains  for the action of $L_b$ on its axis.  Moreover, fixing one such $k_q$, if $k > k_q$ then $\ti \mu_{k}$ is obtained from $\ti \mu_{k_q}$ by changing the number of fundamental domains of the axis of $L_b$ that are contained in the terms whose projection to $G$ is not constant.  Thus, for $k \ge k_q$,    $\ti \mu_k$ has a decompostion into initial and terminal subpaths whose projections into $G$ are constant and a central subpath that is an integral number of fundamental domains of the axis of $L_b$.  If follows that   the $L_b$-orbit of $\ti e'_k$ is independent of $k$ for $k \ge k_q$.   But then $\ti e'_k \in \E$ for all  $k > k_j$ in contradiction to item (3) of Sublemma~\ref{nongrower decomposition}.   This  contradiction completes the proof.
    
\end{proof}


\section{Carrying asymptotic data: Proof of Proposition~\ref{PropFCarriesAll}} 
\label{SectionFCarriesAll}

In this section we prove  

\medskip

\noindent {\bf Proposition~\ref{PropFCarriesAll}.} \ \ \emph{Suppose that $\F$ is a proper free factor system and  that $\h \subgroup \IA_n(\Z/3)$ is a finitely generated subgroup such that $\h \subset \PGF$ and $\h$ is irreducible rel $\F$. Then $\F$ carries $\Lam(\phi)$ for each rotationless $\phi \in \h$.}

\medskip

We assume through this section that $\h \subgroup \IA_n(\Z/3)$ is finitely generated and that $\F$ is a proper free factor system invariant under each element of $\h$. As needed in various lemmas we shall bring in the hypotheses that $\h \subset \PGF$ and that $\h$ is irreducible rel~$\F$. Let 
$$\Asym(\h) = \union_\phi \Asym(\phi)
$$
taken over all rotationless $\phi \in \h$. The hypothesis $\h \subset \PG^\F$ already implies that $\F$ carries all attracting laminations in $\Asym(\h)$, so our goal is to prove that $\F$ carries all eigenrays and twistors in $\Asym(\h)$. 

Our strategy is to use an eigenray or twistor in $\Asym(\h)$ that is not carried by $\F$ to construct something whose existence violates one of the hypotheses: either we construct an element of $\h$ whose attracting laminations are not all carried by~$\F$, contradicting the hypothesis that $\h \subset \PGF$; or we construct a proper $\h$-invariant free factor system that properly contains $\F$, contradicting that $\h$ is irreducible rel $\F$.

\subsection{Carrying eigenrays} 

The following lemma carries out the above strategy with respect to eigenrays.

\begin{lemma} \label{LemmaNonLinearCase} If $\h \subgroup \IA_n(\Z/3)$ is finitely generated, $\F$ is a proper free factor system, $\h \subset \PGF$, and $\h$ is irreducible rel $\F$, then the free factor system $\F$ carries every eigenray of every rotationless $\phi\in \h$.  More precisely, suppose that $\fG$ is a \ct\ representing $\phi \in \h$ and that   $\F$ is represented by a filtration element $G_r$.  Suppose further that $E_s$ is a superlinear \neg-edge of $G$ and that  $f(E_s) =E_s \cdot u_s$ for some non-trivial closed path $u_s \subset G_{s-1}$.  Then $u_s \subset G_r$.
\end{lemma}    

\begin{proof}  This is Proposition~5.5 of \BookTwo\ with some   slight changes.  
We assume that $u_s \not \subset G_r$ and argue to a contradiction.
 
By  Facts~\ref{all neg} and \refGM{FactNEGEdgeImage}, every stratum above $G_r$ is a single \neg\ edge and for each such edge we adopt the notation $f(E_t) = E_t \cdot u_t$ from Fact~\refGM{FactNEGEdgeImage}. If there is an edge $E_t$  such that $f(E_t)$ crosses $E_s$ then $E_s$ or $\overline E_s$ is a term in the complete splitting rel $G_r$ of $u_t$ by Fact~\ref{FactIsATerm}~\pref{ItemLinearEdgeTerm}. It follows that $u_t$ is not a Nielsen path and hence that $E_t$ is superlinear. Obviously $u_t \not \subset G_r$ so there is no loss in replacing $E_s$ with $E_t$. After repeating this finitely many times, we may assume that there is no edge other than $E_s$ whose image contains $E_s$, so by reordering strata we may assume that $E_s$ is the highest edge in the filtration.   

Let $P \in \ray$ be the endpoint of  the eigenray $ E_s\cdot u_s\cdot f_\#(u_s)\cdot \ldots$ associated to $E_s$. By Lemma~\ref{LemmaEigenrayCarried}, $P$ is not  contained in the subset of $\ray$ determined by $\F$. The joint free factor support of $\F$ and the $\h$-orbit of~$P$ is $\h$-invariant by Fact~\refGM{FactFFSPolyglot}. If this free factor system is proper then we have constructed a contradiction to the hypothesis that $\h$ is irreducible rel~$\F$.
 
We may therefore assume that the free factor support of $\F$ and the $\h$-orbit of $P$ equals $\{[F_n]\}$. In particular, there exists $\eta \in \h$ such that $\eta(P)$ is not  carried by   the free factor system  corresponding to the subgraph $G \setminus E_s$.  Thus  any ray with endpoint $\eta(P)$ crosses $E_s$ infinitely many times.
  
Represent $\eta$ by a homotopy equivalence $h : G \to G$. By the bounded cancellation lemma there exists $n_1> 0$ so that $h_{\#\#}(f^{n_1}_\#(E_s)) = h_{\#\#}(E_su_sf_\#(u_s) \ldots f^{n_1-1}_\#(u_s))$ crosses $E_s$ at least three times.    It follows that  $h_{\#\#}(f^{n_1}_\#(E_s))$ contains a subpath $\nu$ that, up to a reversal of orientation, has either the form $E_s\tau\overline E_s$ or the form $E_s\tau_1 E_s \tau_2 E_s$ where $\tau, \tau_1,\tau_2$ have height~$<s$. 
      
Assume for now that $\nu = E_s \tau \overline E_s$.  By Fact~\ref{FactBasicNEGSplitting},  $\nu$ satisfies a universal splitting property:   if $\sigma$ is any path in $G$ then any decomposition into subpaths of the form  $\sigma = \alpha \nu \beta$ is a splitting with respect to $f$.   Thus $f^k_{\#\#}(\nu) = f^k_\#(\nu)$ for all $k\ge 1$. By Fact~\ref{FactRelSplit}~\pref{ItemRelSplitIterate} and Fact~\ref{FactIsATerm}~\pref{FactSuperlinearEdgeTerm}, there exists  $n_2>0$ so that $f_\#^{n_2}(\nu)$ has a complete splitting rel $G_r$  with at least one term being $E_s$ and at least one term being $\overline E_s$.   After increasing $n_2$, we may assume that $f_\#^{n_2}(\nu)$  contains $f^{n_1}_\#(E_s)$ and $f^{n_1}_\#(\overline E_s) $ as disjoint subpaths.   Lemma~\trefGM{LemmaDoubleSharpFacts}{ItemDblSharpComp} implies that $(hf^{n_2})_{\#\#}(\nu)$ contains $h_{\#\#}(f^{n_2}_{\#\#}(\nu))$ as a subpath and so by Lemma~\trefGM{LemmaDoubleSharpFacts}{ItemDisjointCopies}  has  a decomposition into subpaths in which at least one term is $\nu$ and at least one term is $\bar \nu$.    
   
    Let $g = hf^{n_2}$.  We claim that for all $K \ge 1$, if a path $\sigma$ contains $K$ disjoint subpaths, each of which is either $\nu$ or $\bar \nu$, then $g_{\#\#}(\sigma)$ contains $2K$ disjoint subpaths, each of which is either $\nu$ or $\bar \nu$.    The proof is by induction on $K$.  The base case $K=1$ follows from  Lemma~\trefGM{LemmaDoubleSharpFacts}{ItemDblSharpContain}  and the conclusion of the previous paragraph.  For the induction case, choose disjoint subpaths $\alpha$ and $\beta$ of $\sigma$, where $\alpha$ contains $K-1$ disjoint copies of $\nu$ or $\bar \nu$ and $\beta =$    $\nu$ or $\bar \nu$.   Lemma~\trefGM{LemmaDoubleSharpFacts}{ItemDisjointCopies}  implies that $g_{\#\#}(\sigma)$ contains  $g_{\#\#}(\alpha)$ and $g_{\#\#}(\beta)$ as disjoint subpaths so the inductive hypothesis completes the proof of the claim.  
    
Let $\psi \in \h$ be the outer automorphism represented by $g$. If $\tau$ is a circuit in $G$ containing $\nu$ as a subpath then $g^k_\#(\tau)$ contains  $g^k_{\#\#}(\nu)$ as a subpath by Lemma~\trefGM{LemmaDoubleSharpFacts}{ItemCircuit} and so contains $2^k$ disjoint subpaths that are copies of either $\nu$ or $\bar \nu$.  Applying Lemma~\ref{LemmaFindingEG} we have constructed an attracting lamination $\Lambda \in \L(\psi)$ which is not supported by $[\pi_1 G_r]=\F$, contradicting the assumption that $\h \subgroup \PGF$ and thereby completing the proof of the lemma in the case that $\nu = E_s \tau \overline E_s$.
   
If $\nu = E_s \tau_1E_s\tau_2E_s$ then the universal splitting property is that if   $\sigma = \alpha \nu \beta$ is a decomposition into subpaths  then $\sigma = \alpha \nu' \beta'$ is a splitting where $\nu' =E_s\tau_1E_s\tau_2$ and $\beta' = E_s \beta$.  Thus  $f^k_{\#\#}(\nu)$ contains $  f^k_\#(E_s\tau_1E_s\tau_2)$ as a subpath for all $k\ge 1$. The rest of the argument requires only  minor modifications that are left to the reader.
\end{proof} 

\subsection{The exponential growth digraph for twistors} 

In this section and the next we carry out our strategy with respect to twistors. The first thing one must understand is why the proof given for eigenrays does not immediately apply to twistors. Following the notation of Lemma~\ref{LemmaNonLinearCase} but assuming that $u_s$ is a Nielsen path, there still exists $\eta \in \h$ such that $\eta(u_s^n)$ crosses $E_s$ at least three times.  The problem is that each occurence of $E_s$ in $h_{\#\#}f_\#^{n_1}(E_s)$ might be contained in a Nielsen path and so the occurrences of $E_s$ do not increase under further iteration. In this situation the feedback mechanism needed to produce exponential growth breaks down and the proof fails (Fact~\ref{FactIsATerm}~\pref{ItemLinearEdgeTerm} is what prevents this breakdown in the superlinear case). Lemma~\ref{LemmaNoClosedEdgePaths} is designed to exploit this situation to set up a different kind of feedback mechanism using twistors.

Before defining the exponential growth digraph for twistors, we set up some notation needed for the definition.

Recall, given a rotationless $\phi \in \Out(F_n)$, the finite set $\Axes(\phi)$ of twistors of $\phi$, each twistor being a certain unoriented conjugacy class $[a]_u$ where $a$ is root-free. See Definition~\ref{DefTwistorInvariant} for the invariant definition $\Axes(\phi)$, Definition~\ref{DefTwistorCT} for the definition in the \ct\ context, and Fact~\ref{FactTwistor} for the equivalence of these two definitions. Given a free factor system $\F$ invariant under $\phi$, let $\Axes(\phi,\F)$ denote the set of those $[a]_u \in \Axes(\phi)$ such that $[a]_u$ is \emph{not} carried by $\F$.
 
\begin{notn}
\label{notn:generators}  
Consider a subgroup $\h \subgroup \IA_n(\Z/3)$, a proper free factor system $\F$ invariant under each element of $\h$, and a finite subset $\{\psi_1,\ldots,\psi_\gen\}$ of $\h$ (in Section~\ref{SectionFollowTheBouncingBall} this subset is a generating set; for now we allow any finite subset). For each $i=1,\ldots,\gen$, choose a rotationless power $\phi_i$ of $\psi_i$, and choose a \ct\ $f_i :G^i \to G^i$ representing $\phi_i$ in which $\F$ is represented by a filtration element denoted $G^i_{r(i)}$. For each linear edge we have $f_i(E^i_s) = E^i_s\cdot (w^i_s)^{d^i_s}$ where $w^i_s \subset G_{s-1}$ is a closed path (see Definition~\ref{DefTwistorCT}). Recalling from Section~\refGM{SectionGraphsPathsCircuits} that ``paths'' are parameterized by closed intervals and ``circuits'' are parameterized by circles, to help differentiate the path $w^i_s$ from its corresponding root-free circuit, we will denote the latter by $x^i_s$. Thus each twistor for $\phi_i$ is realized in $G^i$ by some $x_s^i$, and in this case we shall write that twistor as $[x_s^i]$. Note that $[x^i_s] \in \Twist(\phi_i,\F)$ if and only if the height of the path $w^i_s$ is greater than $r(i)$.
\end{notn}

\begin{definition}[The exponential growth digraph for twistors] 
\label{DefExpGrowthGraph}
We define the \emph{exponential growth digraph $\Gamma$}, associated to the data of Notation~\ref{notn:generators}, as follows. The vertex set is $\cup_{i=1}^\gen \Axes(\phi_i,\F)$. Given a vertex $[a]_u \in \cup_{i=1}^\gen \Axes(\phi_i,\F)$, an integer $i \in \{1,\ldots,\gen\}$ such that $[a]_u \in \Axes(\phi_i,\F)$, and an integer $s$ such that $E^i_s$ is a linear edge in the linear family of $f_i$ associated to the twistor $[a]_u = [x_s^i]$, we say that the ordered pair $(i,s)$ is a \emph{label} for the vertex $[a]_u$. Given two vertices $[a]_u, [a']_u \in \cup_{i=1}^\gen \Axes(\phi_i,\F)$, there is an edge in $\Gamma$ from $[a]_u$ to $[a']_u$ if and only if there exists $\eta \in \h$ and a label $(i,s)$ for $[a']_u$ such that the circuit $\sigma$ in $G_i$ representing $\eta [a]_u$ has 
a complete splitting rel $G^i_{r(i)}$, at least one term of which (up to a reversal of orientation) has  the form $E^i_s$ or $E^i_s {(w^i_s)}^p \overline E^i_{t}$ where $E^i_s\ne E^i_{t}$ are linear edges in the   linear family determined by $x^i_s$; we say in this case that the triple $(\eta,i,s)$ is a label for the edge from $[a]_u$ to $[a']_u$. Note that if this holds then $\sigma$ grows subpaths of the form $({w^i_s})^q$ for arbitrarily large $q$ under iteration by $f_i$.
\end{definition}

\noindent
\textbf{Remark:} One can show that $\Gamma$ depends only on $\h$, $\F$, and the set $\{\psi_1,\ldots,\psi_\gen\}$, not on the choice of rotationless powers nor on the choice of their \ct\ representatives. We do not need this fact and so we omit its proof.

\begin{lemma}  \label{LemmaNoClosedEdgePaths}  
Adopting Notation~\ref{notn:generators}, if $\h \subset \PG^\F$ then the exponential growth digraph $\Gamma$ has no closed oriented edge paths.
\end{lemma}

The proof of Lemma~\ref{LemmaNoClosedEdgePaths} follows some remarks and a further lemma.

\begin{remark} It is helpful to compare Lemma~\ref{LemmaNoClosedEdgePaths}  to the analogous step in analyzing 
subgroups of mapping class groups of surfaces. Consider a surface $S$, its 
mapping class group $\MCG(S)$, a subgroup $H <\MCG(S)$, and a subsurface $F 
\subset S$ that carries every stable lamination of the Thurston 
decomposition of every element of $\h$. Given $\phi \in H$, the Thurston 
decomposition of $\phi$ might contain a \lq twistor\rq, i.e. a Dehn twist curve $x$, 
that is not contained in  $F$. Given two such twistors $x_1$ and $x_2$ for 
two rotationless mapping classes $\phi_1,\phi_2$, one could define a 
relation (i.e.\ draw an arrow of a directed graph) where $x_1$ is related to 
$x_2$ if there exists $\eta \in \h$ such that $\eta(x_1)$ wraps around $x_2$ under 
iteration of $\phi_2$. Since this happens if and only if $\eta(x_1)$ and $x_2$ 
have nontrivial intersection, this relation is clearly symmetric, and it 
is clearly independent of $\phi_1,\phi_2$. If this relation holds then one 
shows that for sufficiently large $n$, the Thurston decomposition of 
$\phi_2^n\eta$ has a stable lamination that is not contained in 
$F$.

The analogous relation in our current situation is defined by the directed graph $\Gamma$, the definition of which reflects several subtleties of $\Out(F_n)$ not present in mapping class groups. First, the relation need not be symmetric, which is why $\Gamma$ is a digraph rather than an undirected graph. More importantly, the relation does not depend just on the 
axis $x$, it depends on the choice of $\phi \in \h$ for which $x \in \Axes(\phi,\F)$, because the set of conjugacy classes that wrap around $x$ under iteration of $\phi$ depends on $\phi$.
\end{remark}

The following technical lemma (c.f.\ Lemma~5.7.9 of \BookOne) improves the bounded cancellation lemma in a particular situation.  

\begin{lemma}   \label{limited cancellation} Let $\fG$ be a \ct, let $G_r$ be a filtration element, let $\sigma$ be a closed path in $G$, and suppose that the following hold:
\begin{enumerate}
\item If $E$ is a superlinear \neg\ edge with height $>r$ then $f(E) = Eu$ where $u \subset G_r$.
\item $\height(\sigma)> r$.
\item $\sigma$ is not a Nielsen path.
\item The path $\sigma$ is obtained from a circuit that is completely split rel $G_r$ by subdividing between two terms of this splitting.
\end{enumerate}
Then there is a constant $C$, independent of $\sigma$, so that
$f^k_{\#\#}(\sigma^p)$ contains $(f^k_\#(\sigma))^{p-C}$ as a subpath for all  $k\ge 1$ and all $p >C$.
 \end{lemma}
  
\begin{proof}  By symmetry, it suffices to show that there is a constant $C$, independent of $\sigma$, so that if $\alpha \sigma^p$ is a path then $f^k_\#(\alpha \sigma^p) =f^k_\#(\alpha \sigma^{C})(f^k_\#( \sigma)^{p-C}) $ for all  $k\ge 1$ and all $p >C$.


We reduce to the case that every edge above $G_r$ is fixed or linear, as follows.  Suppose that $E'$ is a non-fixed, non-linear edge above $G_r$.   By (1), we may re-order the strata above $G_r$ so that $G_{r+1} = G_r \cup E'$.  Fact~\ref{FactIsATerm} implies that $E'$ is not contained in any Nielsen path.  Item  (1) therefore implies that $E'$ is not crossed by $f(E)$  for any edge  $E \ne E'$.  After re-ordering the strata, we may assume that $E'$ is the highest stratum in the filtration.  By Fact~\ref{FactBasicNEGSplitting}, no copies of $E'$ are ever cancelled when the image of a path is tightened.  Thus $C=1$ works if  $\sigma$ crosses  $E'$.    If $\sigma$ does not cross $E'$ then then  there is no loss in replacing $G_r$ by $G_{r+1}$ which reduces the number of edges above $G_r$.  After repeating this finitely many times, we may assume that every edge above $G_r$ is fixed or linear.  
  
 We complete the proof by induction on $l =\height(\alpha)$.  The case $l \le \height(\sigma)$ follows from Fact~\ref{FactBasicNEGSplitting}.

Suppose now that $l > \height(\sigma)$ and that we have a constant $C_{l-1}$ that works for paths of height $< l$.  We may reduce to the case that $\alpha = E_l \alpha'$ where $\alpha'$ has height $< l$, for if not then we may apply Fact~\ref{FactBasicNEGSplitting} to the copy of $E_l$ or $\overline E_l$ which is closest to $\sigma$, splitting $\alpha \sigma^p$ at the initial point of $E_l$ in that copy. What remains on the right of the split point therefore has one of the two forms $\alpha' \sigma^p$ or $E_l \alpha' \sigma^p$ where $\alpha'$ has height $<l$: in the case of the first form we are done by induction; and in the second case we have completed the reduction.

In what follows $[\tau_1\tau_2 \ldots \tau_m]$ is the path obtained from the concatenation of paths $\tau_i$ by tightening.   Applying $f^k_\#$ to the equation $\alpha \, \sigma^p = E_l \, \alpha' \, \sigma^p$, and using that $f^k_\#(E_l) = E_l u^k_l$ for some (possibly trivial) Nielsen path $u_l$, we get
\begin{align*}
f^k_\#(\alpha \sigma^p) &= \left[ f^k_\#(\alpha) \, f^k_\#(\sigma^p)   \right] \\
  &= \left[E_l \, u^k_l \, f^k_\#(\alpha') f^k_\#(\sigma^p) \right] \\
  &= \left[E_l \, f^k_\#(u^k_l) \, f^k_\#(\alpha') f^k_\#(\sigma^p) \right] \\
  &= E_l \, f^k_\#(u^k_l \alpha' \sigma^p) 
\end{align*}
where we may pull $E_l$ out since $f^k_\#(u^k_l \alpha' \sigma^p)$ has height $<l$.

We claim that there exists $p'$, independent of $k$, $\alpha'$, $\sigma$, and $p$, such that  $[u^k_l \alpha' \sigma^p] = [u^k_l \alpha' \sigma^{p'}]\sigma^{p-p'}$.   Assuming this for now, let $C_l = C_{l-1} + p'$ and   $\nu =  [u^k_l \alpha' \sigma^{p'}]$.  If $p > C_l$ then   
$$
f^k_\#(\alpha \sigma^p)  = E_l f^k_\#(\nu \sigma^q)  
$$
where $q >C_{l-1}$ and the inductive hypothesis completes the proof.



It remains to prove the existence of $p'$. Since $\alpha' \sigma^p$ is already tight, no edges of $\sigma^p$ are cancelled during the  tightening of $u^k_l \alpha' \sigma^p$ unless all of $\alpha'$ is cancelled when $u_l^k\alpha'$ is tightenened, say to a path $\mu$. We are therefore reduced to the case that $\mu$ is a subpath of $u_l^{k'}$ for some $k'$. If $\height(u_l) \le \height(\sigma)$ then $p'=1$ works. If $\height(u_l) > \height(\sigma)$ then at less than one full copy of $u_l$ is cancelled when $\mu \sigma^p$ is tightened. Since $u_l$ takes on only finitely many values, we may choose $p'$ so that the length of $\sigma^{p'}$ is greater than the length of $u_l$.  
\end{proof}
  
Much of the following proof is taken from page 47 of \BookTwo. 

\begin{proof}[Proof of Lemma~\ref{LemmaNoClosedEdgePaths}.] 
%
%
First, in the special case that $\Gamma$ has an oriented edge that begins and ends at the same vertex $[a]_u$, then we shall construct an element of $\h$ with an attracting lamination not supported by~$\F$.  Choose a label $(i,s)$ for $[a]_u$ and an $\eta \in \h$ such that the triple $(\eta,i,s)$ is a label for an edge from $[a]_u$ to itself, in particular $[a]_u = [x^i_s]$, and $x^i_s, E^i_s \not\subset G^i_{r(i)}$. Choosing any homotopy equivalence $g \from G^i \to G^i$ representing $\eta$, the following hold: the circuit $\sigma = g_\#(x^i_s) \subset G^i$ has a complete splitting rel $G^i_{r(i)}$ with respect to the \ct\ $f^i \from G^i \to G^i$; and at least one term of this splitting (up to a reversal of orientation) has the form $E^i_s$ or $E^i_s (w^i_s)^p \overline E^i_{t}$ where $E^i_s\ne E^i_{t}$ are linear edges in the same linear family. Note that $\sigma$ is not a Nielsen path of $f^i \from G^i \to G^i$, $\sigma$ is not contained in~$G^i_{r(i)}$, and $\sigma$ is the circuit freely homotopic to the closed path $g_\#(w^i_s)$. 

By subdividing $\sigma$ at one of the vertices between terms in its complete splitting rel $G_{r(i)}$, we may view $\sigma$ as a closed path with a complete splitting rel $G_{r(i)}$; for each $A \ge 1$ it follows that the closed paths $\sigma^A$ and $g_\#((w^i_s)^A)$ are freely homotopic. Choose $A$ so large that $g_{\#\#}((w^i_s)^A) $  contains the path ${\sigma}^{C_i+2}$  as a subpath where  $C_i$ is the constant of Lemma~\ref{limited cancellation} applied to $f_i$. As observed at the end of Definition~\ref{DefExpGrowthGraph}, we may choose $m$ so that  $(f_i^m)_\#(\sigma)$ contains  $(w^i_s)^A$ a subpath. The homotopy equivalence $f_i^m g$ represents $\mu = \phi^m \eta \in \h$.    

We claim that if a path $\rho\subset G^i$ contains $L$ disjoint subpaths of the form   $(w^i_s)^{A}$ then $(f_i^m)_\#g_\#(\rho)$ contains $2L$ disjoint subpaths of the form $(w^i_s)^{A}$. To see this, note that by Lemma~\trefGM{LemmaDoubleSharpFacts}{ItemDisjointCopies} and the obvious induction argument, disjoint subpaths of the form  $(w^i_s)^{A}$ in $\rho$ determine disjoint subpaths of the form $\sigma^{C_i+2}$ in $g_{\#}(\rho)$  which by Lemma~\ref{limited cancellation} determine disjoint  subpaths of the form   $((f_i^m)_\#(\sigma))^2$ in $(f_i^m)_\#g_\#(\rho)$. Since $(f_i^m)_\#(\sigma)$ contains a subpath of the form $(w^i_s)^{A}$, the claim follows. By Lemma~\ref{LemmaFindingEG} 
it follows that $\mu$ has an attracting lamination not supported by~$\F$, which gives a contradiction in the special case that $\Gamma$ does not contain a closed edge path with only one edge.

\smallskip
 
Having handled the special case, to complete the proof it suffices, by the obvious induction argument, to show that the oriented graph $\Gamma$ is transitive, that is: if there are oriented edges from a vertex $[a]_u$ to $[b]_u$ and from $[b]_u$ to $[c]_u$, then there is an oriented edge from $[a]_u$ to $[c]_u$. Choose $\eta_1,\eta_2 \in \h$ and choose labels $(i,s)$ for $[b]_u$ and $(j,t)$ for $[c]_u$ so that $(\eta_1,i,s)$ is a label for the edge from $[a]_u$ to $[b]_u$ and $(\eta_2,j,t)$ is a label for the edge from $[b]_u$ to $[c]_u$. We shall construct $\eta \in \h$ so that $(\eta,j,t)$ is a label of an edge from $[a]_u$ to $[c]_u$.

Let $\sigma_i$ be the circuit in $G^i$ representing $\eta_1[a]_u$ and let $\sigma_j$ be the circuit in $G^j$ representing $\eta_2[b]_u$. Let $\kappa \subset G^j$ be a closed path obtained by subdividing $\sigma_j$ at a vertex between terms in its complete splitting rel $G_r$. Let $g  \from G^i \to G^j$ be a homotopy equivalence representing~$\eta_2$. Using our choice of labels for the given edges we have $[x^i_s]=[b]_u$, and therefore letting $C$ be the constant of Lemma~\ref{limited cancellation} we may choose $B$ so large that ${g}_{\#\#}((w_s^i)^B)$ contains $\kappa^{C + 3}$ as a subpath. Choose $m_1$ so large that ${f_i^{m_1}}_\#(\sigma_i)$ contains $(w^i_s)^B$ as a subpath.

By Lemma~\ref{limited cancellation}, 
$(f_j^{m} g f_i^{m_1})_\#(\sigma_i)$ contains ${f_j^m}_\#(\kappa^3) = ({f_j^m}_\#(\kappa))^3$ as a subpath for all $m \ge 1$. Choose $m_2$ so large that the circuit 
$\sigma_3 = (f_j^{m_2}g f_i^{m_1})_\#(\sigma_i) \subset G^i$ 
is completely split rel $G_r$, let $\eta =\phi_j^{m_2}\eta_2\phi_i^{m_1}\eta_1 \in \h$ and consider the closed path $\tau = {f_j^m}_\#(\kappa)$. Letting $\sigma_3$ be the representative in $G^j$ of $\eta [a]_u$, it follows that $\sigma_3$ contains $\tau^3$ as a subpath. 

By construction, at least one term in the complete splitting of $\tau$ is (up to reversal of orientation) either $E^j_t$ or an exceptional path whose first edge is $E^j_t$. In particular, there exists $\tau_0$ a subpath of $\tau^3$, and therefore also a subpath of $\sigma_3$, such that
\begin{itemize}
\item  $\tau_0$ begins with $E^j_t$ (in the middle copy of $\tau^3$).
\item  $\tau_0$ ends with an edge in the same linear family as $E^j_t$ but otherwise does not cross any edge in the same linear family as $E^j_t$.
\item  $\tau_0$ is not a Nielsen path.
\end{itemize}
Lemma~\ref{FactIsATerm}~\pref{ItemLinearEdgeTerm} therefore implies that some term in the complete splitting rel $G_r$ of $\sigma_3$ is either $E^j_t$ or an exceptional path whose first edge is $E^j_t$. This completes the proof that there is an edge in $\Gamma$ from $[a]_u$ to $[c]_u$ as desired, labelled by $(\eta,j,t)$.

\end{proof}

\subsection{Bouncing sequences}  
\label{SectionFollowTheBouncingBall}

To complete the proof of Proposition~\ref{PropFCarriesAll} it remains to analyze the situation where the group $\h$ has at least one twistor not supported by~$\F$ and the exponential growth digraph $\Gamma$ has no closed oriented edge paths. The result of this analysis will be to construct, in various cases, an $\h$-invariant proper free factor system properly containing~$\F$.

Throughout this section we fix a finitely generated subgroup $\h \subgroup \IA_n(\Z/3)$ and a proper free factor $\F$ invariant under every element of $\h$ such that $\h$ is irreducible rel~$\F$ and $\h \subset \PG^\F$. Let $\Twist(\h,\F) = \union_{\psi \in \h} \Twist(\psi,\F)$. We may choose a generating set $\{\psi_1,\ldots,\psi_\gen\}$ for $\h$ which has \emph{maximally filling twistors rel $\F$} meaning that the following equation holds:
$$\F_\supp(\union_{i=1}^\gen \Twist(\phi_i,\F)) = \F_\supp(\Twist(\h,\F)) \qquad\qquad (*)
$$
where $\F_\supp(\cdot)$ stands for free factor support. If $(*)$ is not already true, choose a rotationless element $\psi_{\gen+1} = \phi_{\kappa+1} \in \h$ having a twistor not supported by the left hand side of $(*)$, and add $\psi_{\kappa+1}$ to the list of generators. After finitely many additions equation $(*)$ holds, because there is an upper bound to the length of a strictly ascending chain of free factor systems. 

We also adopt Notation~\ref{notn:generators} throughout this section, in particular a rotationless power $\phi_i$ of each $\psi_i$ with representative \ct\ $f_i \from G^i \to G^i$ and associated notation for twistors of $f_i$. Let $\Gamma$ be the exponential growth digraph associated to this data. 

\begin{definition} 
\label{DefBouncingSequence}
Consider $T_0$, a minimal simplicial $F_n$ tree with trivial edge stabilizers and with $\F(T_0) = \F$. For all $i \ge 1$, inductively define $T_i = T_{i-1}\phi_i^\infty$, using either Proposition~\ref{prop:grower} or Proposition~\ref{prop:nongrower} depending on whether or not $\A_i = \A_+(T_{i-1};\phi_i)$---the set of twistors of $\phi_i$ having positive translation length in $T_{i-1}$---is nonempty or empty, where the subscript on $\phi_i$ is taken mod $\gen$. Note that $\A_i \subset \Axes(\phi_i,\F) \subset \Axes(\h,\F)$. We refer to $T_0,T_1,T_2,\ldots$ as the \emph{bouncing sequence} starting with~$T_0$ with respect to $\{\phi_1,\ldots,\phi_\gen\}$.
\end{definition}

The very last line of the proof of the following lemma is where we use apply Lemma~\ref{LemmaNoClosedEdgePaths} and the hypothesis that $\h \subset \PG^\F$.
 
\begin{lemma}\label{LemmaFinitelyMany} For any $T_0$, letting $T_0,T_1,T_2\ldots$ be the bouncing sequence started by $T_0$ with respect to $\{\phi_1,\ldots,\phi_\gen\}$, there exist $0 \le a \le b$ so that the following hold for all $l \ge b$. 
\begin{enumerate}\label{item:Ai is empty}  \item  $\A_l = \emptyset$.
\item \label{item:Ti independent of i}The set of edge stabilizers of $T_l$ is independent of $l$.
\item \label{item:edge stabilizer} If $\langle e\rangle $ is the non-trivial stabilizer of some edge in $T_l$ then $[e]\in \A_a$. 
\end{enumerate}
\end{lemma}
 
\begin{proof} We denote $L_i$ for the translation length function $L_{T_i}$, which associates to each conjugacy class $[c]$ the translation length $L_i[c]$ of $c$ acting on $T_i$.

If $\A_l = \emptyset$ for all $l$ then the edge stabilizers of each $T_l$ are trivial by Proposition~\ref{prop:nongrower} and there is nothing to prove.  We may therefore assume that not all $\A_l$ are empty.  We prove below that only finitely many $\A_l$ are non-empty.  Assuming this for now we complete the proof. 

Let $a$ be the largest value such that $\A_a \ne \emptyset$. Then (1) is satisfied and (3) follows from Propositions~\ref{prop:grower} and \ref{prop:nongrower}.   Since $T_i$ is very small, its  non-trivial edge stabilizers  are primitive infinite cyclic groups and  any two such are equal or have trivial intersection.  Proposition~\ref{prop:nongrower}  therefore implies that there are only finitely many values of $l$ for which the edge stabilizers of $T_l$ are different than the edge stabilizers of $T_{l-1}$. Choosing $b$ greater than the last such value of $l$ establishes (2).
   
It remains to show that only finitely many $\A_l$ are non-empty.

Propositions~\ref{prop:nongrower} and \ref{prop:grower} imply that if $\A_{i+1} = \emptyset$ then    for all $c \in F_n$ there exists $n_{i+1}>0$ so that 
$$
L_{i+1}[c]  = L_{i}(\phi_{i+1}^{n_{i+1}}[c])
$$ 
 and  that if $\A_{i+1} \ne  \emptyset$ then for all $c \in F_n$ and all $\epsilon > 0$ there exists $n_{i+1} >0$ so that 
$$
\abs{L_{i+1}[c] -\frac{1}{n_{i+1}} L_{i}(\phi_{i+1}^{n_{i+1}}[c])} < \epsilon
$$ 

We claim more generally that for all $i < j$,  all $c \in F_n$ and all $\epsilon > 0$ there exists positive integers $n_{i+1},\ldots, n_{j} $ and $N$ so that     
 \begin{equation} \label{finding eta} 
\abs{L_j[c] -\frac{1}{N} L_{i}(\phi_{i+1}^{n_{i+1}}\ldots \phi_{j-1}^{n_{j-1}}\phi_{j}^{n_{j}}[c])} < \epsilon
\end{equation}   
The proof is  by induction on $j\ge i+1$ with the base case $j=i+1$ already established.   Assuming that the claim holds for   $j$, we prove it for $j+1$. 

 If $\A_{j+1} = \emptyset$ apply  Proposition~\ref{prop:nongrower} to choose $n_{j+1}$ so that $$L_{j+1}[c]  = L_{j}(\phi_{j+1}^{n_{j+1}}[c])$$  Applying the inductive hypothesis to $i,j,[c']=\phi_{j+1}^{n_{j+1}}[c]$ and $\epsilon$ to produce $n_{i+1},\ldots, n_{j} $ and $N$, we have 
 $$ \abs{L_{j+1}[c]  -\frac{1}{N} L_{i}(\phi_{i}^{n_{i+1}}\ldots \phi_{j}^{n_{j}}\phi_{j+1}^{n_{j+1}}[c])} = \abs{L_{j}([c']) - \frac{1}{N} L_{i}(\phi_{i}^{n_{i+1}}\ldots \phi_{j}^{n_{j}} [c'])} < \epsilon
 $$

If $\A_{j+1} \ne \emptyset$ apply  Proposition~\ref{prop:grower} to choose $n_{j+1}$ so that 
$$\abs{L_{j+1}[c] -\frac{1}{n_{j+1}} L_{j}(\phi_{j+1}^{n_{j+1}}[c])} < \epsilon/2
$$  
Apply the inductive hypothesis to $i,j,c'=\phi_{j+1}^{n_{j+1}}[c]$ and $\epsilon' = \epsilon/2$ to produce $n_{i+1},\ldots, n_{j} $ and $N'$.  Letting $N = n_{j+1}N'$,  we have
\begin{align*}
 \bigg| L_{j+1}[c]  &- \frac{1}{N} L_{i}(\phi_{i}^{n_{i+1}} \ldots \phi_{j}^{n_{j}}\phi_{j+1}^{n_{j+1}}[c]) \bigg|    \\  
& \le \abs{L_{j+1}[c] -\frac{1}{n_{j+1}} L_{j}( [c'])}\  +   \frac{1}{n_{j+1}} \abs{L_{j}([c']) - \frac{1}{N'} L_{i}(\phi_{i+1}^{n_{i+1}}\ldots \phi_{j}^{n_{j}} [c'])}  \\ &   < \epsilon/2 + \epsilon/(2n_{j+1}) \le \epsilon
\end{align*}
This completes the proof of the claim.
  
Suppose now that $i < j$,   that  $[x^j_s]  \in \A_j$ and that $\A_{i} \ne \emptyset$.  Applying  \pref{finding eta} with $[c] = [x^j_s]$ and  $\epsilon = L_j[c]$  we have   
$$L_{i}((\phi_{i+1}^{n_{i+1}}\ldots \phi_{j}^{n_{j}})_\#([x^j_s])) 
$$
Since $L_i$ is invariant under the action of $\phi_i$,
$$L_{i}((\phi_i^n\phi_{i+1}^{n_{i+1}}\ldots \phi_{j}^{n_{j}})_\#([x^j_s]))  > 0
$$
 is a positive number independent of $n$. Choosing $n=n_i$ sufficiently large, and letting    $\eta_{ij} = \phi_{i}^{n_{i}}\ldots \phi_{j-1}^{n_{j-1}}\phi_{j}^{n_{j}}$, we conclude that the circuit $\sigma \subset G_i$ that realizes $\eta_{ij}([x^j_s])$ has a  complete splitting rel $G_r$ and that $L_i(\sigma) > 0$. By  item (3) of Proposition~\ref{prop:grower} and item (2) of Lemma~\ref{LemmaGrowerFormula},  some term in the complete splitting of $\sigma$ has  the form  $E^{i}_q$ or $E^{i}_q{w^{i}_q}^p\overline E^{i}_t$ where $E^{i}_q$ is an edge in $G^{i}$ with  $[x^{i}_q]  \in \A_{i}$ and $E^{i}_t \subset G^{i}$ an edge in the same linear family as $E^{i}_q$.    
  
  
To summarize, we have shown that if $i<j$ and if $\A_j$ and $\A_{i}$ are non-empty then for all $[x^j_s] \in \A_j \subset \Twist(\phi_j,\F)$ there exists $[x^i_q] \in \A_{i} \subset \Twist(\phi_i,\F)$ such that $\Gamma$ has an oriented edge from the vertex $[x^j_s]$ to the vertex $[x^i_q]$, that edge being labelled by the triple $(\eta_{ij},i,q)$. It follows that if there are $N$ values of $l$ for which $\A_l$ is non-empty, then there is an oriented edge path of length $N-1$ in $\Gamma$. Lemma~\ref{LemmaNoClosedEdgePaths} implies that $\Gamma$ has no closed oriented edge paths, and so there is an upper bound to the length of all oriented edge paths in $\Gamma$, implying that $\A_l$ is non-empty for only finitely many values of~$l$.
\end{proof}

\begin{proof}[Proof of Proposition~\ref{PropFCarriesAll}]  
In light of Lemma~\ref{LemmaNonLinearCase}, it suffices to show that $\Axes(\h,\F) =  \emptyset$. Pick a generating set $\{\psi_1,\ldots,\psi_\gen\}$ for $\h$ with maximally filling twistors rel~$\F$, pick rotationless powers $\{\phi_1,\ldots,\phi_\gen\}$, and assume Notation~\ref{notn:generators} for representative \cts\ and their twistors. Since the sets $\Axes(\h,\F)$ and $\union_{j=1}^\gen \Axes(\phi_j,\F)$ have the same free factor support, 
it suffices to show that $\Axes(\phi_j,\F) = \emptyset$ for $1 \le j \le \gen$.

Choosing $T_0$ to be a very small simplicial $F_n$ tree with trivial edge stabilizers and with \hbox{$\F(T_0) = \F$}, let $T_0,T_1,T_2,\ldots$ be the bouncing sequence of $T_0$ as in Definition~\ref{DefBouncingSequence}, and let $\A_i = \A_+(T_{i-1};\phi_i)$. Let $a \le b$ be as in Lemma~\ref{LemmaFinitelyMany}. If $\langle e \rangle$  stabilizes  an edge in $T_b$ then it stabilizes an edge in $T_i$ for all $i \ge b$.  Since $T_i$ is $\phi_i$-invariant it follows that $[e]$ is $\phi_i$-invariant, and hence $\psi_i$-periodic. Since $\psi_i \in \IA_n(\Z/3)$, Theorem~\ref{ThmPeriodicConjClass} implies that $[e]$ is $\psi_i$-invariant.  As this holds for each $\psi_j$, $j=1,\ldots,\gen$, it follows that $[e]$ is $\h$-invariant. The smallest free factor system $\F'$ that carries $\F$ and $[e]$ is $\h$-invariant. Lemma~\ref{LemmaFinitelyMany}~\pref{item:edge stabilizer} implies that  $[e]$ and $\F$ are carried by a proper subgraph of $G^a$ and hence that $\F'$ is proper. By irreducibility,  $\F' = \F$. Thus  $[e]$ is carried by $\F$  in contradiction to the fact (Lemma~\ref{LemmaFinitelyMany}~\pref{item:edge stabilizer}) that $[e] \in \A_a$. We conclude that  $T_b$, and hence $T_i$ for all $i \ge b$,  have trivial edge stabilizers.

For each $i \ge b$, choose a finite collection $\C_{i}$ of conjugacy classes    so that $\F(T_{i})= \cffs(\C_{i})$. Since $\A_{i+1} = \emptyset$ and each $\alpha \in \C_{i+1}$ is elliptic in $T_{i+1}$, Proposition~\ref{prop:nongrower} implies that there exists $K > 0$ so that  $\phi_{i+1}^K(\alpha)$ is elliptic in $T_i$ and hence carried by $\F(T_i)$ for each $\alpha \in \C_i$. It follows that for each $i \ge b$ we have:
$$ \F(T_{i+1}) = \phi_{i+1}^{K}(\F(T_{i+1}))= \phi_{i+1}^{K}(\cffs(\C_{i+1})) =\cffs(\phi_{i+1}^K(\C_{i+1})) \sqsubset \F(T_i)
$$
We conclude that:
\begin{itemize}
\item $\{\F(T_i): i \ge b\}$ is a nested sequence of free factor systems and is hence eventually constant, say $\F(T_i) $ is independent of $i$ for all $i \ge c \ge b$.   
\end{itemize}
For $i \ge c$, each element of $\Axes(\phi_{i+1},\F)$ is elliptic in $T_i$  (because $\A_{i+1} = \emptyset$) and hence carried by $\F(T_i) = \F(T_c)$. This proves that the smallest free factor system carrying $\F$ and $\cup_{j=1}^\gen\Axes(\phi_j,\F)$ is proper. By irreducibility, $\cup_{j=1}^\gen\Axes(\phi_j,\F) = \emptyset$  as desired.
\end{proof}

\bigskip   

\section{Finding Nielsen pairs: Proof of Proposition~\ref{PropNielsenPairsExist}} \label{SectionNielsenPairsExist}

Recall that $\Lam(\phi)$ is the union of the twistors, eigenrays and attracting laminations of~$\phi$. In this section we prove

\medskip

\noindent {\bf Proposition~\ref{PropNielsenPairsExist}} \emph{Suppose that $\k \subgroup \IA_n(\Z/3)$ is generated by a finite number of rotationless elements, that $\F$ is a proper $\k$-invariant free factor system, and that $\F$ carries $\Asym(\phi)$ for each rotationless $\phi \in \k$. Then there is a Nielsen pair for $\k$ associated to~$\F$.  Moreover, following  the notation of Definition~\ref{DefNielsenPairs}, one may choose the tree $T$ and Nielsen pair $(V,W)$ so that  $\ti \gamma$  is an edge of~$T$.} 

\medskip

A proof in the case that  $\F =\emptyset$ is given in sections 5.4 and 5.5 of \BookTwo.   Most of that proof applies in our context. The main exception is Lemma~5.14, which we replace   with Lemma~\ref{choose N} below. Given the amount of notation that has to be introduced, it made sense to us to give a complete self-contained proof of the main statement of the proposition. For the ``Moreover'' sentence, we depend more heavily on the relevant two pages of \BookTwo.

We begin by recalling the setup of section 5.4 of \BookTwo, a setup which using our earlier results we are able to mimic.

\begin{notn} \label{NotationSectionSeven} Let $\phi_1,\ldots,\phi_\gen$  be rotationless generators of $\k$. Knowing that $\F$ carries $\Asym(\phi_i)$ and is $\phi_i$-invariant, we may may apply Lemma~\ref{lem:cofinal}, allowing us to choose, for each $1 \le i \le \gen$,  a \ct\ $f_i :G^i \to G^i$ representing $\phi_i$ and a filtration element $G^i_{r(i)}$ realizing $\F$  such that each stratum above $r(i)$ is a single edge    $E$  satisfying $f_i(E) =E$ or  $f_i(E) =E\cdot u$ for some   closed path $u \subset G^i_{r(i)}$.    

For each $1 \le i \ne j \le \gen$,  let  $h_{ij} : G^i \to G^j$ be a homotopy equivalence that preserves the markings, maps vertices to vertices and restricts to a homotopy equivalence $G^i_{r(i)} \to G^j_{r(j)}$.  
  
The non-trivial elements of $\k$ that we consider are compositions of positive iterates of the generators $\phi_1,\ldots,\phi_\gen$.  Such elements are realized by  homotopy equivalences between the various marked graphs that are compositions of the $h_{ij}$'s and iterates of the $f_i$'s.    For example,   $\phi_2^9\phi_1^4$ is realized by $h_{23}f_2^9 h_{12}
f_1^4 :G_1 \to G_3$.

  
A circuit $P \subset G^i$ that that crosses edges in both  $G^i_{r(i)}$ and its complement has a \emph{cyclic decomposition} 
$$P = \nu_0 \, H_1 \, \nu_1 \, H_2 \ldots H_p
$$ 
where the $\nu_j$'s are the maximal, possibly trivial, subpaths in $G^i_{r(i)}$.  Thus each \emph{vertical element} $\nu_j$ is a (possibly trivial) path in $G^i_{r(i)}$   and each \emph{horizontal element} $H_j$ is  a  non-trivial path that intersects  $G^i_{r(i)}$ exactly in its endpoints. Note that all edges in a horizontal element $H_j$, other than perhaps the first and last, are fixed and that $f_i(H_j) = u_j H_j v_j$ for some (possibly trivial) paths $u_j,v_j \subset G^i_{r(i)}$. 

Given a subpath $\sigma = \nu_{a}H_a \ldots H_b \nu_b$  of $P$, choose a lift $\ti \sigma$ to the universal cover $\wt G^i$.  Let $\wt C_{s}$ and $\wt C_t$ be the components of  the full pre-image $\wt G^i_{r(i)}$ of  $G^i_{r(i)}$ that contain the initial and terminal endpoints of $\ti \sigma$ respectively and let $V$ and $W$ be the corresponding elements of $\wt\V$ (see Definition~\ref{DefNielsenPairs}). The element $[[V,W]]$ of $\V^{(2)}$   is well defined and  \emph{ determined by $\sigma$}.
\end{notn}  

 \noindent{\bf Strategy and Outline.} Our strategy is to show that either  there is a Nielsen pair for $\k$ satisfying the properties of the proposition or there is an element $\theta \in \k$ and a conjugacy class $[a]$ such that the number of horizontal elements in the cyclic decomposition of the circuit in $G^1$ representing $\theta^l([a])$ grows exponentially in $l$. Applying Lemma~\ref{LemmaFindingEG} we obtain a lamination in $\L(\theta)$ not carried by $[\pi_1 G^1_{r(1)}] = \F$, contradicting the hypothesis and thus completing the proof of Proposition~\ref{PropNielsenPairsExist}.

Corollary~\ref{CorInactiveIsNielsen} gives a sufficient condition for a given element  of $\V^{(2)}$ to be a Nielsen pair for~$\phi_i$.  This is promoted in Lemma~\ref{event}  to a sufficient condition for a given element  of $\V^{(2)}$ to be a Nielsen pair for $\k$. Under the assumption that no Nielsen pair exists, and so in particular all of these sufficient conditions fail, the definition of $\theta$ is given in Remark~\ref{RemarkMoreExplicitStrategy} following the statement of Lemma~\ref{choose N}. The proof of Proposition~\ref{PropNielsenPairsExist} that finishes this section gives the exponential growth property for $\theta$ under the assumption of no Nielsen pairs.

\medskip
Having made no assumptions on how elements of $\k$ act on $\F$, there is not much we can say about the  vertical elements of cyclic decompositions. We will in fact simply treat them as place holders which---some of the time, when we can manage to apply bounded cancellation, see Lemma~\ref{LemmaProperHoldsItsPlace}---prevent adjacent horizontal elements from canceling each other. This motivates the following definition.

\begin{definition} \label{defn:place holder}   Suppose that $P $ is  a circuit in $G^i$ that crosses edges in both $G^i_{r(i)}$ and its complement, that $\nu$ is a vertical element in the cyclic decomposition of $P$ and that  $g :G^i \to G^j$ is a homotopy equivalence such that $g(G^i_{r(i)}) =  G^j_{r(j)}$ where  $1 \le i,j \le \gen$.  Choose a path  $\ti \nu$ that is a lift of $\nu$, a line   $\ti P  $ that is a lift of $P$ and that contains $\ti \nu$  and a lift $\ti g:\wt G^i \to \wt G^j$.    We say that  $\ti \nu$ is a vertical element of $\ti P$ and that the element $V$ of $\wt\V$ corresponding to the  component of  $\wt G^i_{r(i)}$ that contains $\ti \nu$  is \emph{has non-trivial intersection $\ti \nu$ with $\ti P$ }.    Since $\ti g$ induces a bijection between the set of components of $\wt G^i_{r(i)}$ and the set of components of $\wt G^j_{r(j)}$, it induces a self map $\ti g_\#$ of $\wt\V$ (which agrees with  the self map of $\wt\V$ induced by the automorphism of $F_n$ corresponding to $\ti g$). 


Recall that a subpath $\ti \tau$ of  $\ti P$ is pre-trivial  if $\ti g_\#(\ti \tau)$ is the trivial path or equivalently if the $\ti g$-images of the endpoints of $\ti \tau$ are equal.   Note that if $\ti \tau$ is maximal (with respect to inclusion) among all pre-trivial subpaths of $\ti P$ then $\ti g_\#(\ti \tau)$ is a point in $\ti g_\#(\ti P)$.

Suppose that  $\ti \nu$ is a vertical element of $\ti P$ and that $\Gamma$ is the component   of $\wt G^i_{r(i)}$ that contains $\ti \nu$.  We say that   $\ti \nu$    \emph{holds its place} in $\ti P$ \emph{with respect to $\ti g$} if every  pre-trivial subpath  of $\ti P$ that contains $\ti \nu$  has $\ti g$-image in   $\ti g(\Gamma)$.  In this case, $\ti g_\#(\ti P)$ has non-trivial intersection, say $\ti \nu'$, with $\ti g( \Gamma)$ because the tightened image of every pre-trivial subpath of $\ti P$ that intersects $\ti \nu$ is contained in $\ti g(\Gamma)$.  We say that $\ti \nu'$ is \emph{determined by $\ti \nu$ and $\ti g$}. 
We also say that $\nu$ \emph{holds its place in} $P$ \emph{with respect to $g$} and \emph{determines} $\nu'$, the projection of $\ti \nu'$; we write $\nu \to \nu'$ if $g$ and $P$ are understood. 
\end{definition}  

\begin{lemma} \label{circular order}  With notation as in Definition~\ref{defn:place holder}, the 
determination relation $\nu \rightarrow \nu'$ between the subset of 
vertical elements of P that hold their place with respect to $g$ and the subset of vertical 
elements in $g_\#(P)$ that they determine is a bijection that preserves  circular order. 
\end{lemma}

\begin{proof}  We continue with the notation of Definition~\ref{defn:place holder}.   List the vertical elements that hold their place in order along $\ti P$ as $\ldots \ti \nu_{-1}, \ti\nu_0, \ti \nu_1, \ti \nu_2,\ldots$ and let $\Gamma_i$ be the component of $\wt G^i_{r(i)}$ that contains $\ti \nu_i$.    For each $i$ let $\ti \tau^-_i$ be the  maximal pre-trivial path that contains the initial endpoint of $\ti \tau_i$ and let $\ti \tau^+_i$ be the maximal  maximal  pre-trivial path that contains the terminal endpoint of $\ti \tau_i$.   We claim that   $\ti \tau^-_i \cap \ti \tau^+_{i+1} = \emptyset$.    The lemma then follows from the fact that $\ti g_\#(\ti \tau^-_i) \in \ti g(\Gamma_i)$ precedes $\ti g_\#(\ti \tau^+_{i+1}) \in \ti g(\Gamma_{i+1})$.

To prove the claim, note that $\ti \tau^-_i$ can not contain or have terminal endpoint in $\ti \nu_{i+1}$.  If it did, it would have $\ti g$-image that is contained in $\ti g(\Gamma_i)$ and intersects $\ti g(\Gamma_{i+1})$ which is impossible.  Similarly $\ti \tau^+_{i+1}$ can not contain or have initial endpoint in $\ti \nu_{i}$.  We are therefore reduced to considering the case that   $\ti \tau^-_i$ contains  $\ti \nu_{i}$ and $\ti \tau^+_{i+1}$ contains  $\ti \nu_{i+1}$.  In this case  $\ti g(\ti \tau^-_i) \subset \ti g(\Gamma_i)$ and $\ti g(\ti \tau^+_i) \subset \ti g(\Gamma_{i+1})$  so the claim is clear.  
\end{proof}




In the next two lemmas we consider the determination relation $\nu \to \nu'$ of Definition~\ref{defn:place holder} in various situations: the general situation in Lemma~\ref{LemmaProperHoldsItsPlace}; the special situation of powers of $f_i$ in Lemma~\ref{the f case}.
      
Let $\abs{\nu}$ be the combinatorial length of the edge path $\nu$.  

\begin{lemma}\label{LemmaProperHoldsItsPlace}   Assume notation as in Definition~\ref{defn:place holder}.   
\begin{enumerate}
\item  \label{item:holds place}  There is a constant $C=C(g)$, independent of $P$, so that if $\abs{\nu} \ge C$ then $\nu$ holds its place in $P$ with respect to $g$. 
\item \label{item:lower bound} For all $L$ there exists   $K=K(g,L) \ge C$,  independent of $P$,  so that $\abs{\nu} \ge K \implies  \abs{\nu'} \ge L$.   
\item \label{item:upper bound}  For all $K\ge C$ there exists    $L=L(g,K)$,  independent of $P$,  so that $\abs{\nu} \le K \implies  \abs{\nu'} \le L$.
\end{enumerate}
\end{lemma} 

 \begin{proof} Let $C_1 =2\bcc(g)$.  It follows immediately from the definitions and bounded cancellation (Lemma~\ref{FactBCCVerySmall}) that if  $\abs{g_\#(\nu)} > C_1$  then $\nu$ holds its place in $P$ with respect to $g$ (in fact $\ti g( \ti P^+_{\ti \nu})$ and $\ti g( \ti P^-_{\ti \nu})$ are disjoint) and  $ \abs{\nu'} \ge \abs{g_\#(\nu)} -C_1 $. Since $\abs{g_\#(\nu)} \to \infty $ as $ \abs{\nu} \to \infty$, we may choose $C$ so that  $\abs{g_\#(\nu)} > C_1$ whenever $\abs{\nu} \ge C$.   This proves \pref{item:holds place}; item~\pref{item:lower bound} now also follows from the fact  that $\abs{g_\#(\nu)} \to \infty $ as $ \abs{\nu} \to \infty$. For~\pref{item:upper bound} choose a homotopy inverse $g':G^j \to G^i$ for $g$ satisfying $g'(G^j_{r(j)}) = G^i_{r(i)}$.   Applying \pref{item:holds place} to $g'$, we see that  if $\abs{\nu'}$ is sufficiently large then $\nu'$ holds its place in $g_\#(P)$ with respect to $g'$.  By construction the vertical element of $P$ that it determines is $\nu$.  Item~\pref{item:upper bound} therefore follows from \pref{item:lower bound} applied to~$g'$.  
 \end{proof}

\begin{corollary} \label{C}  Assume that $h_{ij} : G^i\to G^j$, $1 \le i\ne j \le \gen$, are as in Notation~\ref{NotationSectionSeven}.   Then for all $M$ there exist positive constants $C_0\le C_1\le C_2 \le \ldots \le C_{M}$ so that the following hold for all $1 \le i\ne j \le \gen$ and  $1 \le m \le M$.
\begin{enumerate}
\item \label{C0}If $\nu$ is a vertical element of a circuit $P$ in   $G^i$ and ${\nu} \ge C_0$ then $\nu$ holds its place in $P$ with respect to $h_{ij}$ determining a vertical element $\nu'$ in ${h_{ij}}_\#(P)$.
\item \label{CM}If $\abs{\nu} \ge C_{m}$ then $\abs{\nu'} \ge C_{m-1}$.
\item\label{not too short}  There is a circuit $P \subset G_1$ crossing edges in both $G^1_{r(1)}$ and its complement such that each vertical element $\nu$ of $P$ satisfies $\abs{\nu} \le C_1$.
\end{enumerate}
\end{corollary}

\begin{proof}  Keeping in mind that we are only considering finitely many maps $h_{ij}$, the existence of $C_0$ satisfying \pref{C0} follows from Lemma~\ref{LemmaProperHoldsItsPlace}~\pref{item:holds place}.  By increasing $C_0$ we may assume that \pref{not too short} is satisfied by any $C_1 \ge C_0$.   The existence of $C_m$ for $m\ge 1$ satisfying \pref{CM} follows from Lemma~\ref{LemmaProperHoldsItsPlace}~\pref{item:lower bound}  by the obvious induction argument.
\end{proof}

In the next lemma we assume that  $g =f_i^N$ for some $N\ge 1$ and show that for all $P$, every vertical element $\nu$ holds its place in $P$.   In this special case, we denote the vertical element that $\nu$ determines by   $\nu^{(N)}$ and say that a vertical element $\nu$ of a circuit $P \subset G^i$ is \emph{inactive} if $\nu^{(N)}$  is independent of $N$ and is \emph{active} otherwise.  Note that  a trivial $\nu$ can be active and that these definitions   depend on $P$.  The terms active and inactive only apply in the special case that $g =f_i^N$.

\begin{lemma} \label{the f case} 
  Suppose that $P = \nu_0H_1\nu_1H_2\ldots H_p$ is the cyclic decomposition of a circuit $P$ contained in   $G^i$ for some $1 \le i \le \gen$. Then the following hold  for all $N\ge 1$. 
\begin{enumerate}
\item \label{preserves horizontal elements}The  cyclic decomposition of ${f_i}^N_\#(P)$ is given by ${f_i}^N_\#(P) = \nu^{(N)}_0H_1\nu^{(N)}_1H_2\ldots \nu^{(N)}_{p-1}H_p$ where  $\nu_j^{(N)}$ is the vertical element of ${f_i}^N_\#(P)$ determined by  $\nu_j$ and $f_i^N$.      In particular, each vertical element $\nu_j$ holds its place in $P$ with respect to $f_i^N$ and $\nu_j \rightarrow \nu_j^{(N)}$ induces a bijection between the vertical elements of  $P$ and the vertical elements of ${f_i}^N_\#(P)$.
 \item\label{grows with N}  For all $K_1$ and $K_2$  there exists $N' = N'(f_i)$, independent of $P$ and $i$, so that if $\nu_j$   is active    and  $|\nu_j| \le K_1$ then $ |\nu_j^{(N)}| >K_2$ for all $N \ge N'$.
\end{enumerate}
\end{lemma}

  \begin{proof}     There is no loss in assuming that all of our circuits $P$ are contained in a single $G^i$.   To simplify notation we drop the $i$ superscripts and subscripts. 
  
 Considering $H_j$ as an edge path, denote the first and last edges   by $\alpha_j$ and $\omega_j$  respectively.  Recall from Notation~\ref{NotationSectionSeven} that all other edges of $H_j$ are fixed by $f$ and that there are (possibly trivial) subpaths $u_j, v_j \subset G_r$ such that $f(H_j) = u_j H_j v_j$. Fact~\ref{FactBasicNEGSplitting} implies that ${f_i}^N_\#(P) = \nu^{(N)}_0H_1\nu^{(N)}_1H_2\ldots \nu^{(N)}_{p-1}H_p$ for some vertical elements $\nu^{(N)}_j$.
 
 To be more explicit, we  consider  four cases depending on the orientation of $\omega_j$ and $\alpha_{j+1}$.  
  \begin{itemize}
  \item If $\omega_j = E_a$ and $\alpha_{j+1} = \overline E_b$ then $\nu_j^{(N)}$ is the path obtained from $f^N_\#(E_a\nu_j \overline E_b)$ by removing the initial $E_a$ and the terminal $\overline E_b$.   In this case let $\sigma = E_a\nu_j \overline E_b$.
  \item  If $\omega_j = \overline E_a$ and $\alpha_{j+1} = \overline E_b$ then $\nu_j^{(N)}$ is the path obtained from $f^N_\#(\nu_j \overline E_b)$ by removing the  terminal $\overline E_b$. In this case let $\sigma = \nu_j \overline E_b$.

  \item  If $\omega_j = \overline E_a$ and $\alpha_{j+1} =  E_b$ then $\nu_j^{(N)} = f^N_\#(\nu_j )$.  In this case let $\sigma =  \nu_j  $.

  \item   If $\omega_j =  E_a$ and $\alpha_{j+1} =  E_b$ then $\nu_j^{(N)}$ is the path obtained from $f^N_\#(E_a \nu_j ) $ by removing the  initial $ E_a$. In this case let $\sigma = E_a\nu_j $.

\end{itemize}

 In all the cases, each edge that is cancelled when ${f_i}^N(P)$ is tightened to ${f_i}^N_\#(P)$ is contained in $G_r$.  This implies that each $\nu_j$ holds its place in $P$ with respect to $f_i^N$ and so completes the proof of \pref{preserves horizontal elements}.
    
Also in all the cases, $\nu_j$ is active if and only if $\sigma$ is not  a Nielsen path. In the active case, $\sigma$ is not a periodic Nielsen path (\recognition\ Lemma~4.13, Fact~\refGM{FactPNPFixed}) so the length of $\nu_j^{(N)}$ goes to infinity with $N$.  Item \pref{grows with N} now follows  from the fact that $\sigma$ takes on only finitely many values for  $|\nu_j| \le K_1$ .  
\end{proof}

 The following immediate corollary of Lemma~\ref{the f case} relates Nielsen pairs to inactive vertical elements.

\begin{cor}   \label{CorInactiveIsNielsen} Suppose that $P = \nu_0H_1\nu_1H_2\ldots H_p $ is the cyclic decomposition of a circuit in $G^i$  for some $1 \le i \le \gen$ and that for some $0 \le a <b\le p$ and    all $a+1 \le j \le b-1$   , $\nu_j$ is inactive.   Then the element $[[V,W]] \in \V^{(2)}$ determined by $\sigma =  \nu_{a} H_{a+1} \ldots  H_{b}  \nu_{b}$   is a Nielsen pair for $\phi_i$. 
\end{cor}

 The folowing result is used to detect Nielsen pairs for $\k$. 
 
  \begin{lemma} \label{induces same pairs}     Suppose that  vertical elements $\nu_a,\nu_b$ of a circuit $P \subset G^i$  hold their place in $P$ with respect to $h_{ij} :G^i \to G^j$, determining vertical elements  $\nu'_{a'},\nu'_{b'}$ of $P' = {h_{ij}}_\#(P) \subset G^j$.      Then the subpath of $P$ beginning with $\nu_a$ and ending with $\nu_b$ determines the same element of $\V^{(2)}$ as the subpath  of $P'$ beginning with $\nu'_{a'}$ and ending with $\nu'_{b'}$.
 \end{lemma} 
 
 \begin{proof}  This is an immediate consequence of the definitions and the fact that $h_{ij}$ preserves the markings  and restricts to a homotopy equivalence  $G^i_{r(i)} \mapsto G^j_{r(j)}$.
 \end{proof}

  We now come to our main construction.
 
 \begin{notn} \label{notn:M} For the remainder of the section we let $M =7\gen$ and we let $C_0\le C_1\le C_2 \le \ldots \le C_{M}$ be the positive constants produced by Corollary~\ref{C}. 
 \end{notn}
    
\begin{definition}\label{generation}  Suppose that  positive integers $N_1,\ldots, N_M$  have been chosen. When considering $\phi_i, f_i, h_{ij}, G^i$  we take $i$ and $j$  mod $\gen$;  when considering  $N_s$ we take $s$ mod~$M$.     
  
Choose, once and for all, a circuit $P_1$ in $G^1$ that  crosses edges in both $G^1_{r(1)}$ and its complement  and such that  $$ \abs{\nu} \le C_1$$  for each vertical element $\nu$ of $P_1$.  The existence of $P_1$ is guaranteed by Corollary~\ref{C}~\pref{not too short}. Inductively define the infinite sequence $P_s$ in $G_s$ ($s \ge 1$) by  
$$P_{s+1} =(h_{s,s+1}f_s^{N_s})_\#(P_s)
$$
We call this the  \emph{descendant sequence} determined by  $N_1,\ldots, N_M$.  (We suppress the dependence of the descendant sequence on $P_1$ because  $P_1$ is now a fixed parameter.)
  
By  item \pref{preserves horizontal elements} of  Lemma~\ref{the f case},    each      vertical element  $\nu$ of $P_s$  determines a vertical element    $\nu^{(N_s)}$ in  ${f_s^{N_s}}_\#(P_s)$.    If $|\nu^{(N_s)}| \ge  C_0$ then say that  $\nu$ \emph{stays alive for at least one generation}, in which case it follows from Corollary~\ref{C}~\pref{C0} that   $\nu^{({N_s})}$ holds its place with respect to $h_{s, s+1}$  and determines a  vertical element $\nu^{(N_s)} \to \nu'$ in   $P_{s+1}$.  We also say  that     \emph{$\nu'$ is the first successor} to $\nu$ and  that $\nu$ \emph{ gives rise to $\nu'$}. 
        For example, if $\nu$ is inactive then $\nu$ stays alive for at least one generation if and only if $\abs{\nu} \ge C_0$.   

If  $\nu$ gives rise to $\nu'$ and  $\nu'$ gives rise to $\nu''$ then we say that $\nu$ \emph{stays alive for at least two  generations} and that  $\nu''$ is the \emph{second successor} to $\nu$.  This can be iterated in the obvious way. 
\end{definition}


The following  lemma allows us to trace the generations backward.

\begin{lemma}\label{trace back}  If  $\nu'$ is a vertical element of $P_{s+1}$ and $\abs{\nu'} \ge C_1$ then there is a vertical  element $\nu$ of $P_s$ that gives rise to $\nu'$.
\end{lemma}

\begin{proof}   Since $f^{N_s}_\#(P_s) = (h_{s+1,s})_\#(P_{s+1})$,  Corollary~\ref{C} implies that $\nu'$ determines a vertical element of $f^{N_s}_\#(P_s) $ with length $\ge C_0$.   Item \pref{preserves horizontal elements} of Lemma~\ref{the f case} implies that this vertical element  is
$\nu^{(N_s)}$  for some vertical element $\nu$ of $P_s$.   The determination relation induced by $h_{s+1,s}$ is a bijection by   Lemma~\ref{circular order}  so $\nu^{(N_s)} \to \nu'$ with respect to $h_{s,s+1}$.  Thus $\nu$ gives rise to $\nu'$. 
\end{proof}

 
 \begin{lemma}  \label{defn of D} For all $B>0$ there exists $D(B)> 0$, independent of the $N_s$'s,  so that if $\nu$ is an inactive vertical element of a circuit $P$ in some  $ G^i$ and if  $\abs{\nu} \ge D(B)$ then $\nu$ gives rise to $\nu'$ and $\abs{\nu'} > B$.
 \end{lemma}
 
 \begin{proof}  This is an immediate consequence of the definitions, Lemma~\ref{LemmaProperHoldsItsPlace}\pref{item:lower bound} and the fact that there are only finitely many marking homotopy equivalences $h_{ij}$.
 \end{proof}
 
 \begin{notn}     For any $B>0$ let $D_0(B) = B$ and then inductively define $D_l(B) = D(D_{l-1}(B))$.
 \end{notn}


\begin{lemma}  \label{choose N}  
The positive integers $N_1,\ldots, N_M$ of Definition~\ref{generation} can be chosen so that, letting $P_1,P_2,\ldots$ be the descendent sequence determined by $N_1,\ldots, N_M$,  for any $s \ge 1$ and  any active vertical element $\nu$ in $P_s$, if $\abs{\nu} \le C_M$ then $\nu$ stays alive for at least $M$ generations.
\end{lemma}

\begin{remark}\label{RemarkMoreExplicitStrategy}
We can now be more explicit about our strategy  for proving Proposition~\ref{PropNielsenPairsExist}. We will assume that that there are no Nielsen pairs for $\k$ associated to~$\F$ and show that $\theta= \phi_M^{N_M}\cdots\phi_2^{N_2}\phi_1^{N_1}$   must have at least one attracting lamination not  carried by $\F$.   In very brief outline, using the descendant sequence of Lemma~\ref{choose N}, we will study how the set of active vertical elements in $P_s$ grows as $s$ increases: either the growth is sufficiently moribund that we can detect a Nielsen pair for $\k$ associated to~$\F$, or it is sufficiently rapid that we can detect the desired attracting lamination for $\theta$.
\end{remark}

\begin{proof}   Set $B_0 =1$ and $s_M = s $ mod $ M \in \{1,\ldots,M\}$.   
 
 In addition to $N_1,\ldots, N_M$, we will define constants $A_1,...,A_M$ and  $B_1,...,B_M$ satisfying the following 
properties for each $1 \le m \le M$:  
\begin{enumerate}
\item [($1_m$)] If  $1 \le s_M  \le m$ and if  $\nu$ is an active vertical element in  $P_s$  such that $\abs{\nu} \le C_M$  then $\nu$ stays alive for at least $ m+1-s_M$ generations.  The $(m+1-s_M)$ successor $\mu$ to $\nu$ satisfies  $A_m \ge  | \mu| \ge D_{M-m}(B_{s_M-1})$.
\item [($2_m$)]  If   $\eta \subset G^1$ is a vertical element  of  $P_{jM+1}$   for some $j$, and if   $|\eta| \ge B_{m}$, then  $\eta$ stays alive  for at least $m$ generations.   As the length of $\eta$ goes to $\infinity$, the length of the $m^{th}$ successor to $\eta$ goes to infinity.
\end{enumerate}

 The  lemma follows easily from the  $m=M$ case:   by ($1_M$), $\nu$ stays alive for at least $ M+1-s_M$ generations and  its $(M+1-s_M)$ successor $\mu$  satisfies  $  | \mu| \ge D_0(B_{s_M-1}) =  B_{s_M-1}$. Since  $s+(M +1 -s_M)$ mod $M = 1$, item  ($2_M$)  implies that $\mu$ stays alive for at least $s_M-1$ generations and hence that $\nu$ stays alive for at least $  M$ generations.  

The  constants $N_m,A_m$ and $B_m$ are chosen inductively.  For the base case $m=1$ we allow  $N_2,\ldots,N_M$ to be arbitrary,    choose $N_1$ relative to $B_0$ and  then choose  $A_1$ and $B_1$ relative to $N_1$.  In the inductive step, we use the already chosen $N_1,\ldots, N_m$ and  allow  $N_{m+2},\ldots,N_M$ to be arbitrary.  After choosing $N_{m+1}$ relative to the previously chosen constants,     we choose $A_{m+1}$ and $B_{m+1}$   relative to the previously chosen constants and relative to $N_{m+1}$.  

 Throughout the induction argument  $\nu$ is an active vertical element in some $P_s$ and    $\abs{\nu} \le C_M$.  
 
The base case  is $m =1$ and hence $s_M = 1$.    By item \pref{grows with N} of Lemma~\ref{the f case}  applied to $f_1$ and items \pref{item:holds place} and \pref{item:lower bound} of Lemma~\ref{LemmaProperHoldsItsPlace} applied to $h_{12}$, we may choose   $N_1$   so  large  that  if  $\nu$ is an active vertical element in $P_s$ and $\abs{\nu} \le C_M$ then $\nu$ gives rise to $\mu$ satisfying $|\mu| \ge D_{M-1}(B_0)$.    By   Lemma~\ref{LemmaProperHoldsItsPlace}\pref{item:upper bound}  applied to $h_{12}f_1^{N_1}$, there is  an upper bound $A_1$ to  $|\mu|$.  This verifies $(1_1)$.  For $(2_1)$, suppose that $\eta$ is a vertical element of $P_{jM+1}$ for some $j$.  If $\eta$ is inactive and $|\eta| \ge C_0$ then  $\eta$ stays alive  for at least $1$ generation.  By Lemma~\ref{the f case}\pref{grows with N}  applied to $ f_1$,  there exists $B_1 \ge C_0$ so that if $\eta$ is active and $|\eta| \ge B_1$ then    $\eta$ stays alive  for at least $1$ generation.  This proves the first statement in $(2_1)$. The second statement follows from Lemma~\ref{LemmaProperHoldsItsPlace}\pref{item:lower bound} applied first  to $f_1^{N_1}$ and then to $h_{12}$. This completes the base case.

  For the inductive step, we may assume that  constants $N_1,\ldots, N_m$   and   $A_1,...,A_m$ and  $B_1,...,B_m$ have been defined  satisfying $(1_m)$ and $(2_m)$ for arbitrary $N_{m+1}\ldots,N_M$.   For $(1_{m+1})$, we assume that   $1 \le s_M \le m+1$ and that $\nu$ is an active vertical element  in $P_s$ such that $\abs{\nu} \le C_M$.  We define  a vertical element  $\nu'$ of $P_{s+m+1-s_M} \subset G^{m+1}$ with uniformly bounded length as follows:    if $s_M \le m$ then  the $(m+1-s_M)$ successor $\nu'$ to $\nu$ is defined and satisfies $A_m \ge \abs{\nu'} > D_{M-m}(B_{s_M-1})$ by $(1_m)$;  if $s_M = m+1$ then we let $\nu' = \nu$ and note that $\abs{\nu'} \le  C_M$.        By    Lemma~\ref{the f case}\pref{grows with N}  applied to $f_{m+1}$ and items \pref{item:holds place} and \pref{item:lower bound} of Lemma~\ref{LemmaProperHoldsItsPlace} applied to $h_{m+1,m+2}$, we may choose   $N_{m+1}$   so  large  that   if $\nu'$  is active then $\nu'$ gives rise to $\mu$ satisfying $|\mu| \ge D_{M-(m+1)}(B_{s_M-1})$ so the second inequality in ($1_{m+1}$) is satisfied for active $\nu'$. 
  
  If $\nu'$ is inactive then $s_M \le m$ and $\abs{\nu'} > D_{M-m}(B_{s_M-1})$.  Lemma~\ref{defn of D} implies that   $\nu'$ gives rise to $\mu$ satisfying $\abs{\nu} \ge D_{M-m-1}(B_{s_M-1}) = D_{M-(m+1)}(B_{s_M-1})$   so the second inequality in ($1_{m+1}$) is satisfied in this case as well.  An upper bound $A_{m+1}$ for  $|\mu|$ comes from the upper bound for $\abs{\nu'}$ and   Lemma~\ref{LemmaProperHoldsItsPlace}\pref{item:upper bound} .   This completes the proof of ($1_{m+1}$).

  For $(2_{m+1})$, suppose that $\eta$ is a vertical element of $P_{jM+1}$ for some $j$.     By $(2_m)$ and      Lemma~\ref{the f case}\pref{grows with N} applied to $f_m$, we may choose $B_{m+1} \ge B_m$  so  that if $|\eta| \ge B_{m+1}$ then   $\eta$ stays alive for at least $m$ generations and so that the $m^{th}$ successor $\eta'$  of $\eta$  satisfies   $|\eta'| \ge C_0$ and    $|\eta'^{(N_{m+1})}| \ge C_0$.      It follows that $\eta'$ lives at least one generation and hence that $\eta$ lives for at least $m+1$ generations.    The second statement in $(2_{m+1})$ is proved similarly.           This completes the induction step and so the proof of the lemma.
\end{proof}

We now label (c.f.\ Section 5.4  of \BookTwo) certain vertical elements of those $P_s$'s  with $s = j\gen+1$ for some $j $.   In $P_1$ we assign the label $1$ to each vertical element that lives for at least $6\gen$ generations. For $j\ge 1$,  labels in $P_{j\gen+1}$ are defined inductively in two stages.  First, all the labelled elements in $P_{(j-1)\gen+1}$ that live for at least $\gen$ generations (measured from  $P_{(j-1)\gen+1}$) determine elements in $P_{j\gen+1}$ and we label these by increasing their previous label by $1$.  Any other vertical element in $P_{j\gen+1}$ that lives at least $6\gen$ generations (measured from  $P_{j\gen+1}$) is then labelled $1$.  Note that none of elements labelled $1$ can be traced back $\gen$ generations to $P_{(j-1)\gen+1}$ because their ancestors would have lived at least $7\gen$ generations and so be labelled.   

\begin{lemma}  \label{segue} Suppose that $P_1,P_2,\ldots$ is the descendent sequence determined by the constants $N_1,\ldots, N_M$ of Lemma~\ref{choose N}. If a vertical element $\nu$ in some $P_{j\gen+1}$ is not labelled then $\nu$ is inactive and  $\abs{\nu} \le C_{6 \gen}$. 
\end{lemma}

\begin{proof}  Assume that $\nu$ is not labelled.  If $\abs{\nu} \le C_{6 \gen} \le C_M$ then $\nu$ is inactive by  Lemma~\ref{choose N}.      We may therefore assume that $\abs{\nu} > C_{6 \gen}$ and argue to a contradiction;  as we are assuming (see Definition~\ref{generation}) that each vertical element of $P_1$ has length at most $C_1$,  $j \ge 1$.  Lemma~\ref{trace back} implies that   $\nu$ traces back one generation to a vertical element $ \nu_{-1}$ which survives for fewer than $6\gen+1$ generations.      If $\nu_{-1}$ is inactive, then $\nu_{-1} = \nu_{-1}^{(N_{j\gen})}$ is the vertical element of ${h_{1,\gen}}_\#(P_{j\gen+1}) = f_\gen^{N_{j\gen}}(P_{j \gen})$  determined by $\nu$ and $h_{1,\gen}$.  In this case, Corollary~\ref{C}   implies that $| \nu_{-1}| \ge  C_{6 \gen-1}$.   If $\nu_{-1}$ is active, then  Lemma~\ref{choose N} implies that   $| \nu_{-1}| \ge  C_M \ge C_{6 \gen-1}$.     We conclude that $| \nu_{-1}| \ge  C_{6 \gen-1}$ in all cases.     Iterating this argument $\gen-1$ times shows that $\nu$ traces back $\gen$ generations to a vertical element $\nu_{-\gen}$ of $P_{(j-1)\gen+1}$ such that $|\nu_{-\gen}| \ge C_{5\gen}$ and such that $\nu_{-\gen}$ survives fewer than $5\gen$ generations.    Iterating this two step argument five more times shows that $\nu$ traces back $6\gen$ generations which contradicts the assumption that $\nu$ is unlabelled.     
\end{proof}

Let $\E_j$ be the set of labelled vertical elements in $P_{j\gen+1}$.  The circular order on all vertical elements of $P_{j\gen+1}$ restricts to a circular order on the elements of $\E_j$.   Adjacency in the following lemma refers to this induced circular order.

\begin{lemma} \label{event}  Assume that $N_1,\ldots,N_M$ are as in Lemma~\ref{choose N}.    Suppose that $\nu_a,\nu_b $ are adjacent elements of $\E_j$  that survive for at least $\gen$ generations.     Suppose further that
the $\kappa^{th}$ successors to $\nu_a,\nu_b $ are adjacent in $\E_{j+1}$.    Then the element $[[V,W]]$  of $\wt\V^{(2)}$ determined by the subpath $\sigma$ of $P_{j\gen+1}$ beginning with $\nu_a$ and ending with $\nu_b$ (see Notation~\ref{NotationSection8}) is a Nielsen pair for $\k$ associated to~$\F$. 
\end{lemma}  

\begin{proof}  Let $s = j\gen+1$. The vertical elements in $P_s$ between $\nu_a$ and $\nu_b$ are unlabelled and so are inactive  and have length at most $C_{6 \gen+1} \le C_{7\kappa} = C_M$ by   Lemma~\ref{segue}.   Corollary~\ref{CorInactiveIsNielsen}  implies that  $[[V,W]]$ is a Nielsen pair for $\phi_i$.    The subpath  of  $f^{N_s}_\#(P_s)$ beginning with $\nu_a^{(N_s)}$ and ending with  $\nu_b^{(N_s)}$ determines $[[V,W]]$.  Let   $\nu'_{a'}$ and  $\nu'_{b'}$ be the vertical elements of $P_{s+1}$ determined by    $\nu_a^{(N_s)}$ and    $\nu_b^{(N_s)}$  respectively.   By definition, these are the first successors to  $\nu_a$ and   $\nu_b$ respectively.  Lemma~\ref{induces same pairs} implies that the subpath $\sigma'$ of $P_{s+1}$ beginning with $\nu'_{a'}$ and ending with $\nu'_{b'}$ determines $[[V,W]]$. 

Suppose that $\nu'$ is a vertical element of $P_{s+1}$ between $\nu'_{a'}$ and  $\nu'_{b'}$. If $\abs{\nu'} \ge C_1$ then Lemmas~\ref{trace back} and \ref{circular order} imply that $\nu'$ pulls back to a vertical element $\nu$ of $P_s$ between $\nu_a$ and $\nu_b$. Since $\abs{\nu} \le C_{6 \gen}$, it follows that $\abs{\nu'} \le C_{6 \gen +1}$. Thus every vertical element of $P_{s+1}$ between $\nu'_{a'}$ and  $\nu'_{b'}$ has length at most  $C_{6 \gen +1}$. Since the $\gen^{th}$ successors to $\nu_a$ and   $\nu_b$ are adjacent in $\E_{j+1}$,  no vertical element  in $P_{s+1}$ between $\nu'_{a'}$ and   $\nu'_{b'}$ survives for at least $7\gen-1$ generations.  Lemma~\ref{choose N} implies that each of these vertical elements is inactive so Corollary~\ref{CorInactiveIsNielsen} implies that  $[[V,W]]$ is a Nielsen pair for $\phi_{i+1}$. Iterating this argument $\gen-2$ more times shows that  $[[V,W]]$ is a Nielsen pair for $\phi_1,\ldots, \phi_\gen$ and hence a Nielsen pair for~$\k$.
\end{proof}

\paragraph{Proof of Proposition~\ref{PropNielsenPairsExist}.} We continue to adopt Notation~\ref{NotationSectionSeven}. Let $N_1,\ldots,N_M$ be as in Lemma~\ref{choose N} and let $\theta= \phi_M^{N_M}\cdots\phi_2^{N_2}\phi_1^{N_1}$, so $\theta \in \IA_n(\Z/3)$ and $\F$ carries $\Asym(\theta^k)$ for any rotationless power of $\theta$. In particular we have $\theta \in \PGF$. Furthermore $\theta$ is represented by the homotopy equivalence
$$h_{M,M+1} \circ f_M^{N_M} \circ \ldots \circ h_{2,3} \circ f_2^{N_2} \circ h_{1,2} \circ f_1^{N_1} \from G^1 \to G^1
$$

Applying Lemma~\ref{LemmaFindingEG}, using this homotopy equivalence and the circuit $P_1$, it follows that the number of horizontal edges in $P_{j\kappa+1}$ does not 
grow exponentially in $j$, which implies that the number of vertical 
elements in $P_{j\kappa+1}$ does not grow exponentially. Assuming that $\k$ has 
no Nielsen pairs associated to~$\F$ we shall derive a contradiction.

As above, let $\E_j$ be the set of labelled vertical elements in $P_{j\gen+1}$. 
 
Suppose that $\nu_1,\nu_2,\nu_3 $ are consecutive elements of $ \E_j$ with labels $\ell_1,\ell_2,\ell_3$. We claim that if  $\ell_2=2$ then $\ell_1$ and $\ell_3$ are either $1$ or $2$ and if $\ell_2 \ge 3$ then $\ell_1=\ell_3=1$. The $\nu_1$ and $\nu_3$ cases are symmetric so it suffices to verify the claim for $\ell_1$.   
  
  The claim is obvious if $j=0$ since $1$ is the only label that occurs.  Assume that the claim holds for $\E_0,\ldots,\E_{j-1}$.   Assuming without loss that $\ell_2 \ne 1$, $\nu_2$ traces back $\gen$ generations to $\nu'_2 \in \E_{j-1}$.   Let $\nu_1'$ and $\nu_3'$ be the elements of  $\E_{j-1}$ so that    $\nu'_1, \nu_2', \nu'_3$ are adjacent in  $\E_{j-1}$.     There are two cases to consider.   If $\nu'_1$ survives $\gen$ generations, then  its $\gen^{th}$ successor is not adjacent to $\nu_2$ by Lemma~\ref{event} and our assumption that there do not exist Nielsen pairs for $\k$ associated to~$\F$.   By definition of the labeling process, the labels of all elements of $\E_j$ between the $\gen^{th}$ successor to $\nu'_1$ and $\nu_2$ are $1$'s; thus $\ell_1 = 1$ and we are done.   
  
  The second case is that   $\nu'_1$ does not survive $\gen$ generations and so its label is at least~$6$.  The inductive hypothesis implies that both $\nu'_2$ and the other element $\nu'_0$  of $\E_{j-1}$ that is adjacent to    $\nu'_1$ are labelled $1$.  It follows that  $l_2 =2$ and that the   $\gen^{th}$ successor  to $\nu'_0$ is labelled $2$; if the latter  equals $\nu_1$ we are done.  Otherwise, the elements between the   $\gen^{th}$ successor   to $\nu'_0$ and $\nu_2$  are all labelled $1$.  This completes the proof of the claim. 

Let $A_j(l)$ be the number of elements of $\E_j$ that are   labelled $l$.     The above claim implies 
$$A_j(1) \ge A_j(3)+ A_j(4)
$$   
It is an immediate consequence of the definitions that 
$$A_j(1) = A_{j+1}(2) = A_{j+2}(3) = A_{j+3}(4)
$$  
Thus 
\begin{align*} 
A_{j+5}(1) &\ge  A_{j+5}(3)+ A_{j+5}(4) \\
&= A_{j+3}(1) + A_{j+2}(1)\\
&\ge A_{j+3}(3) + A_{j+3}(4) + A_{j+2}(3) + A_{j+2}(4) \\
&= A_{j+1}(1) + A_j(1) + A_j(1) +A_{j-1}(1) \\
&\ge 2 A_j(1)
\end{align*}
  
 This proves that $A_j(1)$ grows exponentially in $j$ and hence that the number of vertical edges in $P_{j\gen+1}$ grows exponentially in $j$, and as explained earlier we obtain a contradiction which shows that $\k$ does indeed have a Nielsen pair associated to~$\F$.
 
\bigskip

To complete the proof, using the existence of one Nielsen pair for $\k$ associated to~$\F$ we shall prove:
\begin{description}
\item[$(*)$] There exists a Nielsen pair $(V,W)$ for $\k$ associated to~$\F$, and an $F_n$-tree $T$ with trivial edge stabilizers and $\F(T)=\F$, such that if $\ti\gamma \subset T$ is the path connecting the vertices with stabilizers $V,W$, respectively, then the interior of $\ti\gamma$ does not contain any vertex with nontrivial stabilizer and is disjoint from all of its translates. 
\end{description}
Once $(*)$ is proved, if $\ti\gamma$ is an edge then we are done. Otherwise, let $\ti \gamma_0$ be an initial segment of $\ti \gamma$ that contains all but the last edge of $\ti \gamma$ and collapse each component of the union of all translates of $\ti \gamma_0$ to a point. The resulting tree (still called $T$) has trivial edge stabilizers and $\F(T) =\F$. It is now true that the path in $T$ connecting the vertex with stabilizer $V$ to the vertex with stabilizer $W$ is a single edge. 

After a bit of setup using what we have already proved, the proof of $(*)$ very closely follows Sections~5.5 and~5.6 on pages 54--56 of \BookTwo, in particular we follow closely the proof of Theorem~5.20 in Section~5.6. Fix a Nielsen pair $V_0,W_0$ for $\k$. Choose $T_0$ having trivial edge stabilizers and satisfying $\F(T_0)=\F$. Let $T_0,T_1,T_2,\ldots$ be the bouncing sequence of~$T_0$. Applying Propositions~\ref{prop:grower} and~\ref{prop:nongrower} inductively it follows that $\F \sqsubset \F(T_i)$ for each $i$. Since $\Asym(\phi_i)$ is carried by $\F$ for all $i$, it follows that $\A_+(T_{i-1};\phi_i) = \emptyset$, so $T_{i-1}$ is a nongrower and Proposition~\ref{prop:nongrower} applies to $T_{i-1}$. Proposition~\ref{prop:nongrower}~\pref{ItemLengthEventuallyConstant} says for each conjugacy class $[c]$ in $F_n$ that $L_{T_{i-1}}(\phi^k[c])$ is eventually equal to $L_{T_{i}}[c]$; assuming by induction that $\F(T_{i-1}) = \F$, if $[c]$ is not elliptic in $T_{i-1}$ then $[c]$ is not carried by $\F$, so no iterate $\phi^k[c]$ is carried by~$\F$, so $[c]$ is not elliptic in $T_i$, and it follows that $\F(T_i)=\F$. Proposition~\ref{prop:nongrower}~\pref{ItemArcsStabilizersDontGrow} tells us that every edge stabilizer of $T_i$ is an edge stabilizer of $T_{i-1}$, and so by induction starting with the fact that $T_0$ has trivial edge stabilizers it follows that each $T_i$ has trivial edge stabilizers.

Lemmas~5.18 and~5.19 of \BookTwo\ are a study of the distances between vertices with nontrivial stabilizers in each tree $T_i$. The proofs of those lemmas apply verbatim in our situation, with $\mathcal{O}_i$ replaced by $\phi_i$, and citing our Proposition~\ref{prop:nongrower} where appropriate in order to inductively verify properties of the tree $T_{i+1} = T_i \phi_{i+1}^\infinity$. Applying Lemma~5.18, for any Nielsen pair $(V,W)$ the distance in $T_i$ between the vertices fixed by $V,W$ is a constant independent of~$i$. Letting $D_i$ be the set of all natural numbers of the form $d(v,w)$ where $v \ne w \in T_i$ have nontrivial stabilizer, applying Lemma~5.19~(2) it follows that $D_i \subset D_{i+1}$ for all $i$, and so $\min D_i \le \min D_{i+1}$. But $\min D_i$ is bounded above for all $i$ by the value of $D_i$ determined by the vertices fixed by the Nielsen pair $(V_0,W_0)$, and so $\min D_i$ is constant for sufficiently large $i$, say $i \ge c$. Let $V,W$ be two nontrivial vertex stabilizers in $T_c$ that realize $\min D_c$. By Lemma~5.19~(4) and~(6), $(V,W)$ is a Nielsen pair for every $\phi_i$, and hence $(V,W)$ is a Nielsen pair for $\k$ associated to~$\F$. Let $\ti\gamma$ be the path in $T_c$ between the vertices stabilized by $V,W$, so $\ti\gamma = \min D_c$. Each vertex in the interior of $\ti\gamma$ has trivial stabilizer and the projection of $\ti\gamma$ to the quotient graph of groups $T_c / F_n$ is an embedding except perhaps at its endpoints, and so $(*)$ is proved.  
\qed

%

\bibliographystyle{amsalpha} 
\bibliography{mosher} 

\end{document}